\documentclass[12pt,letterpaper,english]{amsart}
\usepackage{lmodern}
\usepackage{helvet}

\usepackage[T1]{fontenc}
\usepackage[latin9]{inputenc}
\usepackage{mathrsfs}
\usepackage{amsthm}
\usepackage{amstext}
\usepackage{amssymb}
\usepackage{tikz}
\usepackage{esint}
\usepackage{graphicx}
\usepackage{cancel}
\usepackage{comment}
\usepackage{enumitem}
\usepackage{amscd,color}
\DeclareFontFamily{OT1}{pzc}{}
\DeclareFontShape{OT1}{pzc}{m}{it}{<-> s * [1.10] pzcmi7t}{}
\DeclareMathAlphabet{\mathpzc}{OT1}{pzc}{m}{it}
\usepackage[bbgreekl]{mathbbol}
\DeclareSymbolFontAlphabet{\mathbb}{AMSb}
\DeclareSymbolFontAlphabet{\mathbbl}{bbold}

\makeatletter

\pdfpageheight\paperheight
\pdfpagewidth\paperwidth

\numberwithin{equation}{section}
\numberwithin{figure}{section}
\usepackage{enumitem}		
\theoremstyle{plain}
\newtheorem{thm}{\protect\theoremname}[section]
\theoremstyle{remark}
\newtheorem{rem}[thm]{\protect\remarkname}
\theoremstyle{definition}
\newtheorem{defn}[thm]{\protect\definitionname}
\theoremstyle{definition}
\newtheorem{example}[thm]{\protect\examplename}
\theoremstyle{plain}

\theoremstyle{plain}
\newtheorem{prop}[thm]{Proposition}
\theoremstyle{plain}
\newtheorem{lem}[thm]{Lemma}
\theoremstyle{plain}
\newtheorem{conj}[thm]{\protect\conjecturename}

\usepackage{amsfonts}
\usepackage{amssymb}
\usepackage[all]{xy}
\usepackage{array}
\usepackage{lmodern}
\usepackage[T1]{fontenc}
\usepackage{bm}

\def\uo{\underline{0}}
\def\ul{\underline{\ell}}
\def\un{\underline{n}}

\def\ut{\underline{t}}
\def\uv{\underline{\nu}}
\def\um{\underline{m}}
\def\uy{\underline{y}}
\def\ux{\underline{x}}
\def\QQ{\mathbb{Q}}
\def\RR{\mathbb{R}}
\def\CC{\mathbb{C}}

\def\FF{\mathbb{F}}
\def\ZZ{\mathbb{Z}}
\def\PP{\mathbb{P}}
\def\TT{\mathbb{T}}
\def\DD{\mathbb{D}}

\def\uz{\underline{z}}
\def\ut{\underline{t}}

\def\Z{\mathcal{Z}}
\def\V{\mathcal{V}}

\def\T{\mathcal{T}}

\def\A{\mathcal{A}}
\def\B{\mathcal{B}}
\def\M{\mathcal{M}}
\def\H{\mathcal{H}}
\def\HH{\mathbb{H}}
\def\NN{\mathbb{N}}

\def\vf{\varphi}

\def\Res{\text{Res}}

\def\ay{\mathbf{i}}

\def\ve{\varepsilon}

\def\co{\mathcal{O}}

\def\F{\mathcal{F}}

\def\E{\mathcal{E}}

\def\C{\mathcal{C}}

\def\NP{\mathbbl{\Delta}}
\def\x{\mathsf{x}}
\def\y{\mathsf{y}}
\def\hx{\hat{\x}}
\def\hy{\hat{\y}}
\def\ua{\underline{a}}
\def\LL{\mathbb{L}}
\def\ui{\underline{i}}
\def\ud{\underline{d}}
\def\J{\mathcal{J}}
\def\ee{\mathbf{e}}
\def\BB{\mathbb{B}}
\def\uBB{\underline{\BB}}
\def\ub{\underline{b}}
\def\tl{\tau}
\def\qt{\mathfrak{t}}
\def\uqt{\underline{\qt}}
\def\qth{\uqt^{\hbar}}
\def\qtp{\uqt^{2\pi}}
\def\NP{\mathbbl{\Delta}}
\def\R{\mathcal{R}}
\def\RT{\tilde{\R}}
\def\uAA{\underline{\mathbb{A}}}
\def\uj{\underline{j}}
\def\uN{\underline{N}}
\def\uq{\underline{q}}
\def\fz{\mathfrak{z}}
\def\fzmn{\fz_{m,n}}
\def\fw{\mathfrak{w}}
\def\pz{\mathscr{P}_{\tilde{z}_0}^{\tilde{z}}}
\def\tz{\tilde{z}}
\def\po{\mathscr{P}_0^{\pm}}
\def\lo{\mathscr{L}_0^{\pm}}
\def\P{\mathcal{P}}
\def\mi{\text{-}}

\theoremstyle{definition}

\theoremstyle{definition}

\theoremstyle{definition}
\newtheorem*{thx}{Acknowledgments}
\theoremstyle{plain}
\newtheorem*{thma}{Theorem A}
\newtheorem*{thmb}{Theorem B}

\makeatother

\usepackage{babel}
  \providecommand{\corollaryname}{Corollary}
  \providecommand{\definitionname}{Definition}
  \providecommand{\remarkname}{Remark}
\providecommand{\theoremname}{Theorem}
 \providecommand{\examplename}{Example}

\providecommand{\conjecturename}{Conjecture}

\begin{document}

\title{$K_2$ and quantum curves}

\author[C. F. Doran]{Charles F. Doran}
\address{University of Alberta, Edmonton, Canada\\
Center for Mathematical Sciences and Applications, Cambridge, MA}
\email{charles.doran@ualberta.ca}

\author[M. Kerr]{Matt Kerr}
\address{Washington University in St. Louis, Department of Mathematics and Statistics, St. Louis, MO}
\email{matkerr@wustl.edu}

\author[S. Sinha Babu]{Soumya Sinha Babu}
\address{Washington University in St. Louis, Department of Mathematics and Statistics, St. Louis, MO}
\email{soumya@wustl.edu}

\begin{abstract}
A 2015 conjecture of Codesido-Grassi-Mari\~no in topological string theory relates the enumerative invariants of toric CY 3-folds to the spectra of operators attached to their mirror curves.  We deduce two consequences of this conjecture for the integral regulators of $K_2$-classes on these curves, and then prove both of them; the results thus give evidence for the CGM conjecture.  (While the conjecture and the deduction process both entail forms of local mirror symmetry, the consequences/theorems do not:  they only involve the curves themselves.)  Our first theorem relates zeroes of the higher normal function to the spectra of the operators for curves of genus one, and suggests a new link between analysis and arithmetic geometry.  The second theorem provides dilogarithm formulas for limits of regulator periods at the maximal conifold point in moduli of the curves.
\end{abstract}
\maketitle

\section{Introduction}\label{S1}

The simplest Calabi-Yau threefolds are the noncompact toric CYs $X$ determined by a convex lattice polygon $\Delta$ (or more precisely by the fan on a triangulation of $\{1\}\times \Delta$ in $\RR^3$).  Each such CY has a family of \emph{mirror curves} $\C\subset \CC^*\times\CC^*$, of genus $g$ equal to the number of interior integer points of $\Delta$, given by the Laurent polynomials $F(x_1,x_2)$ with Newton polygon $\Delta$.  Recently a fundamental and novel relationship between \textbf{(i)} the enumerative geometry of $X$ and \textbf{(ii)} the spectral theory of certain operators $\hat{F}$ on $L^2(\RR)$ attached to $\C$, has been proposed by M. Mari\~no and his school, in the context of \emph{non-perturbative} topological string theory \cite{GHM,Ma,CGM}.  The goal of this paper is to lay out some mathematical consequences of this meta-conjecture, and provide evidence for it by proving them in two important cases.

A Laurent polynomial $F=\sum_{\um\in \Delta\cap \ZZ^2}a_{\um}\ux^{\um}$ is promoted to an operator $\hat{F}$ (or ``quantum curve'') by a process called \emph{Weyl quantization}, which depends on a real constant $\hbar$.  Writing $r$ for the coordinate on $\RR$, let $\hx$ denote multiplication by $r$, and $\hy:=\ay\hbar\partial_r$, so that $[\hx,\hy]=\ay\hbar$.  Taking $\hat{F}:=\sum a_{\um}e^{m_1\hx+m_2\hy}$, \cite{CGM} define a \emph{generalized spectral determinant} $\Xi_{\C}(\ua;\hbar)$ whose zero-locus describes those curve moduli $\ua$ for which $\ker(\hat{F})\neq \{0\}$.  They conjecture that under a ``quantum mirror map'' $\ua\mapsto \mathfrak{t}^{\hbar}(\ua)$, $\Xi_\C$ is proportional to a \emph{quantum theta function} $\Theta_X(\ut;\hbar)$ derived from the all-genus enumerative invariants of $X$; see Conjecture \ref{c1}.  In particular, the zeroes of $\Theta_X$ should recover the spectrum of any fixed quantum curve $\hat{F}$.

In the formulation of \cite{BKV}, \emph{local mirror symmetry} relates the ``maximally supersymmetric'' case ($\hbar=2\pi$) of \textbf{(i)} to \textbf{(iii)} the Hodge-theoretic invariants (or ``regulators'') of algebraic $K_2$-classes on $\C$.  This allows us to reformulate this case of the conjecture of Codesido-Grassi-Mari\~no \cite{CGM} in \S\ref{S2c} as a putative relationship between quantum curves and regulators (i.e. between \textbf{(ii)} and \textbf{(iii)}).  We do this under the assumption that $F$ ranges only over the \emph{integrally tempered} Laurent polynomials, so that the symbol $\{-x_1,-x_2\}\in K_2(\CC(\C))$ extends to motivic cohomology classes on the compactifications $\bar{\C}_{\ua}\subset \PP_{\Delta}$.  This smaller moduli space $\bm{\M}$ has dimension $g$, and the resulting regulator classes $\tfrac{1}{4\pi^2}R(\ua)\in H^1(\bar{\C}_{\ua},\CC/\ZZ)$ may be projected modulo $H^{1,0}(\bar{\C}_{\ua})$ to yield a section $\bm{\nu}$ of the Jacobian bundle $\bm{\J}\to \bm{\M}$ of the family $\bm{\C}\to\bm{\M}$, called the \emph{higher normal function}.  We deduce from the conjecture of \cite{CGM} that the locus in $\M$ where $\bm{\nu}$ meets a specific torsion shift of the theta divisor in $\J$ should match the zero-locus of $\Xi_{\C}$ after tweaking the signs of the moduli; this is made precise in Conjecture \ref{c2}.  

We may further refine this prediction in the genus-1 case, where $\Delta$ is now reflexive and the Laurent polynomial $F(\ux)=\vf(\ux)+a$ now has only one parameter $a$.  In \S\ref{S3a}, we use integral mirror symmetry to compute the torsion shifts, and show that (after a miraculous cancellation) they simply translate the theta divisor to the origin!  The prediction is now that the spectrum of the quantum curve is given by\footnote{Note the implicit sign flip on $a$:  we are saying that $\ker(\hat{\vf}-a)\neq \{0\}$ when the regulator associated to $\{-x_1,-x_2\}$ on $\vf(\ux)+a=0$ dies in the Jacobian.  The notation for the normal function changes from $\bm{\nu}$ to $\nu$ as it no longer has multiple components.}
\begin{equation}\label{p1}
\sigma(\hat{\vf})=\{a\in \bm{\M}\mid {\nu}(a)\equiv 0\;\in J(\bar{\C}_a) \}.
\end{equation}
Keeping in mind that $g=1$ ($\Delta$ reflexive), $\vf$ is tempered, and $\hbar=2\pi$, our first main unconditional result is then the following
\begin{thma}[Theorems \ref{t1} and \ref{t1a}]\label{tA}
Assume $\Delta\subset \RR\times [-1,1]$.  Then the \textup{``}$\supseteq$\textup{''} direction of \eqref{p1} holds, and the \textup{``}$\subseteq$\textup{''} direction holds for ``almost all'' eigenvalues.
\end{thma}
We prove the ``$\supseteq$'' statement in \S\ref{S3b} by explicitly constructing square-integrable eigenfunctions of $\hat{\vf}$ with eigenvalue $a$, using vanishing of ${\nu}(a)$ to show well-definedness.  The result (in \S\ref{S3c}) on the ``$\subseteq$'' inclusion is obtained by using the coherent state representation of $\hat{\vf}$ to bound the accumulation of eigenvalues in a manner that matches growth ($\sim\text{const.}\times \log^2(a)$) of ${\nu}$ as $a\to \infty$.  One perspective on Theorem A is that we may view $\nu(a)$ as a normalized solution to an inhomogeneous Picard-Fuchs equation, and in effect \eqref{p1} states that the eigenvalues of $\hat{\vf}$ are simply the points where $\nu(a)\in \ZZ$ (see Remark \ref{r3a2}(i)).  The latter condition is a statement about a period of a mixed motive, and combining this with a variant of Grothendieck's period conjecture allows one to show conditionally that the eigenvalues of $\hat{\vf}$ are transcendental numbers (Prop. \ref{p3c1}).

The conjecture of \cite{CGM} yields a different prediction in the 't Hooft limit $\hbar\to \infty$, which is not empty for $g=1$ but much more interesting for $g>1$.  Results of Kashaev, Mari\~no and Zakany \cite{KM,MZ} on the limits of spectral traces of three-term operators can be viewed as providing a general formula for the limiting value of a particular regulator period $R_{\gamma}({\ua})=\int_{\gamma}R\{-x_1,-x_2\}|_{\C_{{\ua}}}$ at the maximal conifold point $\hat{\ua}$, in terms of special values of the Bloch-Wigner (``real single-valued dilogarithm'') function.  Here ``maximal conifold'' means a particular point in moduli at which $\C$ acquires $g$ nodes while remaining irreducible; that is, the normalization $\widetilde{\C_{\hat{\ua}}}$ is a $\PP^1$.  By applying a method from \cite[\S6]{DK} for computing regulator periods on singular curves of geometric genus zero, we are able to verify this in two infinite families of cases, corresponding to 
\begin{align*}
F^{\ua}_{g,g}(\ux)&=x_1+x_2+x_1^{-g}x_2^{-g}+\textstyle\sum_{j=1}^g a_jx_1^{1-j}x_2^{1-j} \;\;\text{and}\\ F^{\ua}_{2g-1,1}(\ux)&=x_1+x_2+x_1^{-2g+1}x_2^{-1}+\textstyle\sum_{j=1}^g a_jx_1^{1-j}.
\end{align*}
The $g=1$ case was already verified in \cite[\S6.3]{DK}, while the $g=2$ identities were partially verified in \cite[\S6]{7K}.

To give a more explicit statement of this result, write $\tilde{F}^{\ua}:=F^{\ua}-a_1$ in either case, and $[\cdot]_{\uo}$ for the operator taking the constant term (in $x_1,x_2$) in a Laurent polynomial.  Then we have:
\begin{thmb}[Theorem \ref{thm_01} and \eqref{1.3}]\label{tB}
The regulator periods at the maximal conifold point satisfy
\begin{equation*}
\log(2g+1)-\textstyle\sum_{k>0}\tfrac{(-1)^{k(g+1)}}{k(2g+1)^k}[(\tilde{F}^{\hat{\ua}}_{g,g})^k]_{\uo}=\tfrac{1}{2\pi\ay}R^{g,g}_{\gamma}(\hat{\ua})=\tfrac{(2g+1)}{\pi}D_2(1+e^{\frac{2\pi\ay g}{2g+1}})
\end{equation*}
and\small
\begin{equation*}
\log(2g+1)-\textstyle\sum_{k>0}\tfrac{1}{k(2g+1)^k}[(\tilde{F}^{\hat{\ua}}_{2g-1,1})^k]_{\uo}=\tfrac{1}{2\pi\ay}R^{2g-1,1}_{\gamma}(\hat{\ua})=\tfrac{(2g+1)}{\pi}D_2(1+e^{\frac{2\pi\ay }{2g+1}}).
\end{equation*}\normalsize
\end{thmb}
In fact, the two families are isomorphic under the moduli-map sending $a_j\mapsto a_{g-j+1}$, and the cycles are just two amongst $g$ (named $\gamma_1,\ldots,\gamma_g$) for which we can compute the regulator period at $\hat{\ua}$, obtaining $g$ different identities.  Part of the proof involves using a method from \cite{Ke2} to determine (from the series expansions of their periods) how many times the ``limits'' of the $\{\gamma_j\}$ at $\hat{\ua}$ pass through each of the $g$ nodes, cf. Prop. \ref{thm_02}; this method may be of independent interest in the study of monodromy.  Incidentally, the identities we prove should have implications for the asymptotic behavior of genus-zero Gromov-Witten numbers of the corresponding CY $X$, but we do not pursue this direction here.

In an appendix we compute some regulator periods used in the paper and relate the torsion constants so crucial in \S\ref{S3a} to integral periods of a limiting mixed Hodge structure.
Finally, as a quick word on notation: we use $\partial_x=\tfrac{\partial}{\partial x}$ and $\delta_x=x\partial_x$ throughout; and we avoid the use of Einstein summation.

\begin{thx}
The authors thank M. Mari\~no for bringing the conjecture to our attention.  This work was partially supported by Simons Collaboration Grant 634268 and NSF Grant DMS-2101482 (MK), and an NSERC Discovery Grant (CD).
\end{thx}

\section{A conjecture in topological string theory and its consequences}\label{S2}

\subsection{Quantum curves}\label{S2a}

Let $\Delta\subset \RR^2$ be a polygon with vertices in $\ZZ^2$ whose interior contains the origin $\uo$.  Write 
\begin{equation}\label{e2.1.1}
\textstyle F(x_1,x_2)=\sum_{\um\in \Delta\cap\ZZ^2} a_{\um}\ux^{\um}
\end{equation} 
for a general Laurent polynomial with Newton polygon $\Delta$.  The affine curve $\C:=\{\ux\in (\CC^*)^2\mid F(\ux)=0\}$ is then smooth of genus $g:=|\text{int}(\Delta)\cap \ZZ^2|$.  It admits a smooth compactification $\bar{\C}$ in $\PP_{\Delta}$, which denotes a minimal toric desingularization of the toric surface constructed from the normal fan of $\Delta$.  For instance, if $\Delta$ is reflexive with polar polygon $\Delta^{\circ}$, then $g=1$ and $\PP_{\Delta}$ is constructed from the fan with rays passing through each of the nonzero points of $\Delta^{\circ}\cap \ZZ^2$.

Taking a maximal integral triangulation $\mathrm{tr}(\Delta)$, consider the fan $\Sigma$ on $\{1\}\times \mathrm{tr}(\Delta)\subset \RR^3$.  The resulting toric variety 
\begin{equation}\label{e2.1.2}
X:=\PP_{\Sigma}
\end{equation}
is called a \emph{local CY 3-fold} since $K_X\cong \co_X$.\footnote{To see this, note that $-c_1(K_X)=c_1(X)$ is the sum of the irreducible divisors corresponding to the elements of $\Delta\cap \ZZ^2$, which is the divisor of the first toric coordinate $w_0$ on $X$ hence rationally equivalent to zero.}  This will be our ``A-model'', on which we do enumerative geometry and run the K\"ahler moduli.  Such noncompact CY 3-folds often arise from the crepant resolution of a finite quotient of $\CC^3$.  For instance, if $1\in \ZZ_{2k+1}$ acts on $\CC^3$ by $\mathrm{diag}\{\zeta_{2k+1},\zeta_{2k+1}^k,\zeta_{2k+1}^k\}$, the resolution $X$ is obtained by taking $\Delta$ to be the convex hull of $(1,0)$, $(0,1)$, and $(-k,-k)$ (with $g=k$).  Another set of examples (with $g=1$) arises when $\Delta$ is reflexive:  in this case, $X$ is just the total space of $K_{\PP_{\Delta^{\circ}}}$.  There is some overlap with the quotient construction:  for instance, $K_{\PP^2}$ [resp. $K_{\FF_2}$, $K_{\mathrm{dP}_6'}$\footnote{We shall use the notation $\mathrm{dP}_6'$ to refer to the generalized del Pezzo of degree $6$ defined by the self-dual polygon with vertices $(1,0)$, $(0,1)$, and $(-3,-2)$.  (This is called the ``$E_8$ del Pezzo'' in \cite{GKMR}.}] arises from a quotient of $\CC^3$ by $\ZZ_3$ [resp. $\ZZ_4$, $\ZZ_6$].

Local mirror symmetry connects the genus-zero enumerative invariants of $X$ to periods of the ``B-model''
\begin{equation}\label{e2.1.3}
Y:=\{(\ux,u,v)\in (\CC^*)^2\times\CC^2\mid  F(x_1,x_2)+uv=0\},
\end{equation}
an open CY 3-fold with $K_{Y}$ trivialized by the form
\begin{equation}\label{e2.1.4}
\eta:=\frac{1}{(2\pi\ay)^2}\Res_{Y}\left( \frac{dx_1/x_1 \wedge dx_2/x_2 \wedge du\wedge dv}{F(\ux)+uv}\right)\in \Omega^3(Y).
\end{equation}
We shall will say more about this in due course.  It has been proposed by Mari\~no and collaborators \cite{GHM,Ma,CGM} that one can capture the higher-genus enumerative invariants of $X$ as well by \emph{quantizing} the curve $\C$ --- that is, turning the Laurent polynomial $F$ into an operator and considering its spectral theory.  The idea is to write $x_1=e^{\x}$, $x_2=e^{\y}$, and promote $\x,\y$ to noncommuting operators $\hx,\hy$ on $L^2(\RR)$ with $[\hx,\hy]=\ay\hbar$ ($\hbar\in \RR$).  More explictly, writing $r$ for the coordinate on $\RR$, we take $\hx=\mu_r$ (multiplication by $r$) and $\hy=-\ay\hbar\partial_r$; and then we set $\hat{x}_1=e^{\hx}$, $\hat{x}_2=e^{\hy}$.  Notice that if $f\in L^2(\RR)$ is the restriction of an entire function, then $\hat{x}_2$ is a shift operator, viz. $(e^{-\ay\hbar\partial_r}f)(r)=f(r-\ay\hbar)$.

The promotion of $F$ to $\hat{F}$ is highly nonunique:  for instance, $e^{\hx}e^{\hy}$ and $e^{\hx+\hy}$ [resp. $e^{\hy}e^{\hx}$] differ by a multiplicative factor of $e^{i\hbar/2}$ [resp. $e^{i\hbar}$] by the Campbell-Baker-Hausdorff formula.  The standard way to fix this (before \cite{CGM}) was to employ a perturbative approach called WKB approximation, which works modulo successive powers of $\hbar$.  In this context a connection between quantization and $K_2(\CC(\C))$ was pointed out in \cite{GS}, which we briefly review in the next paragraph, if only to highlight that it is \emph{completely different} from the link (in the non-perturbative setting) we conjecture in \S\ref{S2c} and establish in \S\ref{S3}. 

So suppose that we want a function $\psi$ on $\C$ (rather than $\RR$) and a choice of $\hat{F}$ given by $\hat{F}_0:=F(\hat{x}_1,\hat{x}_2):=F(\mu_{x_1},e^{-\ay\hbar\delta_{x_1}})$ mod $O(\hbar)$, for which $\hat{F}\psi=0$.  (In this case, we will say $\C$ is \emph{quantizable}.)  Begin with formal asymptotic expansions $\hat{F}=\sum_{i\geq 0}\hbar^i \hat{F}_i$, and $\psi=e^{\frac{\ay}{\hbar}\sum_{j\geq 0}\hbar^j S_j}$.  Choosing a base point $p_0\in \C_F$ with $x_1(p_0)=1$, we take $S_0(p)=\int_{p_0}^p \log(x_2)\tfrac{dx_1}{x_1}$ (integral on $\C$), which locally satisfies $\delta_{x_1}S_0=\log(x_2)$ hence $(\hat{F}\psi)(p)=[F(x_1(p),x_2(p))+O(\hbar)]\psi(p)=O(\hbar)\psi(p)$.  Of course, $e^{\frac{\ay}{\hbar}S_0}$ only gives a well-defined function on $\C$ if the integral is path-independent mod $2\pi\hbar\ZZ$.   When this happens, one then solves for the higher-order corrections $S_i$, by postulating their form in terms of ``topological recursion'', and finally solves for the $\hat{F}_i$.  We remark that for $\hbar=2\pi$, the well-definedness condition on $S_0$ is precisely the statement that the regulator class $R\{x_1,x_2\}\in H^1(\C,\CC/\ZZ(2))$ of the coordinate symbol $\{x_1,x_2\}\in K_2(\CC(\C))$ is \emph{trivial}.  More generally, if the regulator class is \emph{torsion} (which is the quantizability criterion proposed by \cite{GS}), then the well-definedness condition is satisfied for $\hbar=\tfrac{2\pi}{M}$ for some $M\in \ZZ$.  This is a very different condition on the regulator class than the one appearing in RHS\eqref{e2.3.13} below, even in the $g=1$ case (see the discussion leading up to Lemma \ref{l3c1}).

For the rest of this paper we consider only the non-perturbative (exact) approach pioneered in \cite{GHM}.  Namely, we fix the single choice 
\begin{equation}\label{e2.1.5}
\textstyle\hat{F}=\sum_{\um\in \Delta\cap \ZZ^2}a_{\um}e^{m_1\hx+m_2\hy}
\end{equation}
and try to describe its spectrum as an operator on $L^2(\RR)$.  A little more precisely, if $\mathrm{int}(\Delta)\cap\ZZ^2=\{\um^{(j)}\}_{j=1,\ldots,g}$, then writing $a_j:=a_{\um^{(j)}}$, $P_j=\ux^{\um^{(j)}}$, $F_j^{(0)}=P_j^{-1}F|_{a_1=\cdots=a_g=0}$ and $F_j=P_j^{-1}F|_{a_j=0}$, we are interested in determining the eigenvalues $\{e^{E^{(j)}_n(a_1,\ldots,\widehat{a_j},\ldots,a_g)}\}_{n\in \NN}$ of $\hat{F}_j$ for $j=1,\ldots,g$.\footnote{For the time being, one should think of the non-interior parameters $a_{\um}$ as being fixed.  For the assertion that the spectrum is positive and discrete, further restrictions (such as those we impose for temperedness later) should be made.}  We should note here that as long as the $\{a_{\um}\}$ are all real, the $\hat{F}_j,\hat{F}^{(0)}_j$ are obviously Hermitian; even better, their inverses $\rho_j,\rho_j^{(0)}$ are expected to be bounded self-adjoint and of trace class, with a discrete positive spectrum.  These properties, which justify indexing the eigenvalues by $\NN$ and make the Fredholm determinants 
\begin{equation}\label{e2.1.6}
\textstyle \det(1+a_j\rho_j)=\prod_{n\geq 0}( 1+a_j e^{-E_n^{(j)}(a_1,\ldots,\widehat{a_j},\ldots,a_g)})
\end{equation}
well-defined, are proved in \cite{KM} and \cite{LST} for all the specific operators we will discuss below.

\begin{defn}[\cite{CGM}]\label{d2.1.1}
The \emph{generalized spectral determinant} is 
\begin{equation}\label{e2.1.7}
\textstyle\Xi_{\C}(\ua;\hbar):=\det(1+\sum_{j=1}^g a_j \hat{P}_j^{-\frac{1}{2}}\rho_j^{(0)}\hat{P}_j^{\frac{1}{2}}).
\end{equation}
\end{defn}

This function contains all the information we are after.  For any fixed $\{a_k\}_{k\neq j}$, we may recover \eqref{e2.1.6} as $\Xi_{\C}(\ua;\hbar)/(\Xi_{\C}(\ua;\hbar)|_{a_j=0})$, since their zeroes (in $a_j$) are the same and both sides are $1$ at $a_j=0$ \cite[(2.74)]{CGM}.  So the spectra of $\hat{F}_1,\ldots,\hat{F}_g$ are simply slices of the zero-locus of \eqref{e2.1.7}, a union of hypersurfaces in $\RR^g$ indexed by $\NN$.  Note that in the genus one case, \eqref{e2.1.7} is just $\det(1+a_1\rho_1)$.

\subsection{Local mirror symmetry and the CGM conjecture}\label{S2b}

Let $r:=|\partial\Delta\cap\ZZ^2|$, so that $|\Delta\cap\ZZ^2|=g+r$; and denote by $\LL\subset\ZZ^{g+r}$ the rank-($g+r-3$) lattice of \emph{relations vectors} $\{\ell_{\um}\}_{\um\in\Delta\cap\ZZ^2}$ with $\sum_{\um}\ell_{\um}(1,\um)=\uo$.  Each $\um\in \Delta\cap\ZZ$ corresponds to a toric divisor $D_{\um}\subset X$, amongst which we have the $g$ compact $D_j:=D_{\um^{(j)}}$.  If $C\subset X$ is any \emph{compact} toric curve (corresponding to any edge of $\mathrm{tr}(\Delta)$), its intersection numbers with the divisors of the toric coordinates $w_0,w_1,w_2$ are zero, leading to a relations vector $\ell_{\um}=(C\cdot D_{\um})_X$.  Such relations integrally span $\LL$, although the (Mori) cone generated by \emph{effective} curves may not be smooth or even simplicial.  We will ignore such ``finite data'' issues here, as we will eventually pass to a slice of the complex-structure moduli space where this is not an issue.  

So write $\{C_i\}_{i=1,\ldots,g+r-3}$ for independent generators of this cone (i.e. $H_2(X,\ZZ)_{\text{eff}}$), with corresponding relations $\ul^{(i)}$, and define complex structure parameters
\begin{equation}\label{e2.2.1}
\textstyle z_i=z_i(\ua):=\prod_{\um\in \Delta\cap \ZZ^2}a_{\um}^{\ell^{(i)}_{\um}}
\end{equation}
for $\C$ and $Y$.  It is convenient at this stage to fix three vertices of $\Delta$ and set the corresponding $a_{\um}$'s equal to $1$.  We shall mainly work in a neighborhood of the \emph{large complex structure limit} (LCSL) point $\uz=\uo$, though at times will also be concerned with the \emph{maximal conifold} point $\hat{\uz}$ --- the unique point (if it exists) on the ``boundary'' of that neighborhood\footnote{i.e., the region of convergence for certain power series representing the periods of $\C$; see \S\ref{S4}.} where $\C$ develops $g$ nodes (while remaining irreducible) hence has geometric genus zero.

What are the periods parametrized by \eqref{e2.2.1}?  We summarize some results from \cite{BKV}.\footnote{While stated there for $g=1$, the proof --- by ``limiting'' results of \cite{Ir} for compact CY 3-folds to the local setting --- works for any $\Delta$ that makes the BKV polytope $\NP:=\{\text{the convex hull of $(-1,1,0,0)$, $(2,-1,0,0)$, and $(-1,-1)\times\Delta$ in $\RR^4$}\}$ reflexive.  (For instance, take $\Delta$ to be the convex hull of $(1,0)$, $(0,1)$, and $(-g,-g)$ [resp. $(-n,-1)$] for $g\mid 6$ [resp. $n\mid 12$]).  We also expect these results to hold more generally.  A minor difference in formulation here is that instead of applying the BKV limit to derivatives of the prepotential $\Phi$ of a compact CY, we can directly take derivatives of $F_0$.}  One may construct 3-cycles $\T,\A_1,\ldots,\A_{g+r-3}$ on $Y$ such that near the LCSL
\begin{equation}\label{e2.2.2}
\int_{\T}\eta=2\pi\ay,\;\;\;\;-t_i:=\int_{\A_i}\eta\;\sim\; \log(z_i).
\end{equation}
The \emph{mirror map} $\uz\mapsto e^{\ut}$, which we usually express as $\ut(\uz)$ (or $\ut(\ua):=\ut(\uz(\ua))$) then induces a biholomorphism between neighborhoods of the LCSL and the large volume point (in K\"ahler moduli space\footnote{If $\{\J_i\}\subset H^2(X)$ is a basis dual to $\{C_i\}$, then the K\"ahler parameter is $\sum_i \tfrac{-t_i}{2\pi\ay} \J_i$.} of $X$).  Next write
\begin{equation}\label{e2.2.3}
\textstyle \F_0(\ut):=\tfrac{1}{6}\sum_{\ui}c_{i_1 i_2 i_3}t_{i_1}t_{i_2}t_{i_3}+\sum_{\ud\in H_2(X,\ZZ)_{\text{eff}}} N_{0,\ud}e^{-\ud\cdot \ut}
\end{equation}
for the \emph{genus-zero free energy} of $X$, in which the $c_{\ui}\in \QQ$ are certain triple intersection numbers\footnote{by interpreting $X$ as a (decompactifying) limit of a compact CY and computing intersections $-\J_{i_1}\J_{i_2}\J_{i_3}$ there; see \S\ref{S3} for details in the genus one case.} and the $N_{0,\ud}\in \QQ$ are genus-zero local Gromov-Witten numbers.  The basic \emph{Hodge-theoretic} assertion of local mirror symmetry is that there are 3-cycles $\B_1,\ldots,\B_g$ on $Y$ for which\footnote{The 2nd and 3rd terms are required in order for integrality of the periods, and arise from applying the procedure described in \cite{BKV}; the second term arises from the fact that $\mathrm{ch}(\co_{D_j})\equiv[D_j]-\tfrac{1}{2}[D_j^2]$ mod $\QQ[p]$, where $[p]$ is the class of a point.}
\begin{equation}\label{e2.2.4}
\textstyle\int_{\B_j}\eta =\tfrac{1}{2\pi\ay}\sum_{i=1}^{g+r-3} C_{ij} \partial_{t_i} \F_0(\ut)-\tfrac{1}{2}\sum_{i=1}^{g+r-3}A_{ij}t_i+2\pi\ay T_j
\end{equation}
under the mirror map, where $-C_{ij}=(\ell^{(i)}_{\um^{(j)}}=)\,C_i\cdot D_j$, $A_{ij}\underset{\scriptscriptstyle{(2)}}{\equiv}$ the coefficient of $C_i$ in $D_j^2$, and $T_j\in \QQ$.

The 3-cycles are constructed by describing $Y\to (\CC^*)^2$ as a conic bundle, with fibers isomorphic to $\CC^*$ over $(\CC^*)^2\setminus \C$, and to $\CC\cup_0\CC$ (pair of complex lines crossing once) over $\C$.  This yields (cf. \cite[\S 5.1]{DK}) an exact sequence of MHS
\begin{equation}\label{e2.2.5}
0\to\QQ(3)\overset{\mathtt{A}}{\to} H_3(Y) \overset{\mathtt{B}}{\to}  \ker\{H_1(\C)\to H_1((\CC^*)^2)\}(1)\to 0
\end{equation}
in which $\mathrm{im}(\mathtt{A})=\langle \T\rangle$ and the right-hand term has basis ($2\pi\ay$ times) $\alpha_1,\ldots,\alpha_{g+r-3},\beta_1,\ldots,\beta_g$.  On the level of $\QQ$-vector spaces, $\mathtt{B}$ has a section $\M$ sending this basis to the $\A_i=\M(\alpha_i)$ and $\B_j=\M(\beta_j)$.  It is constructed by sending $\vf\in \ker\{H_1(\C,\QQ)\to H_1((\CC^*)^2,\QQ)\}$ first to its bounding $\QQ$-chain $\Gamma_{\vf}$ in $(\CC^*)^2$ (with $\partial \Gamma_{\vf}=\vf$), over which $\M(\vf)$ is a 3-cycle with $S^1$ fibers (shrinking to points over $\vf$).  Writing $R\{f,g\}:=\log(f)\tfrac{dg}{g}-2\pi\ay\log(g)\delta_{T_f}$ for the standard regulator current for Milnor $K_2$-symbols ($T_f:=f^{-1}(\RR_{<0})$ the cut in branch of $\log$), we have on $(\CC^*)^2$ the relation $d[R\{-x,-y\}]=\tfrac{dx}{x}\wedge\tfrac{dy}{y}-(2\pi\ay)^2 \delta_{(\RR_{>0})^2}$.  This leads at once to
\begin{equation}\label{e2.2.6}
2\pi\ay \int_{\M(\vf)}\eta = \int_{\Gamma_{\vf}}\tfrac{dx}{x}\wedge\tfrac{dy}{y}=\int_{\vf}R\{-x,-y\}=:R_{\vf},
\end{equation}
which is to say that $R_{\alpha_i}=-2\pi\ay t_i$ and $R_{\beta_j}\equiv\sum_i C_{ij}\partial_{t_i}\F_0-\pi\ay\sum_i A_{ij}t_i$ mod $\QQ(2)$.

In the physics literature, the nontrivial $a_{\um}$ on the boundary are called \emph{mass parameters}; if we write these as $a_1',\ldots,a_{r-3}'$, then our complex structure parameters take the form $z_i=\prod_{j=1}^g a_j^{-C_{ij}}\times\prod_{k=1}^{r-3}{a_k'}^{C'_{ik}}$.  Taking the $a_j\gg 0$ large but keeping the $a_k'$ bounded, so that $t_i\sim \sum_{j=1}^g C_{ij}\log(a_j)$, the subleading terms (constant in $\ua$) can be shown\footnote{Done from a physics perspective in \cite{GKMR}, and from a regulator perspective in Appendix \ref{appA}.  Here ``negative roots'' means the roots of $P_{\ee}(-w)$.  In particular, if edge polynomials are powers of $(1+w)$, the $q_k$ are all $1$.} to be $\QQ$-linear combinations of logarithms of the negative roots $\{q_k\}_{k=1,\ldots,r}$ of the edge polynomials of $F$.  (The latter are defined as follows:  if $\ee$ is an edge of $\Delta$, with vertex $\uv$, and $\um^{\ee}\in \ZZ^2$ is a primitive lattice vector along $\ee$, then put $P_{\ee}(w):=\sum_{\um\in \ee\cap\ZZ^2} a_{\um}w^{(\um-\uv)/\um^{\ee}}$.)  The key observation is that each $q_k$ is the Tame symbol of $\{-x,-y\}\in K_2(\C)$ at a point $p_k\in \bar{\C}\cap (\PP_{\Delta}\setminus (\CC^*)^2)$, so that a loop $\ve_k\subset \C$ around $p_k$ has $\int_{\ve_k}R\{-x,-y\}=2\pi\ay \log(q_k)$.

The physicists have a \emph{grand potential function} $J_X(\ut;\hbar)$ which says ``everything they know how to say'' about enumerative geometry of $X$, and includes (refinements of) higher-genus GW-invariants.  We refer the reader to \cite{CGM} for details, as we shall only discuss two special cases in which those invariants (mostly) drop out.  First, in the \emph{maximally supersymmetric} case $\hbar=2\pi$, we have\footnote{Remark that $\uq$ is an abuse of notation since the $q_k$ are B-model coordinates; one would ideally replace them by monomials in the $e^{t_i}$ which equal $q_k$ under the mirror map.  (Similar remarks apply to $\underline{\mathsf{m}}$ in \eqref{e2.2.8}.)  But we don't need to be more precise here as these terms quickly become irrelevant.}
\begin{equation}\label{e2.2.7}
\begin{split}
\textstyle J_X(\ut;2\pi)=&\textstyle\tfrac{1}{8\pi^2}\left\{\sum_{i_1,i_2}\delta_{t_{i_1}}\delta_{t_{i_2}}-3\sum_i \delta_{t_i}+2\right\}\hat{\F}_0(\ut)\\
&+\hat{\F}_1(\ut)+\hat{\F}_1^{\text{NS}}(\ut)+A(\underline{q},2\pi),
\end{split}
\end{equation}
where $\hat{\F}_0,\hat{\F}_1,\hat{\F}_1^{\text{NS}}$ are free energies in which the instanton part is twisted by a ``B-field'' $\uBB\in \ZZ^{g+r-3}$:\footnote{In the $g=1$ case, $\uBB_i$ is just $C_{i1}$; see \S\ref{S2c} below and \cite{SWH} for $g>1$.  We will give Hodge-theoretic interpretations of $\ub,\ub^{\text{NS}}$ when $g=1$ in \S\ref{S3}.}
\begin{itemize}[leftmargin=0.5cm]
\item $\hat{\F}_0(\ut)=\tfrac{1}{6}\sum_{\ui}c_{\ui}t_{i_1}t_{i_2}t_{i_3}+\sum_{\ud}N_{0,\ud}e^{-\ud\cdot(\ut-\pi\ay \uBB)}$;
\item $\hat{\F}_1(\ut)=\sum_i b_i t_i + F_1^{\text{inst}}(\ut-\pi\ay\uBB)$; and
\item $\hat{\F}^{\text{NS}}_1(\ut)=\sum_i b^{\text{NS}}_i t_i + F_1^{\text{NS, inst}}(\ut-\pi\ay\uBB)$.
\end{itemize}
In the \emph{'t Hooft limit}, where $\hbar\to\infty$ (and $a_j\to \infty$) while $\mathsf{m}_k:=e^{-\frac{2\pi}{\hbar}\log(q_k)}$, $\zeta_j:=\tfrac{\log(a_j)}{\hbar}$, and $\tl_i:=\tfrac{2\pi t_i}{\hbar}$ remain finite, one finds that 
\begin{equation}\label{e2.2.8}
\textstyle \hbar^{-2}J_X(\ut;\hbar)=\underset{=:J^X_0(\underline{\zeta},\underline{\mathsf{m}})}{\underbrace{\{\tfrac{1}{16\pi^4}\hat{\F}_0(\underline{\tl})+\tfrac{1}{4\pi^2}\textstyle\sum_i b_i^{\text{NS}}\tl_i +A_0(\underline{\mathsf{m}})\} }}+O(\hbar^{-2}).
\end{equation}
We may disregard the unknown functions $A_0(\underline{\mathsf{m}}),A(\underline{q},2\pi)$ of the mass parameters.

To state the main physics conjecture, we need two more ingredients.   First is the \emph{quantum theta function}
\begin{equation}\label{e2.2.9}
\textstyle \Theta_X(\ut;\hbar):=\sum_{\un\in \ZZ^g} \exp\left\{J_X(\ut+2\pi\ay[C]\un;\hbar)-J_X(\ut;\hbar) \right\},
\end{equation}
where $[C]$ is the matrix $C_{ij}$ (and so $[C]\un$ is a $(g+r-3)$-vector with entries $\sum_{j=1}^g C_{ij}n_j$).  Terms in $J_X$ which are $2\pi\ay$-periodic in the $\{t_i\}$, including all but $\sum_i(b_i+b_i^{\text{NS}})t_i$ in the second line of \eqref{e2.2.7}, drop out.
The second is a ``quantum deformation'' $\qth(\uz)=\ut(\uz)+O(\hbar)$ of the mirror map.  (We shall also write $\qth(\ua):=\qth(\uz(\ua))$ where convenient.)  Again, we describe this where we need it:  at $\hbar=2\pi$ it is given by 
\begin{equation}\label{e2.2.10}
\qt_i(\uz):=\qt^{2\pi}_i(\uz)=t_i((-1)^{\uBB}\uz)+\pi\ay\BB_i;
\end{equation}
like $t_i(z)$, this is asymptotic to $-\log(z_i)$, but the signs are (in general) different in the power-series part.  In the `t Hooft limit, the previous asymptotic relation $t_i\sim \sum_j C_{ij}\log(a_j)+\sum_k D_{ik}\log(q_k)$ becomes exact in the sense that
\begin{equation}\label{e2.2.11}
\textstyle \tau_i=2\pi\sum_j C_{ij}\zeta_j - \sum_k D_{ik}\log(\mathsf{m}_k).
\end{equation}
\begin{conj}[\cite{GHM},\cite{CGM}]\label{c1}
Under the quantum mirror map, the generalized spectral determinant of $\C$ is given (up to a nonvanishing factor) by the quantum theta function of its mirror:
\begin{equation}\label{e2.2.12}
\Xi_{\C}(\ua;\hbar)=e^{J_X(\qth(\ua);\hbar)}\Theta_X(\qth(\ua);\hbar).
\end{equation}
\end{conj}
This postulates a \emph{fundamental and very general} relation between spectral theory (of the B-model) and enumerative geometry (of the A-model).  Since local mirror symmetry relates the latter to Hodge theory of the B-model, it should imply relationships between Hodge/$K$-theory and spectral theory of our curves with no reference to mirror symmetry.  We now derive these in our two special cases, under the assumption that $F$ is \emph{integrally tempered}:  all $q_k=1=\mathsf{m}_k$; equivalently, all edge polynomials of $F$ are powers of $w+1$.  Accordingly, by $\ua$ (resp. $\uz(\ua)$) we henceforth shall mean just $(a_1,\ldots, a_g)$, with the remaining $\{a_{\um}\}$ determined uniquely by this constraint.

\subsection{Consequences in the ``maximal SUSY'' case}\label{S2c}

Of course, the use of local mirror symmetry suggested in the last paragraph requires elaboration, since the classical and quantum mirror maps are not the same.  One should rather expect a relation between Hodge theory of $\C_{\uz}$ and spectral theory of a ``partner'' $\C_{\uz'}$ given by $\uz=\ut^{-1}(\qth(\uz'))$ or some variant thereof.  (In fact this is still insufficiently precise, since the spectral theory and the regulator class really depend on $\ua$.)  We now work this out at $\hbar=2\pi$.

First we address the nature and significance of $\uBB$.  Because the monomials $\ux^{\um}$ in $\hat{F}$ were quantized as $e^{m_1\hx+m_2\hy}=e^{\frac{\ay\hbar}{2}m_1 m_2}\hat{x}_1^{m_1}\hat{x}_2^{m_2}$, at $\hbar=2\pi$ we have $\hat{F}=\sum_{\um}(-1)^{m_1 m_2}a_{\um}\hat{\underline{x}}^{\um}$.  The B-field is determined mod $2$ by the effect on the signs of the $z_i$ were we to replace $a_{\um}$ by $(-1)^{m_1 m_2}a_{\um}$:  namely, $\BB_i\underset{\scriptscriptstyle{(2)}}{\equiv} \sum_{\um}m_1 m_2 \ell^{(i)}_{\um}$.  Under the assumption that
\begin{equation}\label{e2.3.1}
\partial\Delta\cap(2\ZZ\times 2\ZZ)=\emptyset,
\end{equation}
this is compatible with taking $\uBB$ to be in the $\ZZ$-span of the columns of $[C]$, which we write $\BB_i=\sum_{j=1}^g \mathbb{A}_j C_{ij}$.\footnote{mod $2$, $\uAA$ is just the characteristic function of $\Delta\cap(2\ZZ\times2\ZZ)$.}  Notice that then $\ut((-1)^{\uAA}\ua)=(-1)^{\uBB}\ut(\ua)$, so that by \eqref{e2.2.10} we have $\qtp((-1)^{\uAA}\ua)=\ut(\ua)+\pi\ay\uBB$ and the conjectured equality \eqref{e2.2.12} becomes
\begin{equation}\label{e2.3.2}
\Xi_{\C}((-1)^{\uAA}\ua;2\pi)=e^{J_X(\ut(\ua)+\pi\ay\uBB;2\pi)}\Theta_X(\ut(\ua)+\pi\ay\uBB;2\pi).
\end{equation}
That is, after absorbing the ``$+\pi\ay\uBB$'' twist into $\Theta_X$ and $J_X$, our Hodge/ spectral ``partners'' are related by at most a change of sign in the complex structure parameters.  The main question is what the \emph{quantization condition} looks like:  which values of $\ua$ make $\Theta_X(\ut(\ua)+\pi\ay\uBB;2\pi)$, hence the spectral determinant, zero?  

This is where the local mirror symmetry enters.  Under our assumption \eqref{e2.3.1}, its previous incarnation in \eqref{e2.2.4} can (by a tedious intersection theory argument) be expressed as\footnote{Although the regulator periods $R_{\vf}$ [resp. periods $\Omega_{j_1 j_2}$ in \eqref{e2.3.7} below] are infinitely multivalued, they are periods of a class $\R$ [resp. classes $\{\omega_j\}$] which are single-valued in $\ua$ [resp. $\uz$]; so we shall loosely write them as functions thereof.}
\begin{equation}\label{e2.3.3}
\textstyle R_{\beta_j}(\ua)=\sum_i C_{ij}\partial_{t_i}\hat{\F}_0\left(\ut(\ua)+\pi\ay\uBB\right)+(2\pi\ay)^2\mathtt{B}^{\circ}_j\;\;\;\;\;(\mathtt{B}^{\circ}_j\in \QQ).
\end{equation}
Next, since our temperedness assumption has eliminated the Tame symbols, the $\{R_{\alpha_i}\}_{i=1}^{g+r-3}$ are no longer independent (unless $r=3$).  More precisely, there are $g$ cycles $\gamma_j\in H_1(\bar{\C},\ZZ)$ with regulator periods $R_{\gamma_j}\sim -2\pi\ay\log(a_j)$ (cf. Appendix A), whence
\begin{equation}\label{e2.3.4}
\textstyle R_{\alpha_i}=\sum_j C_{ij}R_{\gamma_j};
\end{equation}
and the $\mathbb{A}_j$ can be chosen so that $\{\gamma_j,\beta_j\}_{j=1}^g$ is a symplectic basis.\footnote{This is again by local mirror symmetry:  the $R_{\gamma_j}$ [resp. $R_{\alpha_i}$] are the A-model periods of flat sections arising from curves dual to the $D_j$ [resp. $\mathcal{J}_i$]; while the $R_{\beta_j}$ are those arising from $\mathrm{ch}(\co_{D_j}(-E_j))\cup \hat{\Gamma}(X)$ for suitable curves $E_j$.}  The regulator class $\R=R\{-x_1,-x_2\}\in H^1(\bar{\C},\CC/\ZZ(2))$ then has a local lift\footnote{For our purposes, this can be regarded as living on an open neighborhood (in $\uz$-space $\CC^g$) of $(0,\epsilon)^g$ for some $\epsilon >0$.} to $H^1(\bar{\C},\CC)$ given by
\begin{equation}\label{e2.3.5}
\textstyle \RT=\sum_{\ell=1}^g \left( R_{\gamma_{\ell}}\gamma_{\ell}^*+R_{\beta_{\ell}}\beta_{\ell}^*\right),
\end{equation}
whose Gauss-Manin derivatives
\begin{equation}\label{e2.3.6}
\textstyle\omega_j:=\nabla_{{\partial}/{\partial R_{\gamma_j}}}\RT=\gamma_j^*+\sum_{\ell=1}^g \tfrac{\partial R_{\beta_{\ell}}}{\partial R_{\gamma_j}}\beta_{\ell}^*
\end{equation}
are classes of holomorphic 1-forms by Griffiths transversality.  Evidently these are normalized so that the symmetric $g\times g$ matrix
\begin{equation}\label{e2.3.7}
\begin{split}
\Omega_{j_1 j_2}(\uz):&=\textstyle -\tfrac{1}{2\pi\ay}\sum_{i_1,i_2}C_{i_1 j_1}C_{i_2 j_2}\partial_{t_{i_1}}\partial_{t_{i_2}}\hat{\F}_0(\ut(\uz) +\pi\ay\uBB)\\
&=\textstyle -\tfrac{1}{2\pi\ay}\sum_{i_1}C_{i_1 j_1}\partial_{t_{i_1}}R_{\beta_{j_2}}=\textstyle \sum_{i_1} C_{i_1 j_1}\tfrac{\partial R_{\beta_{j_2}}}{\partial R_{\alpha_{i_1}}}\\
&=\textstyle\tfrac{\partial R_{\beta_{j_2}}}{\partial R_{\gamma_{j_1}}}=\int_{\gamma_{j_1}}\omega_{j_2}
\end{split}
\end{equation}
is the standard period matrix of $\bar{\C}$.

We have already observed that the isomorphism class of $\bar{\C}$ depends only on $\uz$, which parametrizes the standard coarse moduli space for toric hypersurfaces; and we are restricting to a ``tempered slice'' of this space.  However, $\R$ only becomes single-valued in $\ua$, forcing us to work on the finite cover $\bm{\M}:=\{\ua\in (\CC^*)^g\mid \C_{\uz(\ua)}\;\text{is smooth}\}$ of this slice.  Let $\bm{\bar{\C}}\overset{\pi}{\to}\bm{\M}$ be the universal (compactified) curve, and set $\mathscr{H}:=R^1 \pi_*\CC\otimes\co_{\bm{\M}}$, $\HH:=R^1\pi_*\ZZ$, and $\mathscr{J}:=\bm{\H}/\{\bm{\HH}+\bm{\F^1 \H}\}$.  Then $\mathscr{J}$ is the sheaf of sections of the Jacobian bundle $\bm{\J}\overset{\bm{\rho}}{\to}\bm{\M}$, and $\mathscr{H}/\HH$ is the sheaf of sections of the $\CC/\ZZ$ cohomology bundle $\bm{\H}^1_{\CC/\ZZ}\to \bm{\M}$, which factors through the obvious $\CC^g$-torsor $\bm{\H}^1_{\CC/\ZZ}\overset{\bm{\varpi}}{\to}\bm{\J}$.  By temperedness, the symbol $\{-x_1,-x_2\}\in K_2(\CC(\bm{\C}))$ lifts to a motivic cohomology class $\bm{\Z}\in H_{\M}^2(\bar{\bm{\C}},\ZZ(2))$, and we make the key

\begin{defn}\label{d2.3.1}
By the \emph{higher normal function} associated to $\bm{\Z}$, we shall mean the well-defined section $\tfrac{1}{(2\pi\ay)^2}\R$ of $\bm{\H}^1_{\CC/\ZZ}$, or its projection $\bm{\nu}:=\bm{\varpi}(\tfrac{1}{(2\pi\ay)^2}\R)$ to a section of $\bm{\J}$.  The latter is computed by evaluating $\R$ as a functional on holomorphic 1-forms (modulo periods), i.e. by the column vector
\begin{equation}\label{e2.3.8}
\begin{split}
\nu_j:&\textstyle =\tfrac{1}{(2\pi\ay)^2}\langle \R,\omega_j\rangle\;\;\;(j=1,\ldots,g)\\
&\textstyle =\tfrac{-1}{4\pi^2}\sum_{\ell=1}^g\langle R_{\gamma_{\ell}}\gamma_{\ell}^*+R_{\beta_{\ell}}\beta_{\ell}^*,\,\gamma_j^*+\sum_{\ell'}\Omega_{j\ell'}\beta_{\ell'}^*\rangle\\
&\textstyle =\tfrac{1}{4\pi^2}(\sum_{\ell=1}^g R_{\gamma_{\ell}}\Omega_{j\ell}-R_{\beta_j})
\end{split}
\end{equation}
modulo the $\ZZ$-span of columns of $(\mathbb{I}_g\mid \Omega)$.
\end{defn}

\noindent To use mirror symmetry to compute $\uv$, put $\tilde{R}_{\beta_j}:=R_{\beta_j}-(2\pi\ay)^2\mathtt{T}_j$, and observe that by \eqref{e2.3.3} thru \eqref{e2.3.7} (together with $\Omega_{jj'}=\Omega_{j'j}$)
\begin{equation}\label{e2.3.9}
\begin{split}
\xi_j(\ua):&=\textstyle\tfrac{1}{4\pi^2}\sum_{i_1} C_{i_1 j}(\sum_{i_2}\delta_{t_{i_2}}-1)\partial_{t_{i_1}}\hat{\F}_0(\ut(\ua)+\pi\ay\uBB)\\
&\textstyle =\tfrac{1}{4\pi^2}(\sum_i \delta_{t_i}-1)\tilde{R}_{\beta_j}=\tfrac{1}{4\pi^2}(\tfrac{-1}{2\pi\ay}\sum_i R_{\alpha_i}\partial_{t_i}R_{\beta_j} - \tilde{R}_{\beta_j})\\
&\textstyle =\tfrac{1}{4\pi^2}(\tfrac{-1}{2\pi\ay}\sum_{i,\ell} C_{i\ell}R_{\gamma_j}\partial_{t_i}R_{\beta_j} - \tilde{R}_{\beta_j})\\
&\textstyle =\tfrac{1}{4\pi^2}(\sum_{\ell}R_{\gamma_{\ell}}\Omega_{j\ell}-\tilde{R}_{\beta_j}) \;=\; \nu_j - \mathtt{B}^{\circ}_j.
\end{split}
\end{equation}

Returning to the quantization condition, the exponent in \eqref{e2.2.9} is
\begin{multline}\label{e2.3.10}
J_X(\ut+2\pi\ay[C]\un;2\pi)-J_X(\ut;2\pi)\\ \textstyle=\pi\ay{}^t\un[\hat{\Omega}]\un+2\pi\ay \un\cdot\hat{\xi}-\tfrac{\pi\ay}{3}\sum_{\ui,\uj}c_{\ui}\prod_{\ell=1}^3 C_{i_{\ell}j_{\ell}}n_{j_{\ell}},
\end{multline}
where
\begin{itemize}[leftmargin=0.5cm]
\item $\hat{\Omega}_{j_1 j_2}:=\tfrac{-1}{2\pi\ay}\sum_{i_1,i_2}C_{i_1 j_1}C_{i_2 j_2}\partial_{t_{i_1}}\partial_{t_{i_2}}\hat{\F}_0(\ut)$ and
\item $\hat{\xi}_j:=\tfrac{1}{4\pi^2}\sum_{i_1}C_{i_1 j}(\sum_{i_2}\delta_{t_{i_2}}-1)\partial_{t_{i_1}}\hat{\F}_0(\ut)+\sum_i C_{ij}(b_i+b_i^{\text{NS}})$
\end{itemize}
by a straightforward computation, cf. \cite[(3.28)]{CGM}.  Substituting in $\ut=\ut(\ua)+\pi\ay\uBB$, the first two terms of \eqref{e2.3.10} become
\begin{equation}\label{e2.3.11}
\textstyle \pi\ay{}^t\un[\Omega(\ua)]\un+2\pi\ay\un\cdot(\uv(\ua)+\underline{\mathtt{B}}+\tfrac{1}{2}[\Omega(\ua)]\underline{\mathbb{A}})
\end{equation}
(for $\underline{\mathtt{B}} \in \QQ^g$) by \eqref{e2.3.7}-\eqref{e2.3.9}.  By an intersection theory argument and the identity $n^3\underset{\scriptscriptstyle{(6)}}{\equiv} n$, the cubic third term becomes $-\tfrac{\pi\ay}{3}\sum_j n_j D_j^3$ mod $\ZZ(1)$, which may be absorbed into $\underline{\mathtt{B}}$.  Therefore, writing $\underline{\mathtt{A}}:=\tfrac{1}{2}\uAA$ and $\theta$ for the usual Jacobi theta function,
\begin{equation}\label{e2.3.12}
\Theta_X(\ut(\ua)+\pi\ay\uBB;2\pi)=\theta(\uv(\ua)+\underline{\mathtt{B}}+[\Omega(\ua)]\underline{\mathtt{A}},[\Omega(\ua)]).
\end{equation}

We have thus deduced from Conjecture \ref{c1} a striking relationship between the quantization condition and the higher normal function.  Let $\bm{\mathcal{D}}_{\theta}\subset \bm{\J}$ be the theta divisor and $\bm{\mathcal{D}}_{\theta}[\substack{\underline{\mathtt{A}}\\ \underline{\mathtt{B}}}]$ its translate by (minus) the torsion section $\underline{\mathtt{B}}+[\Omega]\underline{\mathtt{A}}$.

\begin{conj}\label{c2}
For $\Delta$ satisfying \eqref{e2.3.1} and $F$ integrally tempered, the zero-locus of the twisted spectral determinant $\Xi_{\C}((-1)^{\uAA}\ua;2\pi)$ is exactly the locus where the normal function meets this torsion shift of the theta divisor:  as subsets of $\bm{\M}$, we have
\begin{equation}\label{e2.3.13}
\mathrm{ZL}\left(\Xi_{\C}((-1)^{\uAA}\ua;2\pi) \right)=\bm{\rho}\left(\bm{\nu}(\bm{\M})\cap \bm{\mathcal{D}}_{\theta}[\substack{\underline{\mathtt{A}}\\ \underline{\mathtt{B}}}]\right).
\end{equation}
\end{conj}

\noindent In genus $g=1$, there are $15$ reflexive polygons (up to unimodular transformation) which can be presented inside $\RR\times [-1,1]$.  After making the torsion shifts completely explicit in \S\ref{S3a}, we prove the ``$\supseteq$'' direction of \eqref{e2.3.13} for these cases in \S\ref{S3b}.

\subsection{Consequences in the `t Hooft limit}\label{S2d}

Our spectral determinant $\Xi_{\C}$ has \emph{fermionic spectral traces} which generalize, from the ($g=1$) case of a single operator, the traces of $\rho_1^{\otimes N}$ acting on $\bigwedge^N L^2(\RR)$, cf. \cite[\S3.3]{CGM}.  Defined by
\begin{equation}\label{e2.4.1}
\textstyle \Xi_{\C}(\ua;\hbar)=:\sum_{N_1,\ldots,N_g\geq 0} Z_{\C}(\uN,\hbar)\ua^{\uN},
\end{equation}
these can clearly also be expressed in terms of loop integrals about $0$:
\begin{equation}\label{e2.4.2}
Z_{\C}(\uN,\hbar)= \frac{1}{(2\pi\ay)^g}\oint\cdots \oint \Xi_{\C}(\ua;\hbar)\frac{da_1}{a_1^{N_1+1}}\wedge\cdots\wedge \frac{da_g}{a_g^{N_g+1}}.
\end{equation}
Applying Conjecture \ref{c1} replaces $\Xi_{\C}(\ua;\hbar)$ by $\sum_{\un\in \ZZ^g}e^{J_X(\qth(\ua)+2\pi\ay[C]\un;\hbar)}$, where the $2\pi\ay[C]\un$ simply accounts for the change in $\qth(\ua)$ as the $a_j$ go $n_j$ times around $0$ --- or equivalently, as $\mu_j:=\log(a_j)$ increases by $2\pi\ay n_j$ (for each $j$).  Accordingly, \eqref{e2.4.2} becomes
\begin{equation}\label{e2.4.3}
\textstyle\tfrac{1}{(2\pi\ay)^g}\int_{-\ay\infty}^{\ay\infty}\cdots \int_{-\ay\infty}^{\ay\infty}e^{J_X(\qth(\ua);\hbar)-\sum_{j=1}^g N_j\mu_j}d\mu_1\wedge\cdots\wedge d\mu_g,
\end{equation}

Recall from \S\ref{S2b} that the `t Hooft limit takes $\hbar\to \infty$ while essentially fixing $\zeta_j=\tfrac{\mu_j}{\hbar}$ and $\tau_i=\tfrac{2\pi t_i}{\hbar}$, which we will also impose on $\lambda_j:=\tfrac{N_j}{\hbar}$.  As temperedness makes the $q_k=1$ hence $\mathsf{m}_k=1$, we write $J^X_0(\underline{\zeta}):=J^X_0(\underline{\zeta},\underline{1})$, and note that \eqref{e2.2.11} reduces to $\tau_i=2\pi\sum_j C_{ij}\zeta_j$.  
\begin{rem}\label{r2.4.1}
In fact, even if we don't assume temperedness, but \emph{fix the edge polynomials} hence the $\{q_k\}$, the effect is the same since $\mathsf{m}_k (=e^{-\frac{2\pi}{\hbar}\log(q_k)}) = 1$ in the limit.  
\end{rem}

Now by \eqref{e2.2.8}, for $\hbar\gg0$ \eqref{e2.4.3} becomes
\begin{equation}\label{e2.4.4}
\textstyle \tfrac{\hbar^g}{(2\pi\ay)^g}\int_{-\ay\infty}^{\ay\infty}\cdots \int_{-\ay\infty}^{\ay\infty}e^{\hbar^2\{J^X_0(\underline{\zeta})-\sum_j \lambda_j\zeta_j+O(\hbar^{-2})\}}d\zeta_1\wedge\cdots\wedge d\zeta_g;
\end{equation}
and we write $\hat{\underline{\zeta}}(\underline{\lambda})$ for the stationary point of (the leading part of) the exponential, where $0=\partial_{\zeta_i}(J^X_0(\underline{\zeta})-\sum_j\lambda_j \zeta_j)$, or equivalently $
\lambda_j=\partial_{\zeta_j} J^X_0(\underline{\zeta})$, for each $j$.  By the saddle-point method, we can write \eqref{e2.4.4} as $\exp(\hbar^2\{J^X_0(\hat{\underline{\zeta}}(\underline{\lambda}))-\sum_j \lambda_j \hat{\zeta}_j(\underline{\lambda})+O(\hbar^{-2})\})$, which is to say that
\begin{equation}\label{e2.4.5}
\lim_{\hbar\to \infty}(\partial_{\lambda_j}\hbar^{-2}\log Z_{\C}(\hbar\underline{\lambda},\hbar))|_{\underline{\lambda}=\uo}=-\hat{\zeta}_j(\uo).
\end{equation}
Moreover, according to \cite[\S 2.3]{CGM}, $\hat{\tau}_i(\underline{\lambda})=2\pi\sum_j C_{ij}\hat{\zeta}_j(\underline{\lambda})$ is nothing but the classical mirror map in the ``conifold frame'', with $\underline{\lambda}$ a parameter which vanishes at the maximal conifold point $\hat{\uz}$.\footnote{We are not aware of a proof of this statement, but there is strong computational evidence; it is also consistent with the observation, in view of \eqref{e2.3.3}, that the vanishing of $\partial_{\zeta_j}J^X_0(\underline{\zeta})$ at $\hat{\underline{\zeta}}(\uo)$ is equivalent to that of a $\QQ(2)$-translate of $R_{\beta_j}(\ua)$ at $\ua\in\ut^{-1}(\hat{\underline{\tau}}(\uo)-\pi\ay\uBB)$.  This is exactly what should happen at a $g$-nodal fiber.}  In other words, if $\hat{\ua}$ is any preimage of $\hat{\uz}$ in $\bm{\overline{\M}}$, then we have $R_{\alpha_i}(\hat{\ua})\equiv {-2\pi\ay}\hat{\tau}_i(\uo)$ and 
\begin{equation}\label{e2.4.6}
R_{\gamma_j}(\hat{\ua})\equiv -4\pi^2\ay\hat{\zeta}_j(\uo)\;\;\; \text{mod $\QQ(2)$}.
\end{equation}

On the other hand, if we set $N_j=0$ for $j>1$, then the asymptotic expansion of $Z_{\C}(N_1,0\ldots,0;\hbar)=\mathrm{tr}_{\wedge^{N_1}L^2(\RR)}((\rho^{(0)}_1)^{\otimes N_1})$ can be computed via operator theory and asymptotic properties of the quantum dilogarithm.  This is worked out in \cite{KM,MZ} for the three-term operators $(\rho_1^{(0)})^{-1}=e^{\hx}+e^{\hy}+e^{-m\hx-n\hy}$, corresponding to the Laurent polynomials
\begin{equation}\label{e2.4.7}
\textstyle F^{\circ}_{m,n}(\ux):=x_1+x_2+x_1^{-m}x_2^{-n}+\sum_{j=1}^g a_j x_1^{m^{(j)}_1}\mspace{-10mu}x_2^{m^{(j)}_2}.
\end{equation}
(Here we recall that the $\{\um^{(j)}\}$ index the interior integral points of $\Delta$; for instance, if $m=n=g$, then $\um^{(j)}=(1-j,1-j)$.)  Note that by Remark \ref{r2.4.1}, $\underline{\hat{\tau}}(\underline{\lambda})$ will actually compute the mirror map/regulator periods in the conifold frame \emph{for the families defined by the integrally tempered polynomials}\footnote{Of course, there is no distinction between \eqref{e2.4.7} and \eqref{e2.4.7a} if $g_1=1=g_2$.}
\begin{equation}\label{e2.4.7a}
\begin{split}
\textstyle F_{m,n}(\ux)&:=x_1\textstyle+x_2+x_1^{-m}x_2^{-n}+\sum_{j=1}^g a_j x_1^{m^{(j)}_1}\mspace{-10mu}x_2^{m^{(j)}_2} \\
&\textstyle+\sum_{\ell=1}^{g_1-1}\binom{g_1}{\ell}x_1^{1-\ell\frac{m+1}{g_1}}x_2^{-\ell\frac{n}{g_1}}+\sum_{\ell=1}^{g_2-1}\binom{g_2}{\ell}x_1^{-\ell\frac{m}{g_2}}x_2^{1-\ell\frac{n+1}{g_2}},
\end{split}
\end{equation}
where $g_1:=\gcd(m+1,n)$ and $g_2=\gcd(m,n+1)$.  Anyway, the result of [op. cit.] (see also \cite[\S 4.3]{Ma}) is that
\begin{multline}\label{e2.4.8}
\lim_{\hbar\to \infty}(\partial_{\lambda_1}\hbar^{-2}\log Z_{\C}(\hbar\lambda_1,0,\ldots,0;\hbar))|_{\lambda_1=0}\\=\tfrac{m+n+1}{2\pi^2}D_2(-\fzmn^{m+1}\fw_{m,n}),
\end{multline}
where $D_2$ is the Bloch-Wigner function, $\fzmn:=e^{\frac{\pi\ay}{m+n+1}}$, and $\fw_{m,n}:=\tfrac{\fzmn^m-\fzmn^{-m}}{\fzmn-\fzmn^{-1}}$.  Since LHS\eqref{e2.4.8} must agree with LHS\eqref{e2.4.5} (with $j=1$), in view of \eqref{e2.4.6} we arrive at

\begin{conj}\label{c3}
For the families $\C_{m,n}$ arising from \eqref{e2.4.7a}, the regulator period $R_{\gamma_1}$ asymptotic to $-2\pi\ay\log(a_1)$ at the origin has value
\begin{equation}\label{e2.4.9}
\tfrac{1}{2\pi\ay}R_{\gamma_1}(\hat{\ua})\equiv \tfrac{m+n+1}{\pi}D_2(-\fzmn^{m+1}\fw_{m,n}) =:\mathcal{D}_{m,n}\;\;\;\textup{mod $\QQ(1)$}
\end{equation}
at the maximal conifold point.
\end{conj}

\begin{example}\label{ex2.4.1}
A toric coordinate change brings $F_{2,2}$ into the form $F_{3,1}$, but with $a_1$ and $a_2$ swapped.  So Conjecture \ref{c3} actually yields predictions for both nontrivial regulator periods at $\hat{\ua}=(5,-5)$, namely $\tfrac{1}{2\pi\ay}R_{\gamma_1}(\hat{\ua})\equiv \mathcal{D}_{2,2}=\tfrac{5}{\pi}D_2(e^{\frac{2\pi\ay}{5}}\fw)$ and $\tfrac{1}{2\pi\ay}R_{\gamma_2}(\hat{\ua})\equiv \mathcal{D}_{3,1}=\tfrac{5}{\pi}D_2(e^{\frac{\pi\ay}{5}}\fw)$ mod $\QQ(1)$, where $\fw:=\tfrac{1+\sqrt{5}}{2}$.  This assertion was checked in \cite{7K} by a computation we will generalize (and make more rigorous) in \S\ref{S4}.
\end{example}

\section{From higher normal functions to eigenfunctions}\label{S3}

In this section we state and prove a precise version of Conjecture \ref{c2} in the genus 1 case.

\subsection{Integral mirror symmetry and quantization conditions}\label{S3a}

The condition $g=1$ is equivalent to reflexivity of $\Delta$, whereupon $X$ becomes simply the total space of $K_{\PP_{\Delta^{\circ}}}$.  There is a unique compact toric divisor $D=D_1\cong \PP_{\Delta^{\circ}}\subset X$, corresponding to the ray through $(1,0,0)$, which amounts to the zero-section of $\rho\colon X\twoheadrightarrow D$.  Denoting by $E^{\circ}\subset D$ a general anticanonical (elliptic) curve, we remark that $D^2=-E^{\circ}$ in $H^*_c(X)$.

Let $\vf$ be the unique integrally tempered Laurent polynomial with Newton polygon $\Delta$, constant term $0$, and coefficients $1$ at the vertices, and (writing $a=a_1$) take $F=a+\vf$.  After compactifying fibers in $\PP_{\Delta}$ and birationally modifying the total space, this produces a relatively minimal elliptic fibration $\E\to \PP^1_a$ with rational total space, fibers $E_a$, and discriminant locus $\Sigma\cup\{\infty\}$.  Writing $r:=|\partial\Delta\cap\ZZ^2|$ and $r^{\circ}:=|\partial\Delta^{\circ}\cap\ZZ^2|$, $E_{\infty}$ has type $\mathrm{I}_{r^{\circ}}$, and $\Sigma$ is cut out by a polynomial $P_{\Sigma}$ of degree $12-r^{\circ}=r$.\footnote{For a \emph{generic} choice of $\vf$, the remaining singular fibers of $\E$ are $\mathrm{I}_1$'s.  Since $\E$ is rational (as a blowup of $\PP_{\Delta}$), the degree of the relative dualizing sheaf must be $1$; and as each $\mathrm{I}_k$ contributes $\tfrac{k}{12}$ to this degree, there must be $12-r^{\circ}$ $\mathrm{I}_1$'s.  Each of these contributes $1$ to $\deg(P_{\Sigma})$, and this degree is invariant as we specialize $\vf$.}

A section of the relative dualizing sheaf for our family is given by 
\begin{equation}\label{e3.1.1}
\omega(a):=\tfrac{1}{2\pi\ay}\mathrm{Res}_{E_a}(\tfrac{{dx_1}/{x_1}\wedge{dx_2}/{x_2}}{1+a^{-1}\vf(\ux)}),
\end{equation}
with period\footnote{$[\cdot]_{\uo}$ takes the constant term; $\gamma$ is $\gamma_1$ from \S\ref{S2c}.} 
\begin{equation}\label{e3.1.2}
\textstyle \omega_{\gamma}(a):=\int_{\gamma}\omega(a)=1+\sum_{k>0}(-1)^k[\vf^k]_{\uo}a^{-k}
\end{equation}
in a neighborhood of the large complex structure point $\infty$.  More precisely, this series converges on $\DD^*:=\{a\mid |a|>|\hat{a}|\}\subset U:=\PP^1\setminus (\Sigma\cup\{\infty\})$, where the \emph{conifold point} $\hat{a}$ can be described by $-\hat{a}:=\min(\vf(\RR_+\times\RR_+))$ since the coefficients of $\vf$ are all positive \cite{Ga}.

By assumption, all the tame symbols of $\{-x_1,-x_2\}$ are trivial, and so the $R_{\alpha_i}$ ($i=1,\ldots,r-2$) must be integer multiples of $R_{\gamma}\sim-2\pi\ay \log(a)$.  More precisely, we have $\tfrac{-1}{2\pi\ay}R_{\alpha_i}=t_i=C_{i1}t=-(C_i\cdot D)t=d_i t$, where $d_i\in [0,4]\cap \ZZ$ is the lattice-length of the edge of $\partial\Delta$ corresponding to $C_i$.  From Appendix A, we have on the cut disk $\DD^-:=\DD^*\setminus(\DD^*\cap \RR_-)$
\begin{equation}\label{e3.1.3}
\textstyle t=t(a):=\tfrac{-1}{2\pi\ay}R_{\gamma}(a)=\log(a)+\sum_{k>0}\tfrac{(-1)^{k-1}}{k}[\vf^k]_{\uo}a^{-k},
\end{equation}
which gives $\omega=\tfrac{-1}{2\pi\ay}\nabla_{\delta_a}\R$ hence (in the notation of \S\ref{S2c}) $\omega_1=\omega/\omega_{\gamma}$ globally on $U$.  We also see that $e^{-t}\sim a^{-1}$ makes sense as a coordinate on $\DD=\DD^*\cup\{\infty\}$.  The local mirror symmetry results in \cite{BKV} can be made very explicit:\footnote{Here as above $\beta=\beta_1$, $\Omega=\Omega_{11}$, $\nu=\nu_1$.}

\begin{lem}\label{l3a1}
On $\DD^-$ we have the following identifications:
\begin{enumerate}[label=\textup{(\alph*)}, leftmargin=1cm]
\item $R_{\beta}(a)=\tfrac{r^{\circ}}{2}t(a)^2+\pi\ay r^{\circ}t(a)+(2\pi\ay)^2(\tfrac{1}{2}+\tfrac{r^{\circ}}{12})-\sum_{k>0}k\mathfrak{N}_k e^{-kt(a)}$,
\item $\Omega(a)\,(=\tfrac{\omega_{\beta}(a)}{\omega_{\gamma}(a)})=\tfrac{\ay r^{\circ}}{2\pi}t(a)-\tfrac{r^{\circ}}{2}-\tfrac{1}{2\pi\ay}\sum_{k>0} k^2\mathfrak{N}_k e^{-kt(a)}$, and
\item $\nu(a)=\tfrac{r^{\circ}}{8\pi^2}t(a)^2+(\tfrac{1}{2}+\tfrac{r^{\circ}}{12})+\tfrac{1}{4\pi^2}\sum_{k>0}k(1+kt(a))\mathfrak{N}_k e^{-kt(a)}$,
\end{enumerate}
where $\mathfrak{N}_k$ is the local GW-invariant for $D$ counting rational curves whose classes $\mathtt{C}\in H_2(D)$ satisfy $(\mathtt{C}\cdot E^{\circ})_D=k$.
\end{lem}

\begin{proof}
$X$ is described in \cite[\S6]{BKV} as the large-fiber-volume limit of an elliptically-fibered compact CY 3-fold $W\to \PP_{\Delta^{\circ}}$ with section $D$.  Let $C_1,\ldots,C_r$ be the components of $\PP_{\Delta^{\circ}}\setminus (\CC^*)^2$ (and their images in $X$), $D'_i:=\rho^{-1}(C_i)$, and $C_0:=\rho^{-1}(\text{pt})$.  Then $\{C_0,C_1,\ldots,C_{r-2}\}$ span $H^4(W,\QQ)$, $\{D,D'_1,\ldots,D'_{r-2}\}$ span $H^2(W,\QQ)$, and we can write $-D^2=E^{\circ}= \sum_{i=1}^r C_i =\sum_{i=1}^{r-2}e_i C_i$ for unique $e_i\in \QQ$, whereupon $D^3=\sum_{i=1}^{r-2}d_i e_i=r^{\circ}$.  Let $J_0,\ldots,J_{r-2}$ denote a basis of $H^2(W,\QQ)$ dual to $C_0,\ldots, C_{r-2}$, and define $\J_1,\ldots,\J_{r-2}$ by $\J_i:=J_i-\tfrac{e_i}{r^{\circ}}J_0$.  Then the $c_{\ui}$ in \eqref{e2.2.3} are given by $c_{i_1 i_2 i_3}=-\J_{i_1}\J_{i_2}\J_{i_3}$.\footnote{The results of [loc. cit.] are stated in terms of derivatives of the prepotential $\Phi(t_0,\ut)$ of $W$ in the limit as $t_0\to \infty$.  One can obtain the free energy $\F_0(\ut)$ for $X$ by substituting $t_0=-\sum_{i=1}^{r-2}\tfrac{e_i}{r^{\circ}}t_i$ into $\Phi^{\text{cl}}$ and taking $t_0\to \infty$ in $\Phi^{\text{inst}}$; we then have $\tfrac{1}{(2\pi\ay)^3}\partial_D\Phi=\tfrac{1}{(2\pi\ay)^2}(-\partial_0+\sum_i d_i\partial_i)\Phi=\tfrac{1}{(2\pi\ay)^2}\sum_i d_i \partial_i \F_0$, hence the version of the A-model periods given here.}

The integral periods of the A-model VHS given by \cite[(6.13-15)]{BKV} lead (in the LMHS as $t_0\to 0$) to the following periods for our A-model VMHS.  First, the limit of the Gamma class for $W$ yields $\hat{\Gamma}(X):=1-\tfrac{1}{2}D^2+(\tfrac{11r^{\circ}+r}{24})C_0=1+\sum_{i=1}^{r-2}e_i C_i + (\tfrac{1}{2}+\tfrac{5}{12}r^{\circ})C_0\in H^*(X,\QQ)$.  Next, for integral periods we need to compose $\mathrm{ch}(\cdot)\cup \hat{\Gamma}(X)\colon K_0^{c,\text{num}}(X)\to H^*_c(X,\QQ)$ with the following assignment of periods to cohomology classes:  $\text{pt}\mapsto 1$; $C_i\mapsto \tfrac{1}{2\pi\ay}t_i=\tfrac{-1}{(2\pi\ay)^2}R_{\alpha_i}$; and $D\mapsto \tfrac{1}{(2\pi\ay)^2}\sum_{i=1}^{r-2}d_i\partial_{t_i}\F_0(\ut)$.  Applying this to $\co_D$, we have $\mathrm{ch}(\co_D)=D-\tfrac{1}{2}D^2+\tfrac{1}{6}D^3$, whence $\mathrm{ch}(\co_D)\cup\hat{\Gamma}(X)=D+\tfrac{1}{2}\sum_i e_i C_i+(\tfrac{1}{2}+\tfrac{r^{\circ}}{12})$, and finally (after multiplying the resulting integral period by $(2\pi\ay)^2$)
\begin{equation}\label{e3.1.4}
\textstyle R_{\beta}=\sum_i d_i \partial_{t_i}\F_0(\ut)+\pi\ay \sum_i e_i t_i + (2\pi\ay)^2(\tfrac{1}{2}+\tfrac{r^{\circ}}{12}).
\end{equation}
We also recall from \eqref{e2.3.7} that the period ratio is given by $\Omega=\tfrac{-1}{2\pi\ay}\sum_i d_i \partial_{t_i}R_{\beta}$, and the normal function by $\nu=\tfrac{1}{4\pi^2}(R_{\gamma}\Omega-R_{\beta})$.

The last step is to substitute $t_i=d_i t$, which gives
\begin{equation}\label{e3.1.5}
\textstyle \F_0(t)=-\tfrac{1}{6}(\sum_i \J_i t_i)^3+\sum_{\mathtt{C}} N_{0,\mathtt{C}}e^{-(\mathtt{C}\cdot E^{\circ})_D t}=\tfrac{r^{\circ}}{6}t^3+\sum_{k>0}\mathfrak{N}_k e^{-kt}
\end{equation}
since $\sum_i \J_i d_i=\sum_i d_i J_i - \sum_i \tfrac{e_i d_i}{r^{\circ}}J_0=(J_0-D)-J_0=-D$ \cite[(6.5)]{BKV}.  Using $d_i\partial_{t_i}=\partial_t$ in \eqref{e3.1.4}ff now gives (a)-(c).
\end{proof}

\begin{rem}\label{r3a1}
We point out two immediate consequences of Lemma \ref{l3a1}.  First, along with \eqref{e3.1.3}, (c) makes it clear that $\nu(a)$ as well as
\begin{equation}\label{e3.1.6}
V(a):=\omega_{\gamma}(a)\nu(a)=\tfrac{1}{4\pi^2}(R_{\gamma}\omega_{\beta}-R_{\beta}\omega_{\gamma})
\end{equation}
are real-valued on $\DD^*\cap \RR_+$.  Second, notice that $\tfrac{1}{(2\pi\ay)^2}\partial_t^2 R_{\beta}=\partial_{R_{\gamma}}^2 R_{\beta}=\partial_{R_{\gamma}}\tfrac{\delta_a R_{\beta}}{\delta_a R_{\gamma}}=\partial_{R_{\gamma}}\tfrac{\omega_{\beta}}{\omega_{\gamma}}=\tfrac{\mathcal{Y}(a)}{\omega_{\gamma}^3}$, where the Yukawa coupling $\mathcal{Y}(a)=\omega_{\gamma}\delta_a \omega_{\beta}-\omega_{\beta}\delta_a\omega_{\gamma}$ blows up at $\hat{a}$.  Differentiating (a) twice expresses this as a power series in $e^{-t}$, from which one deduces that
\begin{equation}\label{e3.1.7}
\textstyle \limsup_{k\to \infty} \sqrt[k]{|\mathfrak{N}_k|}=\exp(\Re(t(\hat{a}))).
\end{equation}
as in \cite[\S5.4]{DK} (though this result in now unconditional).
\end{rem}

We may now identify all of the torsion constants in \S\S\ref{S2b}-\ref{S2c}:\footnote{Again, for simplicity writing $\mathtt{T}=\mathtt{T}_1$, $\mathtt{B}^{\circ}=\mathtt{B}^{\circ}_1$, $\mathtt{B}=\mathtt{B}_1$, and $\mathtt{A}=\mathtt{A}_1$.}

\begin{lem}\label{l3a2}
In $\QQ/\ZZ$ the following equalities hold:
\begin{enumerate}[label=\textup{(\roman*)}, leftmargin=1cm]
\item $b:=\sum_i d_i b_i=\tfrac{r^{\circ}}{12}-\tfrac{1}{2}$ and $b^{\textup{NS}}:=\sum_i d_i b_i^{\textup{NS}}=\tfrac{r^{\circ}}{24}-\tfrac{1}{2}$.
\item $\mathtt{T}=\tfrac{1}{2}+\tfrac{r^{\circ}}{12}$ and $\mathtt{B}^{\circ}=\tfrac{1}{2}-\tfrac{r^{\circ}}{24}$.
\item $\mathtt{A}=\tfrac{1}{2}=\mathtt{B}$, where $\mathtt{B}$ is as in \eqref{e2.3.12}-\eqref{e2.3.13}.\footnote{and \emph{not} as in \eqref{e2.3.11}, where $\mathtt{B}$ does not yet incorporate the correction from the cubic term.}
\end{enumerate}
\end{lem}
\begin{proof}
(i) These are the coefficients of $t$ in $\F_1$ and $\F_1^{\text{NS}}$ (after substituting $t_i=d_i t$), which can be derived from \cite[(4.18) and (4.21)]{GKMR}.\footnote{We should point out here that our ``$r$'' is not the ``$r$'' in \cite{GKMR}, where it means $\gcd\{d_i\}$.  (Moreover, their ``$t$'' is $r_{\text{GKMR}}$ times our $t$.)}  Namely, we have $b_i=\tfrac{1}{24}c_2(X)\cdot \J_i$ \cite[(4.18)]{GKMR} and $c_2(X)=(11r^{\circ}+r)C_0+12 \sum_i e_i C_i=(10r^{\circ}+12)C_0-12 D^2$ \cite[\S6.2]{BKV} hence $b=\tfrac{1}{24}c_2(X)\cdot \sum_i d_i\J_i=-\tfrac{1}{24}c_2(X)\cdot D=-\tfrac{10r^{\circ}+12}{24}+\tfrac{12r^{\circ}}{24}=\tfrac{r^{\circ}}{12}-\tfrac{1}{2}$.  According to \cite[(4.21)]{GKMR}, we have $\F_1^{\text{NS}}\sim -\tfrac{1}{24}\log(P_{\Sigma}(a))\sim -\tfrac{\deg(P_{\Sigma})}{24}\log(a)\sim -\tfrac{r}{24}t\sim (\tfrac{r^{\circ}}{24}-\tfrac{1}{2})t$.  (So of course, (i) holds in $\QQ$, but we'll only need it mod $\ZZ$.)

(ii) The value of $\mathtt{T}$ is immediate from Lemma \ref{l3a1}(a).  To compute $\mathtt{B}^{\circ}=\nu(a)-\xi(a)$, we need to revisit $\xi$ from \eqref{e2.3.9}.  The B-field is given by $\BB_i=d_i$ (cf. \S\ref{S2c} above or \cite[\S3.2]{GKMR}), and $\mathbb{A}=\mathbb{A}_1=1$, which means that replacing $\ut$ by $\ut+\pi\ay\BB$ is equivalent to replacing $t$ by $t+\pi\ay$.  Together with $\sum_i\delta_{t_i}=t\sum_i d_i\partial_{t_i}=t\partial_t=\delta_t$ and \eqref{e3.1.5}, this gives
\begin{equation}\label{e3.1.8}
\begin{split}
\xi(a)&\textstyle =\tfrac{1}{4\pi^2}(\delta_t-1)\partial_t \hat{\F}_0(t(a)+\pi\ay)\\
&\textstyle =\tfrac{r^{\circ}}{8\pi^2}t(a)^2+\tfrac{r^{\circ}}{8}+\tfrac{1}{4\pi^2}\sum_{k>0}k(1+kt(a))\mathfrak{N}_k e^{-kt(a)}
\end{split}
\end{equation}
and, together with Lemma \ref{l3a1}(c), the claimed value of $\mathtt{B}^{\circ}$.

(iii) We already have $\mathtt{A}=\tfrac{1}{2}\mathbb{A}=\tfrac{1}{2}$.  For $\mathtt{B}$, we compute
\begin{equation}\label{e3.1.9}
\begin{split}
\hat{\xi}(t(a)+\pi\ay)&\textstyle =\tfrac{1}{4\pi^2}\left((t+\pi\ay)\partial_t-1\right)\partial_t\hat{\F}_0(t(a)+\pi\ay)+(b+b^{\text{NS}})\\
&\textstyle =\xi(a)+\tfrac{\pi\ay}{4\pi^2}\partial_t^2\hat{\F}_0(t(a)+\pi\ay)+(b+b^{\text{NS}})\\
&\textstyle =\nu(a)+\tfrac{1}{2}\Omega(a)+(b+b^{\text{NS}}-\mathtt{B}^{\circ})
\end{split}
\end{equation}
and note that the cubic term in \eqref{e2.3.10} becomes $-\tfrac{\pi\ay}{3}D^3n^3=-\tfrac{r^{\circ}}{3}\pi\ay n^3\equiv -\tfrac{r^{\circ}}{6}2\pi\ay n$ mod $\ZZ(1)$.  Together with (i)-(ii), this results in the apparently miraculous cancellation
\begin{equation}\label{e3.1.10}
\mathtt{B}=b+b^{\text{NS}}-\mathtt{B}^{\circ}-\tfrac{r^{\circ}}{6}=-\tfrac{3}{2}\equiv \tfrac{1}{2}
\end{equation}
modulo $\ZZ$.
\end{proof}

Finally, we turn to the quantization conditions, i.e. to the spectrum (as an operator on $L^2(\RR)$) of\footnote{Remark that $\vf=F_1$ and $\rho=\rho_1$ in the notation of \S\ref{S2a}.  We have $m_1m_2\underset{\scriptscriptstyle{(2)}}{\equiv}m_1+m_2+1$ because \eqref{e2.3.1} always holds for reflexive polygons.}
\begin{equation}\label{e3.1.11}
\begin{split}
\hat{\vf}&\textstyle =\sum_{\um\in \partial\Delta\cap\ZZ^2}(-1)^{m_1m_2}a_{\um}\hat{x}_1^{m_1}\hat{x}_2^{m_2}\\
&\textstyle =\sum_{\um\in \partial\Delta\cap\ZZ^2}(-1)^{m_1+m_2+1}a_{\um}\hat{x}_1^{m_1}\hat{x}_2^{m_2}=-\vf(-\hat{x}_1,-\hat{x}_2)
\end{split}
\end{equation}
or $\rho:=\hat{\vf}^{-1}$.  Writing $\sigma(\cdot)$ for spectrum and $\Lambda(a):=\ZZ\langle \omega_{\gamma}(a),\omega_{\beta}(a)\rangle$ for the period lattice, we have the

\begin{prop}\label{p3a}
In the genus-1 case, Conjecture \ref{c2} is equivalent to
\begin{equation}\label{e3.1.12}
\sigma(\hat{\vf})=\{a\in U\mid V(a)\in \Lambda(a)\}.
\end{equation}
\end{prop}
\begin{proof}
Noting that $\bm{\M}=U$, in the LHS of \eqref{e2.3.13} we are taking the zero-locus of $\Xi(-a;2\pi)=\det(1-a\rho)$, which is precisely the spectrum of $\hat{\vf}$.  The RHS of \eqref{e2.3.13} is the locus in $U$ where $\nu(a)$ meets the theta divisor (which is $\tfrac{1+\Omega(a)}{2}$ mod $\ZZ\langle 1,\Omega(a)\rangle$) shifted by $\mathtt{A}\Omega(a)+\mathtt{B}=\tfrac{1+\Omega(a)}{2}$, which is to say \emph{where $\nu(a)$ is zero mod $\ZZ\langle 1,\Omega(a)\rangle$}.  Outside of $\DD^-$, this condition is only well-defined in the sense of analytic continuation; to fix this, we multiply by $\omega_{\gamma}$ to get the form displayed in RHS\eqref{e3.1.12}.
\end{proof}

\begin{rem}\label{r3a2}
(i) The condition $V(a)\in \Lambda(a)$, which is well-defined on $U$, reduces to $\nu(a)\in \ZZ\langle 1,\Omega(a)\rangle$ for $a\in \DD^-$.  Moreover, the argument in \cite[\S3.1]{LST} using the coherent state representation  shows more generally (for any $\vf$ considered here) that $\sigma(\hat{\vf})$ belongs to $\RR_+$, and is countable with eigenvalues $\lambda_j$ limiting to $\infty$ (so that $\rho$ is bounded).  In fact, we expect that $\sigma(\hat{\vf})\subset (|\hat{a}|,\infty)$, as is clear for $\vf=x_1+x_1^{-1}+x_2+x_2^{-1}$ or $x_1+x_1^{-1}+x_2+x_2^{-1}+x_1x_2^{-1}+x_1^{-1}x_2$ and experimentally observed in other cases.  This would mean that the quantization condition ``$V\in \Lambda$'' reduces not just to $\nu\in \ZZ\langle 1,\Omega\rangle$, but to 
\begin{equation}\label{e3.1.13}
\nu(a)\in \ZZ,
\end{equation}
as $\nu$ is real by Remark \ref{r3a1}.  We'll have more to say about this in \S\ref{S3b}.

(ii) The most crucial ``torsion'' invariant in Lemma \ref{l3a2}, leading to the cancellation in \eqref{e3.1.10} and the simple form of \eqref{e3.1.12}, is surely the constant term $\mathtt{T}$ of the regulator period $R_{\beta}$.  As an independent check, one can directly compute this constant term without using mirror symmetry and the Gamma class; see Appendix A for examples.  Another check on our quantization condition is that it should coincide with that in \cite[\S3.3.2]{GKMR} when all $Q_{m_k}=1$ ($\implies D_0(\um)=0$ and $B(\um,2\pi)=b+b^{\text{NS}}=\tfrac{r^{\circ}}{8}-1$).  Since $\mathrm{vol}_0(E)$ in \cite[(3.24)]{GKMR} is just $R_{\beta}$, we may also identify ``$C$'' there as $\tfrac{r^{\circ}}{2}$.  Taking $E=\log(a)$ and $E_{\text{eff}}=t(a)$, \cite[(3.105)]{GKMR} collapses to $\xi(a)-\tfrac{r^{\circ}}{24}\in \ZZ+\tfrac{1}{2}$, hence to $\nu(a)\in \ZZ$.

(iii) There is an interesting sign discrepancy in \eqref{e3.1.12}:  quantizability of $\hat{\vf}-a$ is being linked to a regulator class on the curve $E_a\subset \PP_{\Delta}$ compactifying solutions to $\vf(\ux)+a=0$.  Blame it on the B-field!  Or better yet, proceed to the next section for a more basic reason why it has to be this way.
\end{rem}

\subsection{Construction of eigenfunctions for difference operators}\label{S3b}

In this section we assume that $\Delta$ is a reflexive polygon satisfying
\begin{equation}\label{e3.2.1}
\Delta \subset \RR\times[-1,1],
\end{equation}
and $\vf$ is as in \S\ref{S3a}, so that
\begin{equation}\label{e3.2.2}
\vf(\ux)=x_1^{m_u}(x_1+1)^{d_u}x_2+\vf_0(x_1)+x_1^{m_{\ell}}(x_1+1)^{d_{\ell}}x_2^{-1}.
\end{equation}
\begin{rem}\label{r3b1}
Regarding unimodular change of coordinates ($x_1,x_2\mapsto x_1^{\mathtt{a}}x_2^{\mathtt{b}},x_1^{\mathtt{c}}x_2^{\mathtt{d}}$ with $\mathtt{ad}-\mathtt{bc}=1$) as an equivalence relation on reflexive polygons, there are 16 equivalence classes.  All but one\footnote{represented by $\Delta=$ convex hull of $\{(-1,-1),\,(2,-1),\,(-1,2)\}$, with $\PP_{\Delta}=\PP^2$} of these has representatives satisfying \eqref{e3.2.1}.
\end{rem}

For each $a\in U$, $E_a\subset \PP_{\Delta}$ denotes as before the Zariski closure of $E_a^*:=\{\ux\in (\CC^*)^2\mid \vf(\ux)+a=0\}$.
Forgetting $x_2$ produces a $2:1$ map $\pi\colon E_a\to \PP^1$ with corresponding involution $\iota\colon E_a\to E_a$ and discriminant
\begin{equation}\label{e3.2.3}
(\vf_0(x_1)+a)^2-4x_1^{m_u+m_{\ell}}(x_1+1)^{d_u+d_{\ell}}=:\mathscr{D}(x_1).
\end{equation}
The latter is a Laurent polynomial (in $x_1$) with ``Newton polytope'' an interval $[-c_-,c_+]$ containing $[-1,1]$ (and contained in $[-2,2]$), whose length is the number of ramification points of $ \pi^{-1}(\CC^*)=:E_a^{\times}\overset{\pi^{\times}}{\to} \CC^*$; denote the set of these by $\mathfrak{B}\subset E_a^{\times}$, and let $p_0\in\mathfrak{B}$ be one of them.  The holomorphic function
\begin{equation}\label{e3.2.4}
\delta(p):=x_1(p)^{m_u}(x_1(p)+1)^{d_u}(x_2(p)-x_2(\iota(p))),
\end{equation}
on $E_a^{\times}$ satisfies $\delta^2=(\pi^{\times})^*\mathscr{D}$, thereby providing a well-defined lift of $\sqrt{\mathscr{D}}$ to $E_a^{\times}$.

Writing $\tilde{E}_a^{\times}$ for the fiber product of $\pi^{\times}$ and $(-\exp)\colon \CC\to \CC^*$ yields a diagram
\begin{equation}\label{e3.2.5}
\xymatrix{E_a \ar @{->>} [d]_{\pi} && E_a^{\times} \ar @{_(->} [ll] \ar @{->>} [d]_{\pi^{\times}} && \tilde{E}_a^{\times} \ar @{->>} [ll]_{\mathcal{P}} \ar @{->>} [d]_{\Pi} & \tilde{z} \ar @{|->} [d] \ar @{}[l]|{\ni}\\
\PP^1 && \CC^* \ar @{_(->} [ll] && \CC \ar @{->>} [ll]^{-\exp} & z \ar @{}[l]|{\ni} }
\end{equation}
with vertical maps of degree 2, and points in $\tilde{E}_a^{\times}$ [resp. $\CC$] denoted by $\tilde{z}$ [resp. $z=\Pi(\tilde{z})$].  We also write $\mathcal{P}(\tilde{z})=:(x_1(\tilde{z}),x_2(\tilde{z}))$, where $x_1(\tilde{z})=x_1(z)=-e^z$, and $\tilde{z}_0\in \tilde{E}_a^{\times}$ for the point with $\mathcal{P}(\tilde{z}_0)=p_0$ and $\Im(z_0)\in (-\pi,\pi]$.  For later reference put $\tilde{E}_a^*:=\mathcal{P}^{-1}(E_a^*)$, which is either all of $\tilde{E}_a^{\times}$ or the complement of $\Pi^{-1}(\ZZ(1))$.\footnote{There are 4 equivalence classes of ploygons for which $\tilde{E}_a^*=\tilde{E}_a^{\times}$, corresponding to $X=\PP^2$, $\PP^1\times\PP^1$, $\FF_1$, and $\FF_2$.  Otherwise, for $\tilde{z}\in \tilde{E}_a^{\times}\setminus\tilde{E}_a^*$, in view of \eqref{e3.2.2} we have $-1=x_1(\tilde{z})=x_1(z)=-e^z$ $\implies$ $z\in \ZZ(1)$.}

Now \textbf{suppose} $\bm{V(a)\in \Lambda(a)}$.  If $a\in \DD^-$, then $\gamma,\beta,\omega_{\gamma},\omega_{\beta},\Omega,R_{\gamma},R_{\beta}$, and $\nu$ are well-defined; if not, we take them to be analytic continuations (along the same path) to $a$ of those objects from $\DD^-$.  (We will not write $\omega(a)$ etc., just $\omega$, since $a$ is fixed and understood.)  Then we have
\begin{equation}\label{e3.2.6}
\nu=\tfrac{1}{4\pi^2}(R_{\gamma}\Omega-R_{\beta})=n_1+n_2\Omega
\end{equation}
for some $n_1,n_2\in\ZZ$.  Notice that the regulator class $\R$ is only well-defined in $H^1(E_a,\CC/\ZZ(2))$, so its value on $\gamma$ is still represented by $\mathfrak{R}_{\gamma}:=R_{\gamma}-4\pi^2n_2$.  This replaces \eqref{e3.2.6} by
\begin{equation}\label{e3.2.7}
R_{\beta}-\mathfrak{R}_{\gamma}\tfrac{\omega_{\beta}}{\omega_{\gamma}}=-4\pi^2 n_1\in\ZZ(2),
\end{equation}
and we claim this allows us to define a holomorphic function on $\tilde{E}_a^*$ by
\begin{equation}\label{e3.2.8}
\textstyle \chi(\tilde{z}):=\exp\left(\tfrac{\ay}{2\pi}\left\{\int_{\mathscr{P}_{\tilde{z}_0}^{\tilde{z}}} z \tfrac{dx_2(\tilde{z})}{x_2(\tilde{z})}-\tfrac{\mathfrak{R}_{\gamma}}{\omega_{\gamma}}\int_{\mathscr{P}_{\tilde{z}_0}^{\tilde{z}}}\mathcal{P}^*\omega\right\}\right),
\end{equation}
where $\omega$ is as in \eqref{e3.1.1}, and $\mathscr{P}_{\tilde{z}_0}^{\tilde{z}}$ is any path from $\tilde{z}_0$ to $\tilde{z}$.

The issue here is well-definedness, since nothing in the braces blows up on $\tilde{E}_a^*$.  To check this, we remind the reader that for a loop $\mathscr{L}$ on $E_a^*$ based at $p_0$, the value of $\R$ on its homology class is computed by\footnote{Of course, $\mathrm{dlog}(-x)=\mathrm{dlog}(x)=\tfrac{dx}{x}$.  Note that \eqref{e3.2.9}, which is due to Beilinson \cite{Be} and Deligne [unpublished], is different from the regulator formula using the current $R\{-x_1,-x_2\}$ (in which the function ``$\log$'' is not analytically continued but has a branch cut), but is easily shown to give the same integral regulator.} 
\begin{equation}\label{e3.2.9}
\textstyle R_{\mathscr{L}}\underset{\scriptscriptstyle{\ZZ(2)}}{\equiv}\int_{\mathscr{L}}\log(-x_1)\mathrm{dlog}(-x_2)-\log(-x_2(p_0))\int_{\mathscr{L}}\mathrm{dlog}(-x_1),
\end{equation}
where $\log(-x_1)$ is analytically continued along $\mathscr{L}$ \cite{Ke}.  If $\mathscr{L}$ lifts to a loop $\tilde{\mathscr{L}}$ on $\tilde{E}_a^*$, then clearly $\int_{\mathscr{L}}\mathrm{dlog}(x_1)=0$, and \eqref{e3.2.9} pulls back to $\int_{\tilde{\mathscr{L}}}z\tfrac{dx_2(\tilde{z})}{x_2(\tilde{z})}$.  Now given two paths $\mathscr{P},\mathscr{P}'$ from $\tilde{z}_0$ to $\tilde{z}$ on $\tilde{E}_a^*$, take $\tilde{\mathscr{L}}$ to be the loop obtained by composing $\mathscr{P}$ with the ``reverse'' of $\mathscr{P}'$, and write $\mathscr{L}=k_1\gamma+k_2\beta$ in $H_1(E_a,\ZZ)$.  (By integral temperedness of $\{-x_1,-x_2\}$, this determines $R_{\mathscr{L}}$ mod $\ZZ(2)$.)  The difference between the braced expression in \eqref{e3.2.8} for these two paths is then
\begin{equation}\label{e3.2.10}
\begin{split}
\textstyle\int_{\tilde{\mathscr{L}}}z\tfrac{dx_2(\tilde{z})}{x_2(\tilde{z})}-\tfrac{\mathfrak{R}_{\gamma}}{\omega_{\gamma}}\int_{\tilde{\mathscr{L}}}\mathcal{P}^*\omega&\textstyle =\int_{\mathscr{L}}\log(-x_1)\mathrm{dlog}(x_2)-\tfrac{\mathfrak{R}_{\gamma}}{\omega_{\gamma}}\int_{\mathscr{L}}\omega\\
&\mspace{-3mu}\underset{\scriptscriptstyle{\ZZ(2)}}{\equiv} k_1R_{\gamma}+k_2R_{\beta}-\tfrac{\mathfrak{R}_{\gamma}}{\omega_{\gamma}}(k_1\omega_{\gamma}+k_2\omega_{\beta})\\
&= k_1(R_{\gamma}-\mathfrak{R}_{\gamma})+k_2(R_{\beta}-\mathfrak{R}_{\gamma}\Omega)\\
&=4\pi^2(k_1n_2-k_2n_1) \underset{\scriptscriptstyle{\ZZ(2)}}{\equiv} 0,
\end{split}
\end{equation}
using \eqref{e3.2.7}.  After multiplying by $\tfrac{\ay}{2\pi}$, this discrepancy is killed by the $\exp$ and the claim is verified.

In fact, $\chi(\tilde{z})$ extends to a meromorphic function on $\tilde{E}_a^{\times}$ which is holomorphic at $\Pi^{-1}(0)$.  Of course, $\omega$ has no poles on $E_a$, and so $\mathcal{P}^*\omega$ has none on $\tilde{E}_a^{\times}$; the potential culprit is $\tfrac{dx_2}{x_2}$, when $d_u,d_{\ell}$ are not both zero.  Writing $z=2\pi\ay n+w+O(w^2)$, $x_2=w^d$ (for $d=-d_u$ or $d_{\ell}$), we find $\int z\tfrac{dx_2}{x_2}\sim 2\pi\ay dn\log(w)$ hence $\exp(\tfrac{\ay}{2\pi}\int z\tfrac{dx_2}{x_2})\sim w^{-nd}$, as desired.

Finally, writing $\tilde{\iota}\colon \tilde{E}_a^{\times}\to \tilde{E}_a^{\times}$ for the involution over $\CC$, we put
\begin{equation}\label{e3.2.11}
\tilde{\Psi}(\tilde{z}):=\frac{\chi(\tilde{z})-\chi(\tilde{\iota}(\tilde{z}))}{\delta(\mathcal{P}(\tilde{z}))}.
\end{equation}
The denominator has zeroes at $\mathcal{P}^{-1}(\mathfrak{B})$, which does not intersect any of the poles of the numerator.\footnote{The only way $\iota$ has a fixed point at $x_1=-1$ is if $d_u=d_{\ell}=0$.}  Moreover, these are simple zeroes, and the numerator also has zeroes at these points (which are just the fixed points of $\tilde{\iota}$).  So $\tilde{\Psi}$ is holomorphic on $\tilde{E}_a^{\times}\setminus \Pi^{-1}(\ZZ(1){\setminus}\{0\})$.  Notice also that applying $\tilde{\iota}$ to $\tilde{z}$ changes the sign in the numerator and denominator of \eqref{e3.2.11} (since $\mathcal{P}\circ\tilde{\iota}=\iota\circ \mathcal{P}$).  We conclude that there exists a meromorphic function $\Psi$ on $\CC$, with (at worst) poles on $2\pi\ay(\ZZ\setminus \{0\})$, such that $\tilde{\Psi}=\Pi^*\Psi$; we write this loosely as
\begin{equation}\label{e3.2.12}
\Psi(z):=\frac{\chi(\tilde{z})-\chi(\tilde{\iota}(\tilde{z}))}{\delta(\mathcal{P}(\tilde{z}))},
\end{equation}
and denote its restriction to the real line by $\psi(r)$.  We are now ready to prove the

\begin{thm}\label{t1}
For $\Delta$ satisfying \eqref{e3.2.1}, the \textup{``}$\supseteq$\textup{''} direction of \eqref{e3.1.12} holds.  That is, if $V(a)\in \Lambda(a)$, then $a\in \sigma(\hat{\vf})$.
\end{thm}
\begin{proof}
First note that $\hat{x}_1=$ multiplication by $e^r$ (\emph{not} $-e^r$), $\hat{x}_2=e^{-2\pi\ay\partial_r}$, and $\hat{\vf}=-\vf(-\hat{x}_1,-\hat{x}_2)$ are \emph{unbounded} operators on $L^2(\RR)$, whose domains are roughly the \emph{proper} linear subspaces on which each operator preserves square integrability.  (See \cite{LST} for details.)  In particular, it is possible in this sense to be in the domain of $\hat{\vf}$ while failing to be in that of $\hat{x}_1^{\pm 1}$ and $\hat{x}_2^{\pm 1}$, which is just what happens for $\psi(r)$.  Indeed, assuming $V(a)\in \Lambda(a)$, we claim that $\psi\in L^2(\RR)\setminus\{0\}$ and 
\begin{equation}\label{e3.2.13}
\hat{\vf}\psi=a\psi,
\end{equation}
which will obviously prove the theorem.

As $\Psi$ is holomorphic on $\{z\in \CC\mid -2\pi\ay<\Im(z)<2\pi\ay\}$, with meromorphic extension to a neighborhood of its closure, we have
\begin{equation}\label{e3.2.14}
\begin{split}
e^{\pm 2\pi\ay\partial_r}\psi(r)&= e^{\pm 2\pi\ay\partial_z}\Psi(r)=\Psi(r\pm 2\pi\ay)\\
&=: \Psi(\tau_{\pm}(r))=:(\mathscr{S}_{\pm}\Psi)(r)=:(\mathscr{S}_{\pm}\psi)(r).
\end{split}
\end{equation}
Furthermore, $\tau_{\pm}$ has a unique lift $\tilde{\tau}_{\pm}\colon \tilde{E}_a^{\times}\to \tilde{E}_a^{\times}$ with the property that $\P\circ \tilde{\tau}_{\pm}=\P$; and so the difference operator $\mathscr{S}_{\pm}$ lifts to $(\tilde{\mathscr{S}}_{\pm}\chi)(\tz):=\chi(\tilde{\tau}_{\pm}(\tz))$.  By the independence of path in \eqref{e3.2.8}, we can take our path from $\tz_0$ to $\tilde{\tau}_{\pm}(\tz)$ to be the composition of $\tilde{\tau}_{\pm}(\pz)$ with a fixed path $\po$ from $\tz_0$ to $\tilde{\tau}_{\pm}(\tz_0)$.  That is, writing $\P(\po)=:\lo$, we have
\begin{equation}\label{e3.2.15}
\begin{split}
\chi(\tilde{\tau}_{\pm}(\tz))&=\textstyle\exp\left(\tfrac{\ay}{2\pi}\left\{\int_{\tilde{\tau}_{\pm}(\pz)+\po}z\tfrac{dx_2(\tz)}{x_2(\tz)}-\tfrac{\mathfrak{R}_{\gamma}}{\omega_{\gamma}}\int_{\tilde{\tau}_{\pm}(\pz)+\po}\P^*\omega\right\}\right)\\
&=\textstyle\exp\left(\tfrac{\ay}{2\pi}\left\{\int_{\pz}(z\pm 2\pi\ay)\tfrac{dx_2(\tz)}{x_2(\tz)}-\tfrac{\mathfrak{R}_{\gamma}}{\omega_{\gamma}}\int_{\pz}\P^*\omega \right\}\right) \\ 
& \mspace{130mu}\textstyle\times \exp\left(\tfrac{\ay}{2\pi}\left\{\int_{\lo}\log(-x_1)\tfrac{dx_2}{x_2}-\tfrac{\mathfrak{R}_{\gamma}}{\omega_{\gamma}}\int_{\lo}\omega\right\}\right).
\end{split}
\end{equation}
Adding and subtracting $-\log(-x_2(\tz_0))\int_{\lo}\tfrac{dx_1}{x_1}\,(=\mp 2\pi\ay \log(-x_2(\tz_0))\,)$ in the last braced expression, \eqref{e3.2.15} becomes
\begin{equation}\label{e3.2.16}
\chi(\tz) e^{\mp\{\log(-x_2(\tz))-\log(-x_2(\tz_0))\}}\times e^{\frac{\ay}{2\pi}\{R_{\lo}-\frac{\mathfrak{R}_{\gamma}}{\omega_{\gamma}}\omega_{\lo}\}} e^{\mp\log(-x_2(\tz_0))}.
\end{equation}
By the same calculation as in \eqref{e3.2.10}, we have $R_{\lo}-\tfrac{\mathfrak{R}_{\gamma}}{\omega_{\gamma}}\omega_{\lo}\in \ZZ(2)$, and so after cancelling $\log(-x_2(p_0))$'s, we arrive at
\begin{equation}\label{e3.2.17}
(\tilde{\mathscr{S}}_{\pm}\chi)(\tz)=-x_2(\tz)^{\pm 1}\cdot \chi(\tz).
\end{equation}

Since $-\hat{x}_1=-\mu_{e^r}=\mu_{-e^r}=\mu_{x_1(r)}$, $\hat{\vf}$ acts on $\psi$ as $-\vf(\mu_{x_1(r)},-\mathscr{S}_-)$, which lifts to $-\vf(\mu_{x_1(r)},-\tilde{\mathscr{S}}_-)$ for functions on $\tilde{E}_a^{\times}$.  Applying this to $\chi(\tz)$ gives $-\vf(x_1(z),x_2(\tz))\cdot \chi(\tz)=a\chi(\tz)$, and applying it to $\chi(\tilde{\iota}(\tz))$ yields $-\vf(x_1(z),x_2(\tilde{\iota}(\tz)))\cdot \chi(\tilde{\iota}(\tz))=a\chi(\tilde{\iota}(\tz))$.  (Here we are just using the equation of the curve, $\vf(x_1(z),x_2(\tz))+a=0$; and we can ignore $\delta(\P(\tz))$ in the denominator of $\tilde{\Psi}$ since $\tilde{\mathscr{S}}_{\pm}$ doesn't affect it.)  So the overall effect on $\tilde{\Psi}$, hence $\psi$, is multiplication by $a$.  This proves \eqref{e3.2.13}.

We still need to check is that $\psi$ is indeed square-integrable.  Clearly $\int\P^*\omega$ has a finite limit as $r\to \pm \infty$, so we consider the behavior of 
\begin{equation}\label{e3.2.18}
\textstyle \int r\tfrac{dz_2(\tilde{r})}{z_2(\tilde{r})}=\int \log(-x_1(r))\mathrm{dlog}(-x_2(\tilde{r})).
\end{equation}
Let $q\in E_a\setminus E_a^{\times}$, and set $o_j:=\mathrm{ord}_q(x_j)$; then $(-1)^{o_1o_2}\lim_{p\to q}\tfrac{x_1(p)^{o_2}}{x_2(p)^{o_1}}=1$ by integral temperedness.  Hence there is a local holomorphic coordinate $w$ on $E_a$ vanishing at $q$, with $-x_1=w^{o_1}$ and $-x_2=\pm w^{o_2}(1+O(w))$, and $\eqref{e3.2.18}=\tfrac{o_1o_2}{2}\log^2 w + O(w\log w)$ is just $\tfrac{o_2}{2o_1}r^2$ (with $o_1\neq 0$) plus terms limiting to zero.  Since this is multiplied by $\tfrac{\ay}{2\pi}$ before taking $\exp$, we conclude that $\chi(\tz)$ is bounded on $\Pi^{-1}(\RR)$.  On the other hand, in the denominator $\delta(\P(\tilde{r}))=\sqrt{\mathscr{D}(-e^r)}$ of $\psi$, $\mathscr{D}(-e^r)=\sum_{j=-c_-}^{c_+}\mathtt{a}_j e^{jr}$ ($\mathtt{a}_{-c_-},\mathtt{a}_{c_+}\neq 0$) is dominated by the $e^{c_+r}$ term as $r\to +\infty$ and the $e^{-c_-r}$ term as $r\to -\infty$.  That is, $|\psi(r)|\leq C e^{-|r/2|}$ for some constant $C$, hence $\psi$ belongs to $L^2(\RR)$.

Finally, we must show that $\psi$ is not identically zero.  If it were, then by basic complex analysis $\Psi$ would be zero; so it suffices to check that (say) $\Psi(z_0+2\pi\ay n)\neq 0$ for some $n\in \ZZ$.  We may choose a local holomorphic coordinate $u$ on $\tilde{E}_a^{\times}$ about $\tilde{z}_0$, such that (locally) $\tilde{\iota}$ sends $u\mapsto -u$ and $z=z_0+u^2$.  Clearly $x_2(\tz)=x_2(p_0)(1+c_1u+O(u^2))$ and $\P^*\omega=(c_2+O(u))du$ for constants $c_1,c_2\in \CC^*$.  The expression in braces in \eqref{e3.2.8} (integrating on a path from $\tz_0$ to $\tz(u)$) takes the form $(c_1z_0-\tfrac{\mathfrak{R}_{\gamma}}{\omega_{\gamma}}c_2)u+O(u^2)$, and we can ensure the coefficient of $u$ is nonzero by replacing $z_0$ by $z_0+2\pi\ay n$ if necessary (since this affects nothing else).  So the numerator of \eqref{e3.2.11} becomes $e^{c_0u+O(u^2)}-e^{-c_0u+O(u^2)}\sim 2c_0u$, and since the denominator also has a simple zero at $u=0$ we are done.
\end{proof}

\begin{rem}\label{r3b2}
Returning to the ``sign flip'' between curve and operator highlighted in Remark \ref{r3a2}(iii), we remind the reader that it is $\{-x_1,-x_2\}$, not $\{x_1,x_2\}$, which is integrally tempered for the simplest choices of Laurent polynomial $\vf$.\footnote{e.g. $x_1+x_2+x_1^{-1}x_2^{-1}$, and including the examples studied in \cite{GKMR} with trivial mass invariants $Q_{m_k}=1$.}  So it is the regulator integral for \emph{this symbol} which produces a well-defined $\tilde{\Psi}(\tz)$.  But the signs in the symbol force the shift operator $\hat{x}_2$ to act on $\chi(\tz)$ through multiplication by $-x_2(\tz)$ rather than $x_2(\tz)$, which in turn forced us to use $(-\exp)$ (not $\exp$) in \eqref{e3.2.5} so that $\hat{x}_1$ acts through multiplication by $-x_1(z)$, resulting in the action of $\hat{\vf}=-\vf(-\hat{x}_1,-\hat{x}_2)$ through multiplication by $-\vf(x_1(z),x_2(\tz))$.  The upshot is that the signs \emph{in the symbol}\footnote{along with those in \eqref{e3.1.11} arising from Weyl quantization and the CBH formula.} are ultimately responsible for the presence of the B-field.
\end{rem}

\begin{rem}\label{rKS}
A result of Kashaev and Sergeev \cite[Theorem 7]{KS}, while expressed in very different terms, can be shown to be equivalent the special case $\vf=x_1+x_1^{-1}+x_2+x_2^{-1}$ of Theorem \ref{t1}.  (The conditions in [loc.~cit.] on a pair $(\lambda,\varepsilon)\in \CC\times \RR_{>4}$ they require for their construction of eigenfunctions of $\hat{\vf}$ amount to taking $\nu(\varepsilon)\in \ZZ$ and $\lambda=-\tfrac{\ay \varepsilon}{8\pi^2}\tfrac{\mathfrak{R}_{\gamma}(\varepsilon)}{\omega_{\gamma}(\varepsilon)}$.)  However, they do not relate their result to the relevant conjecture of \cite{GHM} or prove a partial converse as in Theorem \ref{t1a} below.
\end{rem}

Without stating any results formally, we want to briefly address the higher genus hyperelliptic case, where $F_1=\vf$ still takes the form in \eqref{e3.2.1}-\eqref{e3.2.2} but $\Delta$ is no longer reflexive.  (Note that $\vf_0$ will have $a_2,\ldots,a_g$ as coefficients.)  One easily checks that the construction of $\psi$ and the proof of Theorem \ref{t1} still go through after modifying $\chi(\tz)$, provided we impose a stronger quantization condition than that in RHS\eqref{e2.3.13}.  Namely, referring to \eqref{e2.3.8}, suppose that 
\begin{equation}\label{e3.2.19}
\textit{the normal function vector $\uv(\ua)$ belongs to $(\mathbb{I}_g\mid \Omega)\ZZ^g$.}
\end{equation}  
Then replacing the expression in braces in \eqref{e3.2.8} by
\begin{equation}\label{e3.2.20}
\textstyle \int_{\pz}z\tfrac{dx_2(\tz)}{x_2(\tz)}-\sum_{j=1}^g \mathfrak{R}_{\gamma_j} \int_{\pz}\P^*\omega_j
\end{equation}
for appropriate determinations of $\mathfrak{R}_{\gamma_j}$, the obvious generalization of \eqref{e3.2.10} goes through, ensuring that the generalized $\chi(\tz)$ is well-defined.  Under an additional assumption like \eqref{e2.3.1}, and changing the signs in $\hat{\vf}$ of those $a_j$'s attached to even powers of $\hat{x}_1$, one finds as before that $\hat{\vf}\psi=a_1\psi$.

The criterion \eqref{e3.2.19}, which we expect corresponds to the \emph{exact NS quantization conditions} of \cite{SWH}, will only hold at countably many points in moduli.  On the other hand, Conjecture \ref{c2} predicts the \emph{existence} of eigenfunctions for $\ua$ in a codimension-1 subset of moduli.  So it stands to reason that there should be something special about the eigenfunctions $\psi$, which we can only construct for $\ua$ in the smaller locus.  In the genus-2 example worked out explicitly in \cite[\S4.3]{Za}, whose ``fully on-shell'' quantization conditions (cf. [loc. cit., (4.45)]) should agree with \eqref{e3.2.19}, Zakany highlights the \emph{enhanced decay} of his explicit eigenfunctions.  Indeed, in our construction, for $g>1$ the discriminant $\mathscr{D}$ will involve higher powers of both $x_1$ and $x_1^{-1}$ than for $g=1$, which leads to decay better than $e^{-|r/2|}$ at infinity for $\psi(r)$; this perhaps begins to explain the discrepancy.

\subsection{Remarks on the spectrum of $\hat{\vf}$}\label{S3c}

Notably absent from the last section is any discussion of the ``converse question'', as to whether every eigenfunction of $\hat{\vf}$ arises from the construction described there.  We will prove a fairly strong result in this direction, to the effect that ``almost every'' eigenvalue $\lambda$ satisfies $V(\lambda)\in \Lambda(\lambda)$.  As already mentioned in Remark \ref{r3a2},\footnote{The point is that the proof of \cite[Prop. 3.4]{LST} trivially generalizes to all $\vf$ we consider here, because $\Delta$ always contains a reflexive triangle (or square).  The proof of Theorem \ref{t1a} involves, in contrast, a rather nontrivial generalization of [op. cit., \S3.2].} the spectrum $\sigma(\hat{\vf})$ is a countable subset of $[c,\infty)$ for some $c>0$, whose elements can be arranged in an increasing sequence $\{\lambda_j\}_{j\geq 1}$ with $\lambda_j\to 0$.  We may replace $\hat{\vf}$ by its self-adjoint Friedrichs extension to $L^2(\RR)$ without affecting these statements, cf. \cite{LST}.

Suppose $\textbf{P}$ is a proposition (that can be true or false) about elements of $\sigma(\hat{\vf})$.  Write $N(\lambda):=|\{j\in \NN\mid \lambda_j\leq \lambda\}|$ and $$N_{\textbf{P}}(\lambda):=|\{j\in \NN\mid \lambda_j\leq \lambda\text{ and }\textbf{P}(\lambda_j)\text{ holds}\}|.$$ We will say that $\textbf{P}$ \emph{holds asymptotically} if
\begin{equation}\label{e3.3.1}
\lim_{\lambda\to \infty}\frac{N_{\textbf{P}}(\lambda)}{N(\lambda)}=1.
\end{equation}

\begin{thm}\label{t1a}
In the setting of Theorem \ref{t1}, the \textup{``}$\subseteq$\textup{''} direction of \eqref{e3.1.12} holds asymptotically.
\end{thm}
\begin{proof}
The statement $\textbf{P}(\lambda_j)$ about eigenvalues here is, of course, that $\nu(\lambda_j)\in \ZZ$.\footnote{We can always throw out a finite set of eigenvalues less than $|\hat{a}|$, if they exist (cf. Remark \ref{r3a2}).}  From Lemma \ref{l3a1}(c), we know that $\nu(a)=\tfrac{r^{\circ}}{8\pi^2}\log^2 a+O(\log a)$, whence 
\begin{equation}\label{e3.3.2}
N(\lambda)\geq N_{\textbf{P}}(\lambda)\geq \lfloor \nu(\lambda)-\nu(|\hat{a}|)\rfloor\geq \tfrac{r^{\circ}}{8\pi^2}\log^2 \lambda +O(\log \lambda).
\end{equation}

Now given $f,g\in L^2(\RR)$, write $\langle f,g\rangle:=\int_{\RR}f(r)\overline{g(r)}dr$, and
\begin{equation}\label{e3.3.3}
\textstyle\tilde{f}(y_1,y_2):=2^{-5/4}\pi^{-3/2}\int_{\RR} e^{-\frac{1}{4\pi}\{(r-y_1)^2+2\ay y_2 r\}}f(r)\,dr
\end{equation}
for the \emph{coherent state transform} of $f$.  Adapting the calculations of \cite[\S3.1]{LST} to our setting gives
\begin{equation}\label{e3.3.4}
\textstyle\langle \hat{\vf}f,f\rangle =\iint_{\RR^2}\Phi(y_1,y_2)\,|\tilde{f}(y_1,y_2)|^2 dy_1\,dy_2
\end{equation}
where
\begin{equation}\label{e3.3.5}
\textstyle\Phi(y_1,y_2):=\sum_{\um\in \partial \Delta \cap\ZZ^2}\underset{=:\tilde{a}_{\um}}{\underbrace{a_{\um}e^{-\frac{\pi}{2}(m_1^2+m_2^2)}}}e^{m_1y_1+m_2y_2}.
\end{equation}
This implies, for instance, the semi-boundedness of $\hat{\vf}$, as $\Phi\geq c:=\min_{\uy\in\RR^2}\Phi(\uy)>0$ $\implies$ $\hat{\vf}\geq c\cdot \text{Id}$ $\implies$ $\sigma(\hat{\vf})\subset [c,\infty)$.

Let $(\cdot)_+$ be the function on $\RR$ defined by $(s)_+=s$ for $s\geq 0$ and $(s)_+=0$ for $s\leq 0$, and note that 
\begin{equation}\label{e3.3.6}
\textstyle\int_{0}^{\lambda} N(s)\,ds=\sum_{j\geq 1}(\lambda-\lambda_j)_+.
\end{equation}
Reasoning with Jensen's inequality as in [op. cit., \S2.2], we have 
\begin{equation}\label{e3.3.7}
\textstyle\sum_{j\geq 1}(\lambda-\lambda_j)_+ \leq \frac{1}{4\pi^2}\iint_{\RR^2} (\lambda-\Phi(y_1,y_2))_+\,dy_1\,dy_2.
\end{equation}
Choose $M>0$ so that $M\tilde{a}_{\um}\geq a_{\um}$ ($\forall \um\in\partial\Delta\cap \ZZ^2$).  Writing $Y_j:=e^{y_j}$ and $\Gamma_L:=\{\underline{Y}\in \RR^2_+\mid L\geq \vf(Y_1,Y_2)\}$, note that the boundary $\partial\Gamma_{L}$ is the cycle $\beta$ on $E_{-L}$.  Together with Lemma \ref{l3a1}(a) and \eqref{e2.2.6}, this gives
\begin{equation}\label{e3.3.8}
\begin{split}
\text{RHS}\eqref{e3.3.7}&\textstyle\leq \frac{1}{4\pi^2 M}\iint_{\RR^2} (M\lambda -\vf(Y_1,Y_2))_+ \frac{dY_1}{Y_1}\frac{dY_2}{Y_2}\\
&\textstyle\leq \frac{\lambda}{4\pi^2}\iint_{\Gamma_{M\lambda}}\frac{dY_1}{Y_1}\frac{dY_2}{Y_2} = \frac{\lambda}{4\pi^2}R_{\beta}(-M\lambda)\\
&\textstyle=\frac{r^{\circ}}{8\pi^2}\lambda \log^2 \lambda +O(\log\lambda).
\end{split}
\end{equation}
Putting the last three equations together, we get
\begin{equation}\label{e3.3.9}
\tfrac{r^{\circ}}{8\pi^2}\log^2\lambda +O(\log \lambda) \geq N(\lambda),
\end{equation}
which combined with \eqref{e3.3.2} gives the result.
\end{proof}

The constraints imposed on the zero locus of $\bm{\rho}\circ\bm{\nu}$ by its interpretation as eigenvalues of $\hat{\vf}$ (Theorem \ref{t1}), and vice versa (Theorem \ref{t1a}), seem worth exploring further.  For instance, per Remark \ref{r3a2}, we expect (and know in some cases) that $c>|\hat{a}|$; together with the following Lemma, this essentially rules out points $a\in U$ at which $V(a)\in \Lambda(a)$ (the exact quantization condition) and $\R(a)$ is torsion (the perturbative quantization condition proposed in \cite{GS}).

\begin{lem}\label{l3c1}
For $a\in (|\hat{a}|,\infty)$, $\R(a)\in H_1(E_a,\CC/\ZZ(2))$ is a nontorsion class.
\end{lem}
\begin{proof}
From the known integrality of local instanton numbers of toric CY 3-folds \cite{Ko}, it follows that $\text{LHS}\eqref{e3.1.7}\geq 1$, hence that $\Re(t(\hat{a}))\geq 0$.  From \eqref{e3.1.3} (and positivity of coefficients of $\vf$, and negativity of $\hat{a}$), it is immediate that $t(|\hat{a}|)>\Re(t(\hat{a}))$, hence $t(a)\in \RR_+$ for $a\in (|\hat{a}|,\infty)$.  But if $\R(a)$ is torsion, then $R_{\gamma}(a)\in \QQ(2)$ $\implies$ $t(a)\in \QQ(1)\subset \ay\RR$.
\end{proof}

More striking is a conditional transcendence result on the eigenvalues that arises from their asymptotic Hodge-theoretic interpretation in Theorem \ref{t1a}.  A mixed version of the Grothendieck period conjecture (which we will simply call the GPC) says that the transcendence degree of a period point arising from a motive defined over $\bar{\QQ}$ is equal to the dimension of the minimal mixed Mumford-Tate domain containing it.  The (mixed) motive in question is the $K_2$-cycle $\{-x_1,-x_2\}$ on $E_a$, with MHS the extension of $\ZZ(0)$ by $H^1(E_a,\ZZ(2))$ given by $\tfrac{1}{(2\pi\ay)^2}\R$.  The possibillities for the M-T group are an extension of $\mathrm{SL}_2$ or a $1$-torus (depending on whether $E_a$ is CM) by $\mathbb{G}_a^{\times 2}$ or $\{1\}$ (depending on whether $\R$ is torsion); the corresponding domain is $\mathfrak{H}$, a CM point in it, or the product of either one with $\CC^2$.  The coordinates of the period point are $\Omega(a)$ (in $\mathfrak{H}$) and $(\tfrac{R_{\gamma}(a)}{(2\pi\ay)^2},\tfrac{R_{\beta}(a)}{(2\pi\ay)^2})$ (in $\CC^2$).\footnote{We have to divide by $(2\pi\ay)^2$, of course, because a torsion class must have coordinates in $\QQ$, not transcendental ones in $\QQ(2)$.}

\begin{conj}[GPC]
If $a\in \bar{\QQ}$ and $\R(a)$ is nontorsion, then the transcendence degree of $\bar{\QQ}(\Omega(a),\tfrac{R_{\gamma}(a)}{(2\pi\ay)^2},\tfrac{R_{\beta}(a)}{(2\pi\ay)^2})/\bar{\QQ}(\Omega(a))$ is $2$.
\end{conj}

\begin{prop}\label{p3c1}
Assuming the GPC, asymptotically $\sigma(\hat{\vf})$ consists of transcendental numbers.
\end{prop}

\begin{proof}
Let $\lambda\in \sigma(\hat{\vf})$ be an eigenvalue for which $\nu(\lambda)\in \ZZ$.  (We may assume $\lambda\in (|\hat{a}|,\infty)$.)  That is, we have an algebraic relation $\tfrac{1}{4\pi^2}(R_\gamma(\lambda)\Omega(\lambda_i)-R_{\beta}(\lambda))=n$ on $\tfrac{R_{\gamma}(\lambda)}{(2\pi\ay)^2}$ and $\tfrac{R_{\beta}(\lambda)}{(2\pi\ay)^2}$ over $\bar{\QQ}(\Omega(\lambda))$.  By the GPC, either $\lambda\notin\bar{\QQ}$ or $\R(\lambda)$ is torsion.  But the latter possibility is ruled out by Lemma \ref{l3c1}, and so we are done by Theorem \ref{t1a}.
\end{proof}

We conclude with somthing of a curiosity:  in case $\vf=x_1+x_1^{-1}+x_2+x_2^{-1}+x_1x_2^{-1}+x_1^{-1}x_2$, our normal function is closely related to the Feynman integral $\mathcal{I}$ associated to the sunset graph with equal masses \cite{BKV}.  This is written in [op. cit.] as a function of $s=\tfrac{1}{3-a}=$ the inverse norm of the external momentum, but written as a function of $a$ we have $\mathcal{I}(a)=\tfrac{(2\pi\ay)^2}{a}V(a)$ (see [op. cit., (7.17)]).  The condition that $V(a)\in \Lambda(a)$ means that $V$, or equivalently $\mathcal{I}$, belongs to its own lattice of ambiguities under monodromy.  As we have seen, the values of $a$ at which this happens correspond to eigenvalues of $\hat{\vf}$.  One wonders if there is any deeper physical relation here between Feynman amplitudes and quantum curves.

\section{Regulator periods at the maximal conifold point}\label{S4}

In this section we prove Conjecture \ref{c3} in the cases $(m,n)=(g,g)$ and $(2g-1,1)$, for every $g\geq 1$.  A proof for $(m,n)=(2g,1)$ will appear in a forthcoming work by the third author.  

Because we have to enumerate multiple nodes on the maximal conifold curve, it is better in this section to replace $(x_1,x_2)$ as toric coordinates by $(x,y)$, which we do throughout.  We also denote the zero-locus of a polynomial by $\mathbf{Z}(\cdot)$.

\subsection{The main result and some preliminaries}\label{S4a}

Consider the families  of genus-$g$ curves cut out of $(\CC^*)^2$ by the (integrally tempered) polynomials $F_{g,g}(x,y)$ and $F_{2g-1,1}(x,y)$ from \eqref{e2.4.7a}.  In contrast to \S\ref{S2}, $\mathcal{C}_{g,g}$ and $\mathcal{C}_{2g-1,1}$ will denote their \emph{compactifications} in $\PP_{\Delta}$.  There are no mass parameters in either case, so $r=3$ and the equations take the simpler form \eqref{e2.4.7}.  Moreover, $\mathcal{C}_{g,g}$ is torically equivalent to $\mathcal{C}_{2g-1,1}$ via the map $u=x^{-1}y^{-1},\,v=x^{g}y^{g-1}$. The effect of this map is straightforward: for $n=1,\dots,g$ it simply shifts $n\mapsto g-n+1$ on the level of indices; that is, if $F_{g,g}(x,y)$ is written with parameters $a_n$, then the image (under the above map) is precisely $F_{2g-1,g}(u,v)$ with parameters $a_{g-n+1}$. The upshot of this connection is that statements concerning regulator periods of $\mathcal{C}_{2g-1,1}$ can be pulled back to those corresponding to $\mathcal{C}_{g,g}$, provided we choose the correct cycles.  For our purposes here, the important case is that the cycle $\gamma_{g-n+1}$ of $\mathcal{C}_{2g-1,1}$ giving rise to $R_{\gamma_{g-n+1}}\sim -2\pi \mathbf{i}\log(a_{g-n+1})$ pulls back to the cycle $\gamma_{n}$ of $\mathcal{C}_{g,g}$ corresponding to $R_{\gamma_n}\sim -2\pi\mathbf{i}\log (a_n)$.
\begin{thm}\label{thm_01}
Conjecture \ref{c3} holds for the families $\mathcal{C}_{g,g}$ and $\mathcal{C}_{2g-1,1}$; that is,
\begin{align}
\tfrac{1}{2\pi\mathbf{i}}R_{\gamma_{1}}(\underline{\hat{a}})\underset{\mathbb{Q}(1)}{\equiv}&\mathcal{D}_{g,g}\label{0.3}~\text{and}\\
 \tfrac{1}{2\pi\mathbf{i}}R_{\gamma_{g}}(\underline{\hat{a}})\underset{\mathbb{Q}(1)}{\equiv}&\mathcal{D}_{2g-1,g}.\label{0.4}
\end{align}
\end{thm}

\begin{rem}\label{r4.2}
The predictions of \cite{CGM} aligning with Conjecture \ref{c3} are written in terms of the complex structure/GKZ parameters $z_i:=z_i(\underline{a})$.   (In the $(g,g)$ cases these are given by $z_1=\tfrac{a_2}{a_1^3}$, $z_2=\tfrac{a_1a_3}{a^2_2}$, $\ldots$ , $z_{g-1}=\tfrac{a_{g-2}a_g}{a_{g-1}^2}$, $z_g=\tfrac{a_{g-1}}{a_g^2}$.) Translated into statements about the corresponding regulator periods (cf. \eqref{e2.3.4}), these essentially amount to\footnote{Here $[C^{-1}]$ is the inverse of the first $g\times g$ minor of the intersection matrix $[C]$.  The $R_{\alpha_i}$ ``correspond'' to $z_i$ in the sense of being asymptotic to $2\pi\ay\log(z_i)$.}
\begin{equation}
\textstyle \tfrac{1}{2\pi\ay} \sum_{i=1}^g[C^{-1}]_{1j} R_{\alpha_i}(\hat{\uz})\underset{\mathbb{Q}(1)}{\equiv}\mathcal{D}_{m,n},
\end{equation}
which of course is equivalent to \eqref{e2.4.9}.  
While $z_i$ and $R_{\alpha_i}$ are more natural from the standpoint of GKZ systems, the $\{a_j\}$ and the corresponding regulator periods $R_{\gamma_j}$ simplify the statement of the result, and are more natural to compute directly (cf. Appendix \ref{appA}).  As we will see, the $\{\gamma_j\}$ are also the cycles which limit to loops passing through individual nodes at the maximal conifold point $\hat{\ua}$.
\end{rem}

\begin{rem}
As $R\{-x,-y\}\equiv R\{x,y\}$ mod $\mathbb{Q}(2)$ we may work with the latter. Note also that \eqref{e2.4.9} is stated in terms of the regulator period asymptotic to $-2\pi\mathbf{i}\log(a_n)$; it is convenient in this section to drop the negative sign and work with one asymptotic to $2\pi\mathbf{i}\log(a_n)$. Thus from now on $$R_{{\gamma}_n}\sim 2\pi \mathbf{i}\log(a_n).$$ Furthermore, since we intend to investigate different components of the discriminant locus throughout this section, it will be important to track the moduli; so henceforth we will rename $F_{g,g}$ and $F_{2g-1,1}$ to $F^{\underline{a}}_{g,g}$ and $F^{\underline{a}}_{2g-1,1}$.
\end{rem}

Let us outline a proof of Theorem \ref{thm_01}. Denote by $\hat{\mathcal{C}}_{g,g}$ the fiber of the family over the \textit{maximal conifold} point $\underline{\hat{a}}$. It has $g$ nodes $\{\hat{p}_j\}$, and the cycles $\{\hat{\gamma}_j\}_{j=1}^g$ passing through each node generate $H_1(\hat{\C}_{g,g})$; we set $R_{\hat{\gamma}_j}:=\int_{\hat{\gamma}_i}R\{x,y\}$.  Writing $\bm{\kappa}=\,_{\underline{\hat{\gamma}}}[\mathrm{Id}]_{\underline{\gamma}(\hat{a})}$ for the change-of-basis matrix, we have

\begin{prop}\label{thm_02}
Let $\kappa_j:=\mathrm{gcd}(2j-1,2g+1)$. Then
\begin{equation}
    \bm{\kappa}=\mathrm{diag}(\kappa_1,\dots,\kappa_g).
\end{equation}
It then follows from temperedness that
\begin{equation}
    \tfrac{1}{2\pi\mathbf{i}}R_{\gamma_{j}}(\underline{\hat{a}})\underset{\mathbb{Q}(1)}{\equiv}\tfrac{\kappa_n}{2\pi \mathbf{i}}R_{\hat{\gamma}_{j}}.
\end{equation}
\end{prop}

\noindent In \S\ref{S4b} we detect monodromies via power series representing classical periods, verifying Proposition \ref{thm_02} in the process. In \S\ref{S4c} we use a key technique developed in \cite[\S6]{DK} that allows us to connect conifold limits of regulator periods to special values of the Bloch-Wigner function; this method coupled with Proposition \ref{thm_02} settles Theorem \ref{thm_01}.  As a consequence \emph{$g$-many} series identities are borne out in \S\ref{S4d} --- not just the two required for the Theorem.

${}$

We conclude this subsection with two preliminary results.  The first will help us to control certain power series asymptotics, and the second gives us information on nodal fibers of $\mathcal{C}_{g,g}$.

\begin{lem}\label{lemma_01}
If $a,b,c\in\mathbb{R}_{\gg0}$ are such that $a=2b+c$, then 
\begin{equation}
    \dfrac{\Gamma(1+a)}{\Gamma^2(1+b)\Gamma(1+c)}\sim \dfrac{1}{2\pi b}\sqrt{\dfrac{a}{c}}\Bigg(\dfrac{a}{c}\bigg(\dfrac{c}{b}\bigg)^{2b/a}\Bigg)^a.
\end{equation}
\end{lem}
 
\begin{proof}
Stirling's approximation yields
\begin{align*}
\dfrac{\Gamma(1+a)}{\Gamma^2(1+b)\Gamma(1+c)}\sim \dfrac{1}{2\pi b}&\sqrt{\dfrac{a}{c}}\dfrac{a^a}{b^{2b}c^c}e^{-a+2b+c}=\dfrac{1}{2\pi b}\sqrt{\dfrac{a}{c}}\dfrac{a^a}{b^{2b}c^{a-2b}}\nonumber\\
&=\dfrac{1}{2\pi b}\sqrt{\dfrac{a}{c}}\dfrac{a^a}{c^a}\dfrac{c^{2b}}{b^{2b}}=\dfrac{1}{2\pi b}\sqrt{\dfrac{a}{c}}\Bigg(\dfrac{a}{c}\bigg(\dfrac{c}{b}\bigg)^{2b/a}\Bigg)^a\nonumber
\end{align*}
for $b,c\to \infty$ (and $a=2b+c$).
\end{proof}

\begin{lem}\label{Proposition_0.2}
Suppose that the fiber over $\underline{\tilde{a}}=(\tilde{a}_1,\dots,\tilde{a}_g)$ has $g$-many singularities, say $\tilde{p}_j:=(\tilde{x}_j,\tilde{y}_j), n=1,\dots,g$. Then for each $j$, $\tilde{p}_j$ is a node, and $\tilde{x}_j=\tilde{y}_j$.
\end{lem}

\begin{proof}
Since $x\partial_xF^{\underline{a}}_{g,g}(x,y)-y\partial_yF^{\underline{a}}_{g,g}(x,y)=x-y$, any singularity must have symmetric co-ordinates; that is, $\tilde{x}_j=\tilde{y}_j$.  By toric equivalence we may replace $F_{g,g}^{\tilde{\ua}}(x,y)$ by 
\begin{equation} \label{eS1}
\textstyle F_{2g-1,g}^{\tilde{\ua}}(u,v)=u+v+\sum_{\ell=1}^g \tilde{a}_{\ell}u^{-\ell+1}+u^{-2g+1}v^{-1}
\end{equation}
 (reversing the order of the $\{a_{\ell}\}$); by abuse of notation we continue to label the singularities of $F_{2g-1,1}^{\tilde{\ua}}$ by $\tilde{p}_j$, but with coordinates $(\tilde{u}_j,\tilde{v}_j)$ satisfying $\tilde{u}^{-2g+1}_j=\tilde{v}_j^2$.  Since the edge polynomials of \eqref{eS1} are all $w+1$, the curve intersects each component of the toric boundary with multiplicity 1, and so all $\tilde{p}_j\in \CC^*\times\CC^*$.  Moreover, \eqref{eS1} is irreducible since it is quadratic in $v$, with discriminant $\mathscr{D}(u)$ of odd degree.  As a consequence, the vanishing cycle sequence associated to the smoothing $F^{\tilde{\ua}}_{2g-1,1}+s$ takes the form
\begin{equation}\label{eS2}
0\to H^1(\C^{\tilde{\ua}}_{2g-1,1})\to H^1_{\lim}\to H^1_{\mathrm{van}}\to 0.
\end{equation}
Since $\mathrm{rk}(F^1 H^1_{\lim})=g$ and the $g$ singularities each contribute nontrivially to $\mathrm{rk}(F^1 H^1_{\mathrm{van}})$, each contribution must be exactly $1$.  So the $\tilde{p}_j$ are either nodes or cusps, and to show they are nodes it will suffice to show that the Hessians $H_{F^{\underline{\tilde{a}}}_{2g-1,1}}$ is non-degenerate at $\tilde{p}_j$.

To do this, define
\begin{equation}
 \textstyle   \tilde{P}(u):=2g+1+\sum_{j=1}^{g}(2g+1-2j)\tilde{a}_{j}u^{-j},
\end{equation}
and observe that
\begin{equation}
  \textstyle \tilde{P}(\tilde{u}_j)=\tfrac{2g-1}{\tilde{u}_j}F^{\underline{\tilde{a}}}_{2g-1,1}(\tilde{p}_j)+2\partial_u F^{\underline{\tilde{a}}}_{2g-1,1}(\tilde{p}_j)=0.
\end{equation}
Thus $\textbf{Z}(\tilde{P})=\{\tilde{u}_1,\ldots,\tilde{u}_{g}\}$. It follows that $\tilde{P}$ has no repeated roots; that is, $\tilde{P}'(\tilde{u}_j)\neq 0$ ($\forall j$).  To compute the Hessians, write
\begin{align}
\textstyle\partial_{uu}F^{\underline{\tilde{a}}}_{2g-1,1}(\tilde{p}_j)&=\textstyle\sum_{\ell=1}^{g}\ell(\ell-1)\tilde{a}_{\ell}\tilde{u}_j^{-\ell-1}+2g(2g-1)\tilde{u}_j^{-2g-1}\tilde{v}^{-1}_{n}\nonumber\\
&\textstyle=\sum_{\ell=1}^{g}\ell(\ell-1)\tilde{a}_{\ell}\tilde{u}_{j}^{-\ell-1}+\tfrac{2g(2g-1)\tilde{y}_j}{\tilde{u}_j^2},\\
\partial_{uv}F^{\underline{\tilde{a}}}_{2g-1,1}(\tilde{p}_j)&=(2g-1)\tilde{u}_j^{-2g}\tilde{v}_j^{-2}=\tfrac{2g-1}{\tilde{v}_j},\;\text{and}\\
\partial_{vv}F^{\underline{\tilde{a}}}_{2g-1,1}(\tilde{p}_j)&=2\tilde{u}_j^{2g-1}\tilde{v}_j^{-3}=\tfrac{2}{\tilde{v}_j}.
\end{align}
At this point a few simplifications can be made. Differentiating the defining equation of $\tilde{P}$ and plugging in $u=\tilde{u}_j$, we obtain,
\begin{equation}
   \textstyle \tilde{P}'(\tilde{u}_j)=2\sum_{\ell=1}^{g}\ell(\ell-1)\tilde{a}_{\ell}\tilde{u}_j^{-\ell-1}-\sum_{\ell=1}^{g}(2g-1)\ell\tilde{a}_{\ell}\tilde{u}_j^{-\ell-1}
\end{equation}
On the other hand $\partial_u (F^{\underline{\tilde{a}}}_{2g-1,1}(u,v)/u)$ vanishes at $\tilde{p}_j$, which yields 
\begin{align}
\textstyle-\tfrac{\tilde{v}_j}{\tilde{u}_j^2}-\sum_{\ell=1}^{g}\ell\tilde{a}_{\ell}\tilde{u}_j^{-\ell-1}-2g \tilde{u}_j^{-2g-1}\tilde{v}_j^{-1}&=0\nonumber\\
\implies\textstyle\sum_{\ell=1}^{g}(2g-1)j\tilde{a}_{\ell}\tilde{u}_j^{-\ell-1}&=-\tfrac{(2g-1)(2g+1)\tilde{v}_j}{\tilde{u}_j^2}
\end{align}
Combining everything, we arrive at
\begin{equation}
     \partial_{uu}F^{\underline{\tilde{a}}}_{2g-1,1}(\tilde{p}_j)=\tfrac{(2g-1)^2\tilde{v}_j}{2\tilde{u}_j^2}+\tfrac{\tilde{P}'(\tilde{u}_j)}{2}
\end{equation}
Therefore,
\begin{align}
H_{F^{\underline{\tilde{a}}}_{2g-1,1}}(\tilde{p}_j)&=\left(\partial_{uv}F^{\underline{\tilde{a}}}_{2g-1,1}(\tilde{p}_j)\right)^2-\partial_{uu}F^{\underline{\tilde{a}}}_{2g-1,1}(\tilde{p}_j)\partial_{vv}F^{\underline{\tilde{a}}}_{2g-1,1}(\tilde{p}_j)\nonumber\\
  &\textstyle= \tfrac{(2g-1)^2}{\tilde{u}_j^2}-\tfrac{(2g-1)^2}{\tilde{u}_j^2}-\tfrac{\tilde{P}'(\tilde{u}_j)}{\tilde{v}_j}\; =\;-\tfrac{\tilde{P}'(\tilde{u}_j)}{\tilde{v}_j}\neq 0\nonumber
\end{align}
as was to be shown.
\end{proof}

\subsection{Monodromy calculations via power series}\label{S4b}

Consider a $1$-parameter family of curves $\C\to \PP^1$ with coordinate $t$, endowed with a section $\omega$ of the relative dualizing sheaf; on smooth fibers $\C_t$, $\omega_1$ is a holomorphic $1$-form.  Assume that $\C_c$ has a single node $p_c$ (i.e. is a ``conifold fiber''), and let $\delta_0$ be the ``conifold'' vanishing cycle pinched at $p_c$.  Writing $\ve_0$ for a cycle invariant about $t=0$, its monodromy about $t=c$ is a multiple of $\delta_0$, say $\bm{k}\delta_0$ for some $\bm{k}\in\mathbb{Z}_{\geq 0}$.  We would like to compute this \emph{conifold multiple} $\bm{k}$.

Writing $\epsilon_0(t)=\sum_{m\geq 0}b_m t^m:=\int_{\varepsilon_0}\omega_t$, we have
\begin{equation}
    \int_{\bm{k}\delta_0}\omega_t=(T_c-I)\epsilon_0=2\pi\mathbf{i}C_0+O(t-c)
\end{equation}
for some $C_0\in \CC$.  Observe that 
\begin{equation}
\int_{\bm{k}\delta_0}\omega_c=\bm{k}\int_{\delta_0}\omega_c=\bm{k}\cdot 2\pi\mathbf{i}\cdot \underset{p_c}{\mathrm{Res}}~{\omega_c}  \implies C_0=\bm{k}\cdot \underset{p_c}{\mathrm{Res}}~{\omega_c}.\label{0.5}
\end{equation}
On the other hand, \cite[Lemma 6.4]{Ke2} (with $B(t)=\epsilon_0(t)$, $\lambda=2\pi\mathbf{i}C_0$, and $w=1$) yields
\begin{equation}\label{0.6}
    b_m \sim \dfrac{C_0}{c^m\cdot m}.
\end{equation}
provided $C_0\neq 0$.\footnote{Otherwise, $B_m$ has a smaller exponential growth-rate and RHS\eqref{0.7} is zero, which confirms the Lemma when $C_0=0$ as well.} Therefore we have proven 

\begin{lem}\label{thm_04}
The conifold multiple is computed by
\begin{equation}
\bm{k}=\dfrac{\underset{m\to\infty}{\lim} b_m \cdot c^m\cdot m}{\mathrm{Res}_{p_c}~\omega_c}.\label{0.7}    
\end{equation}
\end{lem}

\begin{example}
Consider the Legendre family, $y^2=x(x-1)(x-t)$. Setting $c=1$ gives rise to a node at $(1,0)$.  Taking $\omega_t=\tfrac{dx}{y}$,  we have
\begin{equation}
  {\mathrm{Res}}_{(1,0)}~\omega_c=   {\mathrm{Res}}_{x=1}\tfrac{dx}{(x-1)\sqrt{x}}=1.
\end{equation}
Moreover $b_m=2\pi\binom{-1/2}{m}^2$, hence (\ref{0.7}) implies
\begin{equation}
  \textstyle  \bm{k}=\underset{m\to\infty}{\lim}2\pi m\binom{-1/2}{m}^2=2.
\end{equation}
\end{example}

\begin{example}
Now consider the family $\C_t$ defined by $f_t(x,y)=xy-t^{1/3}(x^3+y^3+1)$.  In this case $c=\tfrac{1}{3^3}$ and $b_m=\tfrac{(3m)!}{m!^3}$, but $\C_c=\mathbf{Z}(\prod_{\ell=1}^3(1+\zeta_3^i x+\zeta_3^{2i}y))$ is a N\'eron 3-gon with \emph{three} nodes $p_i$.  But since $\ve_0(c)$ will pass through each $p_i$ the same number $\bm{k_0}$ of times, and $\omega_c$ must have the same residue at each, \eqref{0.7} holds (taking say $p_c=p_1:=(1,1)$) provided we interpret $\bm{k}$ as $3\bm{k_0}$.  For the residue of 
\begin{equation}
   2\pi\ay \omega_c=\mathrm{Res}_{\C_c}\frac{dx\wedge dy}{f_c}=\frac{dx}{\partial_y f_c}=\frac{dx}{x-y^2}
\end{equation}
at $p_1$, we can restrict to the component $X_c:=\mathbf{Z}(1+\zeta_3x+\zeta_3^2y)$:
\begin{align}
 {\mathrm{Res}}_{p_1} \omega_c&=\dfrac{1}{2\pi\mathbf{i}}  {\mathrm{Res}}_{(1,1)} \left(\left.\frac{dx}{x-y^2}\right|_{X_c}\right) = \frac{1}{2\pi\ay}{\mathrm{Res}}_{y=1} \left(\frac{\zeta_3dy}{y^2+\zeta_3 y+{\zeta}^2_3} \right) \nonumber\\
&=\dfrac{1}{2\pi\mathbf{i}}\dfrac{\zeta_3}{1-\zeta_3^2} =\dfrac{1}{2\pi \sqrt{3}}.
\end{align}
\noindent Since $b_m=\tfrac{(3m)!}{m!^3}$ we get
\begin{equation}
    \bm{k}=\underset{m\to\infty}{\lim}\dfrac{1}{3^{3m}}\cdot m\cdot \dfrac{(3m)!}{m!^3}\cdot2\pi \sqrt{3}=3,
\end{equation}
which means that $\ve_0(c)$ winds once around the N\'eron 3-gon.
\end{example}

For the proof of Proposition \ref{thm_02}, we need to compute the Picard-Lefschetz matrix $\bm{\kappa}$, whose entries $\bm{\kappa}_{ij}$ tell how many times the specialization $\gamma_i(\hat{\ua})$ passes through $\hat{p}_j$.  In order to invoke Lemma \ref{thm_04} for this purpose, we should reinterpret these numbers as (roughly speaking) conifold multiples for 1-parameter subfamilies of $\C_{\ua}$ acquiring a \emph{single} node.  The idea is that $\hat{\ua}$ is a normal-crossing point of the discriminant locus, whose $g$ local-analytic irreducible components each parametrize fibers carrying a single node $p_j$.  These are labeled in such a way that the $j^{\text{th}}$ component can be followed out to where it meets the $a_j$-axis at $a_j=\mathring{a}_j$.  Call this fiber $\C_{g,g}^{\mathring{\ua}_j}$, and $\mathring{p}_j=(\mathring{x}_j,\mathring{x}_j)$ for the limit of the node to it.

From Appendix \ref{appA} we have the 1-forms
\begin{equation}
\varpi_j=\tfrac{1}{2\pi\ay}\nabla_{\delta_{a_j}}R\{x,y\}=\frac{-a_j}{2\pi\ay}\mathrm{Res}_{\C_{g,g}}\left(\frac{dx\wedge dy}{x^j y^j F_{g,g}(x,y)}\right)
\end{equation}
and 1-cycles $\gamma_j$ ($j=1,\ldots,g$).  The computation that follows will consider periods $\Pi_{jj}=\int_{\gamma_j}\varpi_j$ on the 1-parameter families over the $a_j$-axes (acquiring a single node at $a_j =\mathring{a}_j$), which will suffice to determine the diagonal terms $\bm{\kappa}_{jj}$.  That the remaining, off-diagonal terms are actually zero follows from the fact (cf. Appendix \ref{appA}) that each $\gamma_j$ is well-defined on a tubular neighborhood of the hyperplane in (compactified) moduli defined by $z_j=0$, which is cut by the conifold components carrying $p_i$ for every $i\neq j$.

Now $\C_{g,g}^{\mathring{\ua}_j}$ is defined by
\begin{equation}
     f^{(j)}_{g,g}:=F^{\underline{\mathring{a}}_j}_{g,g}(x,y)=x+y+\mathring{a}_jx^{1-j}y^{1-j}+x^{-g}y^{-g},
\end{equation}
and to find the node $\mathring{p}_j$ we solve
\begin{align}
\mathring{x}_{j}^{2g}f^{(j)}_{g,g}\bigg|_{x=y=\mathring{x}_{j}}&=2\mathring{x}_{j}^{2g+1}+1+\mathring{a}_{j}\mathring{x}_{j}^{2g-2j+2}=0,   \\
\mathring{x}_{j}^{2g+1}\partial_xf^{(j)}_{g,g}\bigg|_{x=y=\mathring{x}_{j}}&=\mathring{x}_{j}^{2g+1}-g-(j-1)\mathring{a}_{j}\mathring{x}_{j}^{2g-2j+2}=0.
\end{align}
to obtain
\begin{align}
    \mathring{x}_{j}&=\sqrt[2g+1]{\dfrac{g-j+1}{2j-1}},\\
    \mathring{a}_{j}&=-\dfrac{2g+1}{2j-1}\bigg(\dfrac{2g+1}{g-j+1}\bigg)^{\tfrac{2(g-j+1)}{2g+1}}.
\end{align}
In particular, we have the relation
\begin{equation}
    \mathring{a}_{j}\mathring{x}_{j}^{2(g-j+1)}=-\dfrac{2g+1}{2j-1}.
\end{equation}
In order to calculate the residue of $\varpi_j$ at $\mathring{p}_j$, recall that for any $f(x,y)=Ax^2+Bxy+Cy^2+\text{higher order terms}\in\mathbb{C}[x,y]$, we have
\begin{equation}
 \mathrm{Res}^2_{\uo}\frac{dx\wedge dy}{f} := {\mathrm{Res}}_{\uo}\left(\mathrm{Res}_{\bm{Z}(f)}\dfrac{dx\wedge dy}{f}\right)=\dfrac{1}{\sqrt{B^2-4AC}}.
\end{equation}
Changing variables to $X:=x-\mathring{x}_{j}$, $Y:=y-\mathring{x}_{j}$ in $f^{(j)}_{g,g}(x,y)$ leads to the equation
\small\begin{align}
x^gy^gf^{(j)}_{g,g}=&\tfrac{\mathring{x}_{j}^{2g-1}(2g^2+2g+1-(g-j+1)(2g+1))}{2}X^2+\mathring{x}_{j}^{2g-1}({\scriptstyle{2g^2+2g-(g-j+1)(2g+1)}})XY\nonumber\\
&+\tfrac{\mathring{x}_{j}^{2g-1}(2g^2+2g+1-(g-j+1)(2g+1))}{2}Y^2 +~ \text{higher order terms}.
\end{align}\normalsize
Therefore
\begin{align}
   {\mathrm{Res}}^2_{\mathring{p}_{j}}\dfrac{dx \wedge dy}{x^gy^gf^{(j)}_{g,g}}&=\tfrac{1}{\mathring{x}_{j}^{2g-1}\sqrt{(2g^2+2g-(g-j+1)(2g+1))^2-(2g^2+2g+1-(g-j+1)(2g+1))^2}}\nonumber\\
   &=\tfrac{1}{\mathring{x}_{j}^{2g-1}\sqrt{(2g-2g-1)(4g^2+4g+1-2(g-j+1)(2g+1))}}\nonumber\\
   &=\tfrac{1}{\mathring{x}_{j}^{2g-1}\sqrt{-(2g+1)(2g+1-2g+2j-2)}}\\
   &=\tfrac{\ay}{\mathring{x}_{j}^{2g-1}\sqrt{(2g+1)(2j-1)}}\nonumber.
\end{align}
Consequently the residue of $\varpi_{j}$ may now be found:
\begin{align}
\mathrm{Res}_{\mathring{p}_{j}}\varpi_{j}&=\dfrac{-\mathring{a}_{j}}{2\pi \ay} {\mathrm{Res}}^2_{\mathring{p}_{j}}\dfrac{dx \wedge dy}{x^jy^jf^{(j)}_{g,g}}\nonumber\\
&=\dfrac{-\mathring{a}_{j}}{2\pi \ay}\cdot \mathring{x}_{j}^{2(g-j)}\cdot  {\mathrm{Res}}^2_{\mathring{p}_{j}}\dfrac{dx \wedge dy}{x^gy^gf^{(j)}_{g,g}}\nonumber\\
&=\dfrac{-1}{2\pi}\cdot (\mathring{a}_{j}\mathring{x}_{j}^{2(g-j+1)})\cdot\dfrac{1}{\mathring{x}_{j}^{2g+1}\sqrt{(2g+1)(2j-1)}}\\
&=\dfrac{\sqrt{2g+1}}{2\pi(g-j+1)\sqrt{(2j-1)}}.\nonumber
\end{align}

For the periods of $\varpi_j$, we start as in Appendix \ref{appA} with those of the regulator class.  Writing $\vf_j:=x^{j-1}y^{j-1}F_{g,g}^{\ua}(x,y)-a_j$, \eqref{eA3} (with the sign flip from our choice of $\gamma_j$) yields
\begin{align}\label{1.3}
  \dfrac{1}{2\pi \mathbf{i}}R_{\gamma_j}(\underline{a})&\underset{\mathbb{Q}(1)}{\equiv}\log (a_{j})-\sum_{m>0}\dfrac{(-a_{j})^{-m}}{m} [\varphi_j^m]_{\uo} \nonumber\\ 
          &=\log (a_{j})-\sum_{m>0}\dfrac{(-a_{j})^{-m}}{m} \times \\ &\textstyle\bm{[(}\underbrace{x^{j}y^{j-1}}_{=:A_j}+\underbrace{x^{j-1}y^j}_{=:B_j}+\textstyle\sum_{\substack{k=1 \\ k\neq j}}^{g}a_{k}\underbrace{x^{j-k}y^{j-k}}_{=:C^{k}_j}+\underbrace{x^{j-g-1}y^{j-g-1}}_{=:D_j}\bm{)}^m\bm{]}_{\uo}\nonumber
\end{align}
where $[L]_{\uo}$ stands for the constant term (in $x,y$) appearing in the Laurent polynomial $L$. Now, given $l_1,l_2,\cdots,l_g\in \mathbb{Z}$, we define
\begin{align}
    \mathfrak{l}_{j}&:=\dfrac{1}{2j-1}\bigg((2g+1)l_{j}+\sum\limits_{\substack{k=1 \\ k\neq j}}^{g}(2k-1)l_{k}\bigg)\\
    \mathfrak{l'}_{j}&:=\dfrac{1}{2j-1}\bigg((g-j+1)l_{j}+\sum\limits_{\substack{k=1 \\ k\neq j}}^{g}(k-j)l_{k}\bigg),\text{ and put}\\
    \mathcal{L}_{j}&:=\{(l_1,l_2,\cdots,l_g)\in \mathbb{Z}^g_{\geq 0}\mid\mathfrak{l}'_{j}\in \mathbb{Z}_{\geq 0}\}\setminus\{(0,\cdots,0)\}
\end{align}
Note that $\mathfrak{l}'_{j}\in\mathbb{Z}_{\geq 0}\implies \mathfrak{l}_{j}\in\mathbb{Z}_{\geq 0}$. The upshot of this construction is if $L_{j},L'_{j}\in\mathbb{Z}_{\geq 0}$ are such that
\begin{align}
A_{j}^{L_{j}}B_j^{L'_{j}}\prod\limits_{\substack{k=1 \\ k\neq j}}^{g-1}(C^{k}_{j})^{l_{k}}D_j^{l_{j}}&=1\text{ and}\\
L_{j}+L'_{j}+\sum\limits_{k=1}^{g}l_{k}&=m
\end{align}
then $L_{j}=L'_{j}=\mathfrak{l}'_{j}$ (by symmetry) and $m=\mathfrak{l}_{j}$. Thus the lattice $\mathcal{L}_{j}\subset \mathbb{Z}^{g}$ encodes all possible constant terms appearing in \eqref{1.3}, giving
\begin{equation}
    \dfrac{1}{2\pi \mathbf{i}}R_{\gamma_{j}}(\underline{a})\underset{\mathbb{Q}(1)}{\equiv}\log(a_{j})~-\sum_{\mathcal{L}_{j}}\dfrac{\Gamma(\mathfrak{l}_{j})}{\Gamma^2(1+\mathfrak{l}'_{j})\prod\limits_{k=1}^{g}\Gamma(1+l_{k})}{(-a_{j})^{-\mathfrak{l}_{j}}}\prod\limits_{\substack{k=1 \\ k\neq j}}^{g}a_{k}^{l_{k}}.\label{0.9}
\end{equation}
For the classical periods $\Pi_{j\ell}=\int_{\gamma_j}\varpi_{\ell}=\tfrac{1}{2\pi\ay}\delta_{a_{\ell}}R_{\gamma_j}$, it is clear from \eqref{0.9} that $\Pi_{j\ell}$ vanishes on the $a_j$-axis for $\ell\neq j$.  Focusing then on
\begin{equation}
\Pi_{jj}(\underline{a})=\int_{\gamma_j}\varpi_j=1+\sum_{\mathcal{L}_{j}}\dfrac{\Gamma(1+\mathfrak{l}_{j})}{\Gamma^2(1+\mathfrak{l}'_{j})\prod\limits_{k=1}^{g}\Gamma(1+l_{k})}(-a_{j})^{-\mathfrak{l}_{j}}\prod\limits_{\substack{k=1 \\ k\neq j}}^{g}a_{k}^{l_{k}},
\end{equation}
we set $\underline{a}_i=0$ for $i\neq j$ to obtain 
\begin{align}
\mathcal{S}:=1+\sum_{\tiny{\tfrac{g-j+1}{2j-1}l_j} \in \mathbb{Z}_{>0}}\dfrac{\Gamma(1+\tfrac{2g+1}{2j-1}l_{j})}{\Gamma^2(1+\tfrac{g-j+1}{2j-1}l_{j})\Gamma(1+l_{j})}(-{a}_{j})^{-\tfrac{2g+1}{2j-1}l_{j}}.
\end{align}
Recall that $\kappa_{j}:=\mathrm{gcd}(2j-1, 2g+1)$, and set
\begin{equation}
\begin{split}
n_{j}:&=\dfrac{2j-1}{\kappa_{j}}, \mspace{50mu} m_{j}:=\dfrac{2g+1}{\kappa_{j}}=\dfrac{(2g+1)n_j}{2j-1},\\
r_{j}:&=\dfrac{l_{j}}{n_{j}},\mspace{50mu}\text{and}\mspace{50mu}s_j:=a_j^{-m_j}.
\end{split}
\end{equation}
Clearly $n_j,m_j,r_j\in \mathbb{Z}_{>0}$. Now we have a power series of the form
\begin{equation}
    \mathcal{S}=1+\sum_{r_{j}\in \mathbb{N}}{\dfrac{(-1)^{m_jr_j}\Gamma(1+m_{j}r_{j})}{\Gamma^2(1+\tfrac{m_{j}-n_{j}}{2}r_{j})\Gamma(1+n_{j}r_{j})}} s_j^{r_{j}} =: \sum_{r_j} b_{r_j}s_j^{r_j}.
\end{equation}
Let $\mathring{s}_j:=\mathring{a}_j^{-m_j}$. Applying Lemma \ref{lemma_01},
\begin{equation}\label{gouri}
    \dfrac{\Gamma(1+m_{j}r_{j})}{\Gamma^2(1+\frac{m_{j}-n_{j}}{2}r_{j})\Gamma(1+n_{j}r_{j})} \approx \dfrac{(-1)^{m_jr_j}2\sqrt{m_{j}}}{2\pi r_{j}(m_{j}-n_{j})\sqrt{n_{j}}}\mathring{s}_j^{r_{j}}
\end{equation}
from which we may conclude that
\begin{equation}
    \lim_{r_{j}\to \infty}b_{r_{j}}\cdot r_{j}\cdot \mathring{s}_j^{r_{j}}=\dfrac{2\sqrt{m_{j}}}{2\pi (m_{j}-n_{j})\sqrt{n_{j}}}.
\end{equation}
Observing that
\begin{align}
\mathrm{Res}_{\mathring{p}_{j}}\varpi_{j}=\dfrac{\sqrt{2g+1}}{2\pi(g-j+1)\sqrt{(2j-1)}}&=\dfrac{\sqrt{n_{j}}}{2\pi n_{j}(g-j+1)}\cdot \sqrt{\dfrac{(2g+1)n_{j}}{2j-1}}\nonumber\\
&=\dfrac{2\sqrt{m_{j}n_{j}}}{2\pi(m_{j}-n_{j})(2j-1)}.
\end{align}
we apply (\ref{0.7}) to obtain
\begin{equation}
\bm{\kappa}_{jj}=\dfrac{\underset{r_j\to\infty}{\lim}b_{r_{j}}\cdot r_j\cdot \mathring{s}_j^{r_j}}{\mathrm{Res}_{\mathring{p}_{j}}\varpi_{j}}=\dfrac{2j-1}{n_{j}}=\kappa_{j}.
\end{equation}
This concludes the proof of Theorem \ref{thm_02}.

\begin{rem}
Notice that $\kappa_1=\kappa_g=1$. We document $\underline{\kappa}:=(\kappa_1,\ldots,\kappa_n)$ for $g=1,\dots,10$ in Table \ref{tab_1}.  The lack of symmetry for $g\geq 4$ should not be surprising given the shape of the Newton polygon.
\begin{table}[h!]
    \centering
    \begin{tabular}{|c|c |}
\hline
      $g$   & $\underline{\kappa}$  \\
      \hline
     1    & 1\\
     \hline
     2    & (1,1)\\
     \hline
     3    & (1,1,1)\\
     \hline
     4    & (1,3,1,1)\\
         \hline
     5    & (1,1,1,1,1)\\
          \hline
     6    & (1,1,1,1,1,1)\\
          \hline
     7    & (1,3,5,1,3,1,1)\\
          \hline
     8   &  (1,1,1,1,1,1,1,1)\\
          \hline
     9   & (1,1,1,1,1,1,1,1,1)\\    
          \hline 
     10  & (1,3,1,7,3,1,1,3,1,1)\\
         \hline
    \end{tabular}
    \caption{Conifold multiples for small genera}
    \label{tab_1}
\end{table}
\end{rem}

\subsection{Normalization of the conifold fibers}\label{S4c}

For the family $\mathcal{C}_{m,n}$ determined by the $\{F^{\underline{a}}_{m,n}\}$, the \textit{maximal conifold} point $\underline{\hat{a}}\in(\CC^*)^{g}$ is defined to be the unique point (if it exists) on the boundary of the region of convergence of the $g$ power series \eqref{eA3} where ${\mathcal{C}}_{m,n}^{\hat{\ua}}$ (given by $F^{\hat{\ua}}_{m,n}=0$) acquires $g$ nodes (labeled by $\hat{p}_j:=(\hat{x}_j,\hat{y}_j)$). In this subsection we determine $\hat{\ua}$ in the $(g,g)$ cases (where $r=0$).

\begin{rem}
In this case it is not necessary to impose a convergence requirement to get uniqueness of a $g$-nodal rational curve in moduli.  This comes along for the ride as we shall see in Remark \ref{r4.15}.  However, one should add right away that it is only $\hat{\uz}$ which is unique (with or without this requirement), not $\hat{\ua}$.  In fact, $\bm{\mathcal{M}}$ is a $(2g+1)$-to-$1$ \'etale cover of $\bm{\mathcal{M}}_{\uz}$, the GKZ moduli space (cf.~Remark \ref{r4.2}).  Precisely one of the $2g+1$ preimages of $\hat{\uz}$ has real coordinates; it is \emph{this one} we shall call $\hat{\ua}$.  Given existence of $\hat{\ua}$, established in Prop.~\ref{prop_04} below, a result of Tyomkin \cite[Prop.~7]{Ty07} guarantees uniqueness of $\hat{\uz}$.

The idea is to begin with the moduli space of all curves on $\PP_{\Delta}$ in the linear system $|\mathcal{O}_{\Delta}(1)|$ avoiding the singularities. (That is, we consider essentially all Laurent polynomials on $\Delta=\mathrm{conv}\{(1,0),(0,1),(-g,-g)\}$, not just the tempered ones.) This has dimension $g+2$, and contains a variety $V$ parametrizing all irreducible nodal rational curves. By [loc.~cit.], $V$ is irreducible and isomorphic to an open subset of $(\CC^*)^2\times (\PP^1)^3$ modulo $\mathrm{PGL}_2(\CC)$ viewed as automorphisms of the mormalizing $\PP^1$, hence of dimension $2$. Quotienting out by toric automorphisms (i.e.~$(\CC^*)^2$) maps each curve to its $\uz$-coordinate. The action of $(\CC^*)^2$ on $V$ has no fixed points, so the image of $V$ in $\bm{\mathcal{M}}_{\uz}$ is zero-dimensional and irreducible, i.e.~a single point.
\end{rem}

Now the most straightforward way to find $\underline{\hat{a}}$ would be via the discriminant locus:  one should look for transverse intersections amongst its local analytic branches. This is a viable strategy in particular cases; however, it requires careful analysis even in genus $2$.

\begin{example}
The family $\mathcal{C}_{2,2}$ arising as the mirror of the resolution of $\mathbb{C}^3/\mathbb{Z}_5$ orbifold was extensively studied in \cite[\S4.1]{CGM}.  Its discriminant locus is described by the equation
\begin{equation}
 3125z_1^2z_2^3 + 500z_1z_2^2 + 16z_2^2 - 225z_1z_2 - 8z_2 + 27z_1 + 1=0,
\end{equation}
where 
\begin{equation}
z_1=\dfrac{a_2}{a_1^3},~~z_2=\dfrac{a_1}{a_2^2}.   
\end{equation}
Figure \ref{fig_1} illustrates the intersection that gives rise to the maximal conifold point $\underline{\hat{z}}=(-\tfrac{1}{25},\tfrac{1}{5})$, which lifts to $\hat{\ua}=(5,-5)$.
\begin{figure*}[h!]
    \centering
    \includegraphics[scale=0.3]{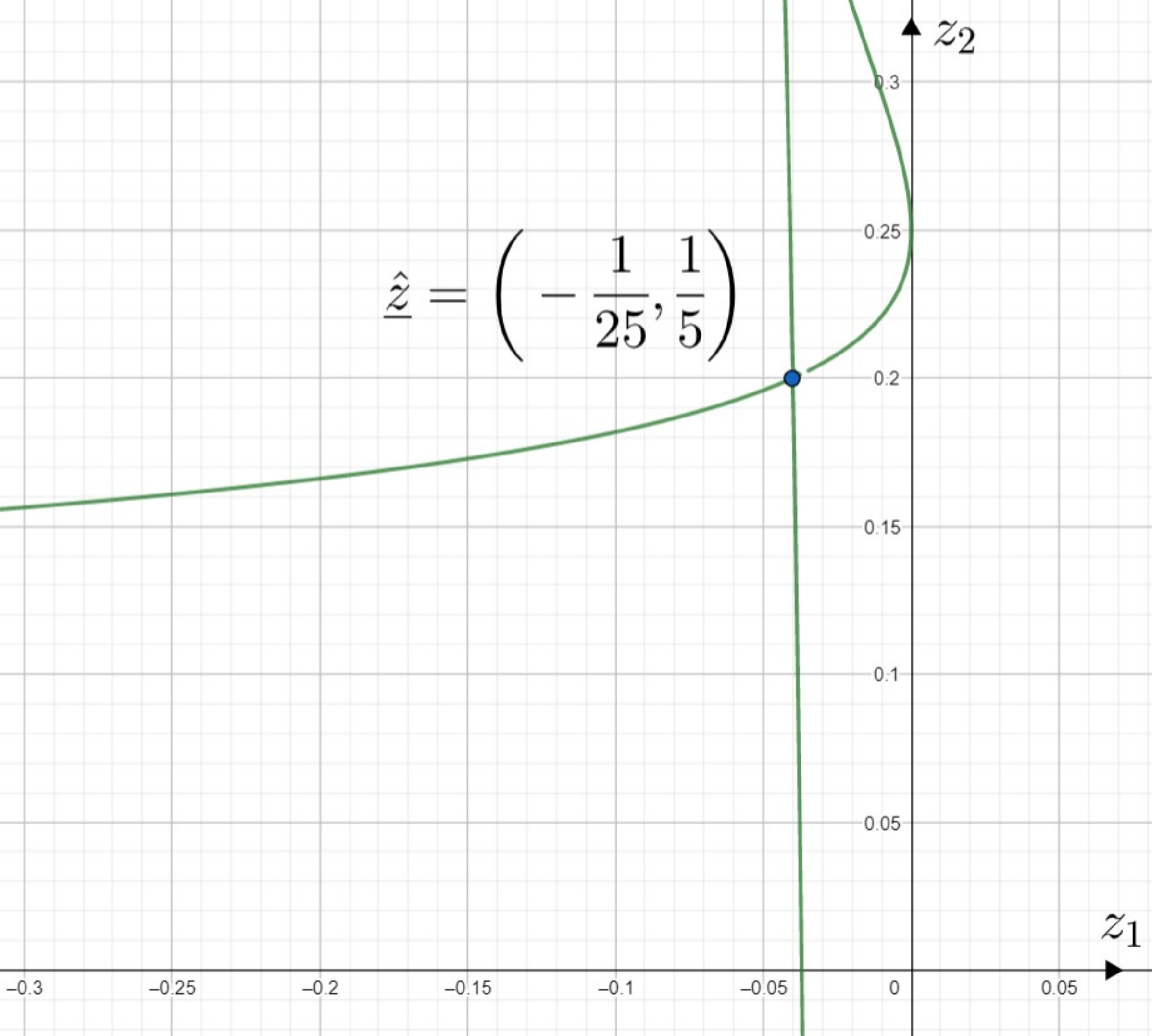}
    \caption{Discriminant locus of $\mathcal{C}_{2,2}$; axes are $z_i$'s.}
    \label{fig_1}
\end{figure*}
\end{example}

It is clear that for the family $\mathcal{C}_{g,g}$, the discriminant locus is described by a degree $2g+1$ polynomial in $g$ variables; so that approach quickly becomes untenable.  However, a close study of the $g=1$ and $g=2$ cases suggested a ``constructive'' approach to producing $g$-nodal fibers, which generalized well and leads to the following:
\begin{prop}\label{prop_04}
Let $\mathcal{T}_m$ denote the $m^{\text{th}}$ \textup{Chebyshev polynomial} of the first kind; this is a degree-$m$ polynomial characterized by $\mathcal{T}_{m}(\cos \theta)=\cos m\theta$. Then we have
\begin{equation}\label{ep04}
    F^{\underline{\hat{a}}}_{g,g}(x,x)=2x(\T_{2g+1}(\tfrac{1}{2x})+1).
\end{equation}
It follows that
\begin{align}
\hat{a}_{j}&=(-1)^{g-j+1}\dfrac{2g+1}{2j-1}\binom{g+j-1}{g-j+1}\;\;\;\text{and}\label{ep04a} \\
\hat{x}_{j}&=\hat{y}_{j}=\frac{(-1)^{g-j}}{2} \sec\left(\frac{g-j+1}{2g+1}\pi\right)  \label{4.56}
\end{align}
for $j=1,\ldots,g$.  In particular, $\underline{\hat{a}}\in\mathbb{Z}^g$.
\end{prop}

\begin{proof}
That $\hat{x}_j\in \mathbf{Z}(\text{RHS}\eqref{ep04})$ is immediate from the defining property of $\T_{2g+1}$, and the $\hat{x}_j$ are distinct and different from $-\tfrac{1}{2}$.  Moreover, writing $\mathcal{U}_m$ for the $m^{\text{th}}$ Chebyshev polynomial of the second kind, the relation $(\T_{2g+1}(w)-1)(\T_{2g+1}(w)+1)=(w^2-1)(U_{2g}(w))^2$ guarantees that all roots other than $-\tfrac{1}{2}$ of $(\T_{2g+1}(\tfrac{1}{2x})+1)$ have even multiplicity.  So they all have multiplicity $2$ and are precisely the $\{\hat{x}_j\}$.

The polynomial $\hat{F}(x,y):=x+y+\sum_{j=1}^g \hat{a}_j x^{1-j}y^{1-j}+x^{-g}y^{-g}$, with $\hat{a}_j$ as in \eqref{ep04a}, satisfies $\hat{F}(x,x)=\text{RHS}\eqref{ep04}$ by standard results on coefficients of $\T_m$.  Clearly $\hat{F}(\hat{p}_j)=0$, and the $\{\hat{p}_j\}$ are in fact singularities of $\mathbf{Z}(\hat{F})$ since $\tfrac{\partial \hat{F}}{\partial x}(x,x)=\tfrac{1}{2}\tfrac{d}{dx}(\hat{F}(x,x))$ and they are double roots of $\hat{F}(x,x)$.  Therefore, by Proposition \ref{Proposition_0.2}, they are all nodes.  Since one can also check that \eqref{0.9} converges at $\hat{p}_j$, $\mathbf{Z}(\hat{F})$ is the maximal conifold curve.
\end{proof}

\begin{rem}
Of course, Proposition \ref{prop_04} recovers the known maximal conifold points for the families $\mathcal{C}_{1,1}, \mathcal{C}_{2,2}$ ($\hat{a}_1=-3$ for $g=1$ and $\hat{a}_1=5,\hat{a}_2=-5$ for $g=2$).  Table \ref{tab_2} gathers $\mathcal{T}_{2g+1}$ and $\underline{\hat{a}}$ for a few low genus cases.
\end{rem}
\begin{table}[h!]
    \centering
    \begin{tabular}{|c|c |c|}
\hline
      $g$ &$\mathcal{T}_{2g+1}(x)$  & $\underline{\hat{a}}$  \\
      \hline
     1    & $4x^3-3x$ & -3\\
     \hline
     2    & $16x^5-20x^3+5x$ &(5,-5)\\
     \hline
     3    & $64x^7-112x^5+56x^3-7x$ & (-7,14,-7)\\
     \hline
     4    & $256x^9-576x^7+432x^5-120x^3+9x$ &(9,-30,27,-9)\\
         \hline
     5    & \tiny$1024x^{11}-2916x^9+2816x^7-1232x^5+220x^3-11x$ &(-11, 55, -77, 44, -11)\\
          \hline
    \end{tabular}\normalsize
    \caption{Maximal conifold points for low genera.}
    \label{tab_2}
\end{table}
Being of geometric genus zero, the maximal conifold fiber $\hat{\mathcal{C}}_{g,g}$ admits uniformizations by $\mathbb{P}^1$.  In particular, we have the $g$ distinct parametrizations $z\mapsto (\hat{X}_{j}(z),\hat{Y}_{j}(z))$, with
\begin{align}\label{4.57}
 \hat{X}_{j}(z)&=\frac{\hat{x}_{j}\left(1-\tfrac{1}{z}\right)^{g+1}}{\left(1-\tfrac{{\zeta^{g-j+1}_{2g+1}}}{z}\right)\left(1-\tfrac{{\zeta^{2(g-j+1)}_{2g+1}}} {z}\right)^g}\;\;\;\text{and}\\ \label{4.58}
 \hat{Y}_{j}(z)&=\frac{\hat{y}_j\left(1-\tfrac{z}{\zeta^{2(g-j+1)}_{2g+1}}\right)^{g+1}}{\left(1-\tfrac{z}{\zeta^{g-j+1}_{2g+1}}\right)\left(1-z\right)^g},
\end{align}
having the property that $z=0,\infty$ are mapped to $\hat{p}_{j}$.  (We defer the proof to the end of this subsection.)  Hence the image of the path from $z=0$ to $z=\infty$ on $\PP^1$ is sent (by the $j^{\text{th}}$ map) to $\hat{\gamma}_{j}$.  As dictated by \cite[\S6.2]{DK}, we assign a formal divisor $\hat{\mathcal{N}}_j$ on $\PP^1\setminus\{0,\infty\}$ to each uniformization:  for $X(z)=c_1\prod_j(1-\tfrac{\alpha_{j}}{z})^{d_{j}}$ and $Y(z)=c_2\prod_k (1-\tfrac{z}{\beta_{k}})^{e_k}$, this divisor is $\mathcal{N}:=\sum_{j,k}d_j e_k [\tfrac{\alpha_j}{\beta_k}]$.  According to [loc. cit.], the imaginary part of $\int_0^{\infty} R\{X(z),Y(z)\}$ is then given by $D_2(\mathcal{N}):=\sum_{j,k}d_je_k D_2(\tfrac{\alpha_j}{\beta_k})$.  

In our present situation,
\begin{align}
\hat{\mathcal{N}}_{j}&=g^2[{\zeta^{2(g-j+1)}_{2g+1}}]+2g[\zeta_{2g+1}^{g-j+1}]-(2g^2+2g-1)[1] \nonumber\\&\mspace{50mu}-
2(g+1)[\zeta_{2g+1}^{-(g-j+1)}]+(g+1)^2[\zeta^{-2(g-j+1)}_{2g+1}]\nonumber\\
&=2(2g+1)[\zeta^{g-j+1}_{2g+1}]-(2g+1)[\zeta_{2g+1}^{2(g-j+1)}]\\
&=2(2g+1)[1+\zeta^{g-j+1}_{2g+1}],\nonumber
\end{align}
where we are working modulo the scissors congruence relations
\begin{align}
[\xi]+[\tfrac{1}{\xi}]=0,~ [\xi]+[\overline{\xi}]=0, ~[\xi]+[1-\xi]&=0 ~\text{and}\\
[\xi_1]+[\xi_2]+[\tfrac{1-\xi_1}{1-\xi_1\xi_2}]+[\tfrac{1-\xi_2}{1-\xi_1\xi_2}]+[1-\xi_1\xi_2]&=0
\end{align} 
of the Bloch group $\mathcal{B}_2(\mathbb{C})$.  Consequently we have the identity
\begin{align}\label{0.67}
 D_2(\hat{\mathcal{N}}_{j})&=2(2g+1)D_2(1+\zeta^{g-j+1}_{2g+1}),
\end{align}
of which two particular cases are of note:  we claim that 
\begin{align}
     D_2(\hat{\mathcal{N}}_1)&=-2\pi \mathcal{D}_{g,g}\;\;\;\text{and}\label{0.68}\\
    D_2(\hat{\mathcal{N}}_g)&=-2\pi \mathcal{D}_{2g-1,1}.\label{0.69}
\end{align}
(See \S\ref{S2d} for notation.)  In fact, we can say something even more general. Given $m\in\mathbb{Z}_{>0}$, we have 
\begin{align}
-\mathfrak{z}^{m+1}\mathfrak{w}_{m,1}&=-\mathfrak{z}^{m+1}\dfrac{{\mathfrak{z}_{m,1}^{m}-{\mathfrak{z}_{m,1}^{-m}}}}{{\mathfrak{z}_{m,1}-{\mathfrak{z}_{m,1}^{-1}}}}=-\zeta_{2(m+2)}^{m+1}\sum_{k=0}^{m-1}\zeta_{2(m+2)}^k\zeta_{2(m+2)}^{-(m-1-k)}\nonumber\\
&=-\dfrac{\zeta_{2(m+2)}^{m+1}}{\zeta_{2(m+2)}^{m-1}}\sum_{k=0}^{m-1}\zeta_{2(m+2)}^{2k}=-\zeta_{2(m+2)}^{2}\sum_{k=0}^{m-1}\zeta_{m+2}^{k}\\
&=\zeta_{m+2}\Big(\zeta_{m+2}^{m}+\zeta_{m+2}^{m+1}\Big) =1+\zeta_{m+2}^{m+1}.\nonumber
\end{align}
Therefore, taking conjugates,
\begin{align}
2(m+2)D_2(1+\zeta_{m+2})&=-2(m+2)D_2(1+\zeta_{m+2}^{m+1})\nonumber\\
&=-2\pi\mathcal{D}_{m,1}
\end{align}
which implies \eqref{0.69} upon setting $m=2g-1$. Similarly one can see that 
\begin{equation}
\mathfrak{w}_{g,g}=\zeta^{1-g}_{2(2g+1)}\textstyle\sum_{k=0}^{g-1}{\zeta^k_{2g+1}}    
\end{equation}
and thus
\begin{align}
2(2g+1)D_2(1+\zeta^g_{2g+1})&=-2(2g+1)D_2\left(-\textstyle\sum_{k=1}^{g}{\zeta^k_{2g+1}}\right)\nonumber\\
&=-2(2g+1)D_2\left(-{\zeta^2_{2(2g+1)}}\textstyle\sum_{k=0}^{g-1}{\zeta^k_{2g+1}}\right)\\
&=-2\pi \mathcal{D}_{g,g},\nonumber
\end{align}
as was to be shown.

We are now ready to prove Theorem \ref{thm_01}.  By the previously mentioned result of \cite[\S6.2]{DK}, we know that $\Im(R_{\hat{\gamma}_j})=D_2(\hat{\mathcal{N}}_j)$ or
\begin{equation}
    \Re(\tfrac{1}{2\pi\ay}R_{\hat{\gamma}_j})=\tfrac{1}{2\pi}D_2(\hat{\mathcal{N}}_j).
\end{equation}
Next, Proposition \ref{thm_02} tells us that $R_{\gamma_j}(\hat{\ua})=\kappa_j R_{\hat{\gamma}_j}$, while \eqref{ep04a} and \eqref{0.9} ensure that (mod $\QQ(1)$) $\tfrac{1}{2\pi\ay}R_{\gamma_j}(\hat{\ua})$ hence $\tfrac{1}{2\pi\ay}R_{\hat{\gamma}_j}$ is real.  Combining this with (\ref{0.67}) gives
\begin{align}
\dfrac{1}{2\pi \mathbf{i}}R_{\gamma_{j}}(\underline{\hat{a}})=\dfrac{1}{2\pi \mathbf{i}}\kappa_j R_{\hat{\gamma}_{j}}\underset{\mathbb{Q}(1)}{\equiv}\dfrac{(2g+1)\kappa_j}{\pi}D_2(1+\zeta^{g-j+1}_{2g+1}),\label{0.15}
\end{align}
whence (\ref{0.3}) [resp. (\ref{0.4})] follows from (\ref{0.68}) [resp. (\ref{0.69})] by setting $j=1$ [resp. $j=g$] in  (\ref{0.15}).

To tie up the remaining loose end, we conclude with the

\begin{proof}[Proof of the parametrizations \eqref{4.57}-\eqref{4.58}]
Consider the map $$\eta_j\colon \PP^1\to \PP_{\Delta}$$ given by \eqref{4.57}-\eqref{4.58} and $\eta_j(z):=(\hat{X}_j(z),\hat{Y}_j(z))$.  Obviously $\eta_j(0)=(\hat{x}_j,\hat{y}_j)=\eta_j(\infty)$.  We must show that $\eta_j$ is of degree $1$ onto its image, and that this image is precisely $\mathcal{C}_{g,g}^{\hat{\ua}}$.

The first part is easy.  Here (only) we take $\PP_{\Delta}$ to be the singular toric variety given by the normal fan of $\Delta$ (and not a refinement).   Write $D_1,D_2,D_3$ for the boundary divisors, ordered so that the divisors of the torus coordinates read 
\begin{equation*}
(x)=(g+1)D_1-gD_2-D_3\;\;\;\text{and}\;\;\;(y)=(g+1)D_2-gD_1-D_3.
\end{equation*}
Now on $\PP^1$, write $\xi_j:=\zeta_{2g+1}^{g-j+1}$, and also $p_1,p_2,p_3$ for $1,\xi_j^2,\xi_j$ respectively.  Clearly we have $(\hat{X}_j(z))=(g+1)[p_1]-g[p_2]-[p_3]$ and $(\hat{Y}_j(z))=(g+1)[p_2]-g[p_1]-[p_3]$.  This shows that $\eta_j^* D_i = [p_i]$ for $i=1,2,3$, so the map has degree $1$ and the image meets all three boundary components transversely.

The next step is to check that it meets each boundary component where the edge coordinate is $-1$, which is where $\mathcal{C}_{g,g}^{\ua}$ hits them for any $\ua$.  That is, we must show that the limits 
\begin{equation*}
\lim_{z\to p_1} \hat{X}_j(z)^{g}\hat{Y}_j(z)^{g+1},\;\;\lim_{z\to p_2} \hat{X}_j(z)^{g+1}\hat{Y}_j(z)^g,\;\;\text{and}\;\;\lim_{z\to p_3} \tfrac{\hat{X}_j(z)}{\hat{Y}_j(z)}
\end{equation*}
are all $-1$.  For the third, since $\hat{x}_j=\hat{y}_j$ we get $\tfrac{\hat{X}_j(z)}{\hat{Y}_j(z)}=(\tfrac{z-1}{z-\xi_j^2})^{2g+1}$ which obviously gives $-1$ after substituting $z=\xi_j$.  For the first, we have $\hat{X}_j(z)^g\hat{Y}_j(z)^{g+1}=\hat{x}_j^{2g+1}(\tfrac{z-\xi_j^2}{z-\xi_j})^{2g+1}$; substituting $z=1$ yields $(\hat{x}_j (1+\xi_j))^{2g+1}$.  Writing $\xi_j^{1/2}:=\zeta_{4g+2}^{g-j+1}$, \eqref{4.56} gives $\hat{x}_j=\tfrac{(-1)^{g-j}}{\xi_j^{1/2}+\bar{\xi}_j^{1/2}}$ hence $\hat{x}_j(1+\xi_j)=(-1)^{g-j}\xi_j^{1/2}$, which has $(2g+1)^{\text{st}}$ power $(-1)^{g-j}(-1)^{g-j+1}=-1$.  The second limit is very similar to the first.  

Now suppose $\eta_j(\PP^1)\neq\mathcal{C}_{g,g}^{\hat{a}}$, and consider the divisor $(F^{\hat{a}}_{g,g})=\mathcal{C}_{g,g}^{\hat{a}}-gD_1-gD_2-D_3$.  The results of the last 3 paragraphs give $(\eta_j^*F_{g,g}^{\hat{\ua}})=\eta_j^*\mathcal{C}_{g,g}^{\hat{a}}-g[p_1]-g[p_2]-[p_3]\geq 2[0]+2[\infty]-(g-1)[p_1]-(g-1)[p_2].$  If $g=1$ or $2$, $(\eta_j^*F_{g,g}^{\hat{\ua}})$ already has positive degree, which is absurd; and the contradiction means that  $\eta_j(\PP^1)=\mathcal{C}_{g,g}^{\hat{a}}$.  If $g>2$, we have to work a bit harder to reach this contradiction.  It will suffice to verify that $\eta_j(\PP^1)$ also passes through the nodes $(\hat{x}_i,\hat{y}_i)$ for $i\neq j$.

To do this, write $\xi_i:=\zeta_{4g+2}^{g-i+1}$ and $\mu_i:=\zeta_{2g+1}^{g-i+1}=\xi_i^2$, and note that $\hat{x}_i=(-1)^{g-i}(\xi_i+\bar{\xi}_i)^{-1}=-\mu_i^{g+1}(1+{\mu}_i)^{-1}$.  We claim that $$\theta_{ij}:=\mu_j(\mu_j\bar{\mu}_i-1)(\bar{\mu}_i-\mu_j)^{-1}$$ (and $\xi_j^2/\theta_{ij}$, too, but we won't need that) are sent to $(\hat{x}_i, \hat{x}_i)$ by $\eta_j$.  For the $x$-coordinate, we have
\begin{equation*}
\begin{split}
\hat{X}_j(\theta_{ij})&=\hat{x}_j\frac{(\theta_{ij}-1)^{g+1}}{(\theta_{ij}-\mu_j)(\theta_{ij}-\mu_j^2)^g}\\
&=\frac{-\mu_j^{g+1}}{1+{\mu}_j}\frac{\mu_j^{g+1}(\frac{\mu_j\bar{\mu}_i-1}{\bar{\mu}_i-\mu_j}-\bar{\mu}_j)^{g+1}}{\mu_j^{g+1}(\frac{\mu_j\bar{\mu}_i-1}{\bar{\mu}_i-\mu_j}-1)(\frac{\mu_j\bar{\mu}_i-1}{\bar{\mu}_i-\mu_j}-\mu_j)^g}\\
&=\frac{-\mu_j^{g+1}}{1+{\mu}_j}\frac{\bar{\mu}_i^{g+1}\bar{\mu}_j^{g+1}(\mu_j^2-1)^{g+1}}{(\mu_j-1)(\bar{\mu}_i+1)(\mu_j^2-1)^g}=\frac{-\bar{\mu}_i^g}{1+\mu_i}=\hat{x}_i,
\end{split}
\end{equation*}
and the $y$-coordinate calculation is similar.
\end{proof}

\subsection{Explicit series identities}\label{S4d}

Spelling out (\ref{0.15}) in light of (\ref{0.9}) kills any torsion modulo $\mathbb{Q}(1)$ as both sides are real,\footnote{after changing $\log(\hat{a}_j)$ to $\log(|\hat{a}_j|)$} and yields the relationship 
\begin{align}
    \dfrac{(2g+1)\cdot\mathrm{gcd}(2j-1,2g+1)}{\pi}&D_2(1+\zeta^{g-j+1}_{2g+1})=\log(|\hat{a}_{j}|)\label{1.11}-\nonumber\\
    &\sum_{\mathcal{L}_{j}}\dfrac{\Gamma(\mathfrak{l}_{j})}{\Gamma^2(1+\mathfrak{l}'_{j})\prod\limits_{k=1}^{g}\Gamma(1+l_{k})}{(-\hat{a}_{j})^{-\mathfrak{l}_{j}}}\prod\limits_{\substack{k=1 \\ k\neq j}}^{g}\hat{a}_{k}^{l_{k}}
\end{align}
valid for $j=1,\dots,g$. The LHS can be shifted to a different avatar via the formula 
\begin{equation}
D_2(1+\zeta^{g-j}_{2g+1})=D_2\left(2\cos(\tfrac{\pi}{2g+1})e^{\pi i(g-j)/(2g+1)}\right).   
\end{equation}
Let us consider some applications of (\ref{1.11}). For the family $\mathcal{C}_{2,2}$ Table \ref{tab_1} and Table \ref{tab_2} say that $\underline{\kappa}=(1,1)$ and $\underline{\hat{a}}=(5,-5)$. Recalling that $\mathfrak{w}:=\tfrac{1+\sqrt{5}}{2}=2\cos(\pi/5)$ and plugging in $j=1$ in (\ref{1.11}) gives
\begin{align}
\dfrac{5}{\pi}D_2(\mathfrak{w} e^{2\pi \ay/5})&=\log 5 - \mspace{-15mu}\sum_{l_1,l_2\in \mathbb{Z}_{\geq 0}}{}^{\mspace{-25mu}\prime}\dfrac{\Gamma(5l_1+3l_2)(-5)^{-5l_1-3l_2}(-5)^{l_2}}{\Gamma^2(1+2l_1+l_2)\Gamma(1+l_1)\Gamma(1+l_2)} \nonumber \\
&=\log 5 - \mspace{-15mu} \sum_{m,r\in \mathbb{Z}_{\geq 0}}{}^{\mspace{-25mu}\prime}\dfrac{(-1)^{m}\Gamma(5m+3r)5^{-5m-2r}}{\Gamma^2(1+2m+r)\Gamma(1+m)\Gamma(1+r)}.\nonumber
\end{align}
On the other hand for $j=2$, 
\begin{align}
    \dfrac{5}{\pi}D_2(\mathfrak{w} e^{\pi \ay/5})=\log 5 - \mspace{-15mu}\sum_{l_1,l_2\in \mathbb{Z}_{\geq 0}}{}^{\mspace{-25mu}\prime}\dfrac{\Gamma(\frac{5l_2+l_1}{3})5^{-\tfrac{5l_2+l_1}{3}}5^{l_1}}{\Gamma^2(1+\tfrac{l_2-l_1}{3})\Gamma(1+l_1)\Gamma(1+l_2)}.
\end{align}
Defining $r:=l_1,m:=(l_2-l_1)/3$,
\small\begin{align}
\dfrac{5}{\pi}D_2(\mathfrak{w} e^{\pi i/5})=\log 5 - \mspace{-15mu}\sum_{m,r\in \mathbb{Z}_{\geq 0}}{}^{\mspace{-25mu}\prime}\dfrac{\Gamma(5m+2r)5^{-5m-r}}{\Gamma^2(1+m)\Gamma(1+r)\Gamma(1+3m+r)}.
\end{align}\normalsize
These identities, conjectured in \cite[A.10]{CGM}, match the identities \cite[(6.13)-(6.14)]{7K}.\footnote{The proof there was incomplete as it did not address $\bm{\kappa}$.}  Likewise, for $\mathcal{C}_{3,3}$ we have $\underline{\hat{a}}=(-7,14,-7)$ and $\underline{k}=(1,1,1)$, and thus\small
\begin{align}
\dfrac{7}{\pi}D_2(1+\zeta_7^3)&=
\log 7 -\mspace{-25mu} \sum_{m,r,p\in \mathbb{Z}_{\geq 0}}{}^{\mspace{-30mu}\prime}\dfrac{(-1)^{r}\Gamma(7m+5r+3p)7^{-7m-4r-2p}2^{p}}{\Gamma^2(1+3m+2r+p)\Gamma(1+m)\Gamma(1+r)\Gamma(1+p)}\\
\dfrac{7}{\pi}D_2(1+\zeta_7^2)&=\log 7 - \mspace{-25mu}\sum_{m,r,p\in \mathbb{Z}_{\geq 0}}{}^{\mspace{-30mu}\prime}\dfrac{(-1)^{r}\Gamma(7m+5r+p)7^{-4m-5r+2p}2^{-7m-5r-p}}{\Gamma^2(1+2m+r-p)\Gamma(1+3m)\Gamma(1+3r)\Gamma(1+3p)}\\
\dfrac{7}{\pi}D_2(1+\zeta_7)&=\log 7 - \mspace{-25mu}\sum_{m,r,p\in \mathbb{Z}_{\geq 0}}{}^{\mspace{-30mu}\prime}\dfrac{(-1)^{m}\Gamma(7m+3r+p)7^{-7m+2p}2^{3r}}{\Gamma^2(1+m-r-2p)\Gamma(1+3m)\Gamma(1+3r)\Gamma(1+3p)}.
\end{align}\normalsize
More generally, for the family $\mathcal{C}_{g,g}$, $\mathcal{L}_1$ becomes the lattice $\mathbb{Z}_{\geq 0}^g\setminus\{0,\ldots,0\}$ and we end up with a tidy expression, 
\begin{equation}\label{e4.78}
\begin{split}
&\mspace{70mu}\dfrac{(2g+1)}{\pi}D_2(1+\zeta^{g}_{2g+1})\mspace{30mu}=\mspace{30mu}\log(|\hat{a}_{1}|) -\\
&\sum_{\underset{1\leq k\leq g}{l_{k}\in \mathbb{Z}_{\geq 0}}}{}^{\mspace{-15mu}\prime}\left(-1\right)^{\sum\limits_{k=1}^{g}l_{k}}\tfrac{\Gamma\left((2g+1)l_1+\sum\limits_{k=2}^{g}(2k-1)l_{k}\right)}{\Gamma^2\left(1+gl_1+\sum\limits_{k=2}^{g}(k-1)l_{k}\right)\prod\limits_{k=1}^{g}\Gamma(1+l_{k})} \hat{a}_1^{-(2g+1)l_1-\sum\limits_{k=2}^{g}(2k-1)l_{k}}\prod_{k=2}^{g}\hat{a}_{k}^{l_{k}},
\end{split}
\end{equation}
where $\underset{l_k}{\sum}{}^{\prime}$ means that we omit the term corresponding to $\{0,\dots,0\}$.

\begin{rem}\label{r4.15}
We briefly address convergence of the power series part of RHS\eqref{e4.78}, to $\tilde{R}(\ua):=\tfrac{1}{2\pi\ay}R_{\gamma_1}(\ua)+\log(a_1)$ evaluated at $\ua=\hat{\ua}$. Replacing $\hat{a}_i$ with $a_i$, then substituting the GKZ variables $z_i$ (cf.~Remark \ref{r4.2}), it becomes a power series of the form 
$$\textstyle\sum_{\ul\geq \uo}'c_{\ul}z_1^{g\ell_1+\sum_{k=2}^g (k-1)\ell_k}z_2^{(g-1)\ell_1+\sum_{k=3}^g (k-2)\ell_k}\cdots z_g^{\ell_1}$$
which represents $\tilde{R}(\ua(\uz))$ for sufficiently small $\uz$.

Moreover, we claim that $\tilde{R}(\ua(uz))$ has no monodromy for $\uz=\uz(t):=(t^m,t,\ldots,t)$ if $m\gg 0$ and $|t|<1$. It is enough to check that there is no monodromy on $z_1=0$ (obvious, as the power series is identically zero there) or when $|z_1|<1$ and $z_i=\hat{z}_i$ ($i\geq 2$). For the latter, note that \eqref{ep04} becomes $2x\{z_1^{-{1}/{2}}\T_{2g+1}(\tfrac{1}{2x}z_1^{{1}/{(4g+2)}})+1\}$, whose discriminant is a power of $z_1-1$. (Roots of $\T_{2g+1}'=(2g+1)U_{2g}$ are $\cos(\tfrac{k\pi}{2g+1})$ for $k=1,\ldots,2g$, and $\T_{2g+1}(\cos(\tfrac{k\pi}{2g+1}))+z_1^{1/2}=(-1)^k+z_1^{1/2}$ is $0$ iff $z_1=1$.)

So $B(t):=\tilde{R}(\ua(\uz(t)))$ is represented by a power series $\sum_m B_m t^m$ on the unit disk, is bounded on $\{|t|<1+\epsilon\}\setminus[1,1+\epsilon)$ (as the $K_2$ symbol is nonsingular at $t=1$), and has monodromy about $t=1$ $(T_1-I)B\sim \text{cst.}\times (t-1)$ (since $(T_1-I)\gamma_1$ is a vanishing cycle with trivial regulator). We are now in the situation of \cite[Lemma 6.4]{Ke2} with $w=2$, so that $B_m\sim \text{cst.}\times m^{-2}$. The power series thus converges at $t=1$, and must evaluate to $B(1)$ by Tauber's theorem.
\end{rem}

\appendix

\section{Some regulator calculations}\label{appA}

Here we demonstrate the existence of integral 1-cycles $\{\gamma_j\}_{j=1}^g$ on $\C$ with regulator periods behaving as $R_{\gamma_j}\sim -2\pi\ay\log(a_j)$ for large $a_j$, as claimed in \S\ref{S2c}.  In the genus $1$ case, we also indicate how one can check the constant term in $R_{\beta}$ (cf. Lemma \ref{l3a1}) \emph{without} using mirror symmetry, and relate the constant term to the limit of a variation of MHS.  We refer the reader to \cite{DK} or \cite{KLi} for background on regulator currents.

We start by defining the 1-cycles in distinct regions of moduli.  We will need some notation.  Set $\TT:=\{\ux\in (\CC^*)^2\bigl\vert |x_1|=1=|x_2|\}$ (with the standard orientation as a 2-cycle) and let $\Gamma\subset \PP_{\Delta}$ be a 3-chain bounding on $\TT$ (but avoiding $\bar{\C}\setminus \C$).  Write $x^{\ee}:=\ux^{\um^{\ee}}$ for the toric coordinate along the boundary component $\DD_{\ee}\subset \PP_{\Delta}$ corresponding to an edge $\ee\subset \partial\Delta$, and $\{q_{\ee,\ell}\}$ for the roots of $P(- x_{\ee})$ (amongst the $\{q_k\}$), repeated with multiplicity; we have $P_{\ee}(x_{\ee})=\prod_{\ell}(1+\tfrac{x_{\ee}}{q_{\ee,\ell}})$, with $\prod_{\ell}q_{\ee,\ell}=1$.  Also, $\log_{\ee}(\xi)$ will mean $\log(\xi)$ for $\xi$ enclosed (counterclockwise on $\DD_{\ee}$) by $\Gamma\cap \DD_{\ee}$ and $0$ otherwise.

Now, fixing $j\in \{1,\ldots,g\}$, take $\ay a_j\in \mathfrak{H}$ and $|a_j|\gg \max_{i\neq j}|a_i|$; and note that then $F(\TT)\cap \RR_-=\emptyset$.  \emph{In this region}, define $\gamma_j := \Gamma \cap \C$, and use the current coboundary
\begin{equation}\label{eA1}
\textstyle \tfrac{1}{2\pi\ay}d[R\{F(\ux),\mi x_1,\mi x_2\}]=\sum_{\ee} R\{P_{\ee}(x_{\ee}),\mi x_{\ee}\}\cdot \delta_{\DD_{\ee}}-R\{\mi x_1,\mi x_2\}\cdot \delta_{\bar{\C}}
\end{equation}
together with the Tame symbols of $R\{P(x_{\ee}),-x_{\ee}\}$ (which are just the $\{q_{\ee,\ell}^{-1}\}$) and the Cauchy integral formula to compute
\begin{equation}\label{eA2}
\begin{split}
\textstyle R_{\gamma_j}&=\textstyle \int_{\gamma_j} R\{\mi x_1,\mi x_2\}= \textstyle \int_{\Gamma} R\{\mi x_1,\mi x_2\}\cdot\delta_{\bar{\C}}\\
&=\textstyle\tfrac{-1}{2\pi\ay}\int_{\TT} R\{F(\ux),\mi x_1,\mi x_2\}+\sum_{\ee}\int_{\Gamma\cap \DD_{\ee}}R\{P_{\ee}(x_{\ee}),\mi x_{\ee}\}\\
&=\textstyle \tfrac{-1}{2\pi\ay}\int_{\TT} \log(a_j(1+a_j^{-1}F_j(\ux)))\tfrac{dx_1}{x_1}\wedge\tfrac{dx_2}{x_2}+\sum_{\ee}\int_{\Gamma\cap\DD_{\ee}}R\{P_{\ee}(x_{\ee}),\mi x_{\ee}\} \\
&=\textstyle 2\pi\ay \left(-\log(a_j)+\sum_k \tfrac{(-1)^k}{k}[(F_j(\ux))^k]_{\uo}a_j^{-k} -\sum_{\ee,\ell}\log_{\ee}(q_{\ee,\ell})\right).
\end{split}
\end{equation}
In the tempered case, the $\{q_k\}$ are of course all $1$, and the last term vanishes.  We are then left with\footnote{Note that the version of this formula in \cite[Prop. 6.2]{KLi} is missing a $\pm \pi\ay$ (``2-torsion'') term:  the $\lambda_j$ parameter there is $-a_j$, so the leading term should have read $-\log(-\lambda_j)$ or $-\log(\lambda_j)+\pi\ay$.}
\begin{equation}\label{eA3}
\textstyle \tfrac{1}{2\pi\ay}R_{\gamma_j}(\ua)=-\log(a_j)+\sum_{k>0}\tfrac{(-1)^k}{k}[F_j^k]_{\uo}a_j^{-k},
\end{equation}
in which (by virtue of the GKZ theory) the sum can always be written as a power series in $z_1,\ldots,z_g$.\footnote{Essentially, this is just because in order to contribute to the constant term in $(F_j(\ux))^k$, a product of monomials must correspond to a sum of relations on points of $\Delta\cap \ZZ^2$, and the relations are how we defined the $\{z_i\}$.}  This gives a common region of convergence for the series for all $j$ (where the $z$-coordinates are small), to which the $\gamma_j$ admit well-defined continuation from the regions on which they were originally defined:  namely, they are the cycles with these regulator periods.  Moreover, they are clearly independent due to the asymptotic behaviors of these periods in the $\{a_j\}$.

In addition, \eqref{eA2}-\eqref{eA3} lead to formulas for periods of 1-forms.  Noting that $d[R\{F(\ux),-x_1,-x_2\}]=\tfrac{dF}{F}\wedge \tfrac{dx_1}{x_1}\wedge \tfrac{dx_2}{x_2}$, one introduces
\begin{equation}\label{eAi}
\varpi_{\ell}:=\tfrac{1}{2\pi\ay}\nabla_{\delta_{a_{\ell}}}\R=\tfrac{-1}{2\pi\ay}\mathrm{Res}_{\C}\left(\tfrac{\delta_{a_{\ell}}F}{F}\tfrac{dx_1}{x_1}\wedge\tfrac{dx_2}{x_2}\right)
\end{equation}\and computes
\begin{equation}\label{eAii}
-\Pi_{j\ell}:=-\int_{\gamma_j}\varpi_{\ell}=\tfrac{-1}{2\pi\ay}\delta_{a_{\ell}}R_{\gamma_j}=\bm{\delta}_{\ell j}+\textstyle\sum_{k>0}(-1)^k[F^k_j]_{\uo}a_j^{-k},
\end{equation}
where $\bm{\delta}_{\ell j}$ is the Kronecker delta.  This formula proves useful in \S\ref{S4b} where we change the sign of $\gamma_j$.

Turning to the $g=1$ case and the computation of $R_{\beta}$, it is more convenient to work with $u=-a\gg0$.  In this coordinate, \eqref{e3.1.3} becomes $t=\log(u)-\pi\ay+O(u^{-1})$.  Substituting this in Lemma \ref{l3a1}(a) and using $12-r^{\circ}=r$ yields
\begin{equation}\label{eA4}
R_{\beta}=\tfrac{r^{\circ}}{2}\log^2 u-\tfrac{r}{6}\pi^2+O(u^{-1}\log u).
\end{equation}
Consider the Laurent polynomial $\vf=x_1+x_1^{-1}+x_2+x_2^{-1}$, which corresponds to local $(\PP_{\Delta^{\circ}}=)\,\PP^1\times \PP^1$.  The discriminant (over the $x_1$-axis) of the equation $x_2+(x_1+x_1^{-1}-u)+x_2^{-1}=0$ has roots $\xi_1\sim \tfrac{1}{u+2}$, $\xi_2\sim \tfrac{1}{u-2}$, $\xi_3\sim u-2$, and $\xi_4\sim u+2$ (in increasing order).  Introduce $2x_{2,\pm}(x_1):=u-x_1-x_1^{-1}\pm\sqrt{(x_1+x_1^{-1}-u)^2-4}$ and $w(x_1):=\tfrac{4}{(u-x_1-x_1^{-1})^2}$.  For $x_1\in (\xi_2,\xi_3)$, $w$ lies in $(0,1)$, and we write $\log(\tfrac{4}{w}\cdot\tfrac{1-\sqrt{1-w}}{1+\sqrt{1-w}})=:\sum_{m\geq 1}\theta_m w^m=\tfrac{1}{2}w+\tfrac{3}{16}w^2+\cdots$.  Now we compute
\begin{equation}\label{eA5}
\begin{split}
R_{\beta}&=\textstyle-\int_{\beta}R\{-x_2,-x_1\}=\int_{\xi_2}^{\xi_3}\log(\tfrac{x_{2,+}}{x_{2,-}})\tfrac{dx_1}{x_1}=\int_{\xi_2}^{\xi_3}\log(\tfrac{1+\sqrt{1-w}}{1-\sqrt{1-w}})\tfrac{dx_1}{x_1}\\
&=\textstyle-\int_{\xi_2}^{\xi_3}\log(\tfrac{w}{4})\tfrac{dx_1}{x_1}-\sum_{m\geq 1}\theta_m\int_{\xi_2}^{\xi_3}w^m\tfrac{dx_1}{x_1}\\
&=\textstyle 2\log(u)\int_{\xi_2}^{\xi_3}\tfrac{dx_1}{x_1}+2\int_{\xi_2}^{\xi_3}\log(1-u^{-1}(x_1+x_1^{-1}))\tfrac{dx_1}{x_1}+O(u^{-1}\log u)\\
&=\textstyle 4\log^2 u - 2\sum_{k>0}\tfrac{u^{-k}}{k}\int_{\xi_2}^{\xi_3}(x_1+x_1^{-1})^k \tfrac{dx_1}{x_1}+O(u^{-1}\log u)\\
&=\textstyle 4\log^2 u-\tfrac{2\pi^2}{3}+O(u^{-1}\log u),
\end{split}
\end{equation}
at the end using the approximations $\int_{\xi_2}^{\xi_3}(x_1+x_1^{-1})^k\tfrac{dx_1}{x_1}\sim \tfrac{2\xi_3^k}{k}\sim \tfrac{2u^k}{k}$ to rewrite the sum as $-4\sum\tfrac{1}{k^2}=-\tfrac{2}{3}\pi^2$ up to $O(u^{-1}\log u)$.  The point is that since $r=4$, this agrees with the result \eqref{eA4} from integral local mirror symmetry.  A similar computation in \cite[\S6]{KLi} for $\vf=x_1+x_2+x_1^{-1}x_2^{-1}$ (mirror to local $\PP^2$) gives $R_{\beta}=\tfrac{9}{2}\log^2 u-\tfrac{\pi^2}{2}+O(u^{-1}\log u)$, where the $-\tfrac{\pi^2}{2}$ arises as $-2\mathrm{Li}_2(\tfrac{1}{2})-2\mathrm{Li}_2(1)-\log^2 2$.  Since $r=3$, this agrees once more with \eqref{eA4} (as it must).

The crucial constant term in $R_{\beta}$ has a nice interpretation via the LMHS at $a=\infty$ of the VMHS $\V$ attached to $\R\in H^1(E_a,\CC/\ZZ(2))$, the regulator class of $\{-x_1,-x_2\}\in H_{\mathrm{M}}^2(E_a,\ZZ(2))$.  (Note that the LMHS depends on a choice of a local coordinate, which we take to be $a^{-1}$ or equivalently $Q:=e^{-t}=a^{-1}(1+O(a^{-1}))$.)  We can present $\V$ and its dual as extensions
\begin{equation}\label{eA6}
H^1(E,\ZZ(2))\to \V_{\ZZ}\to \ZZ(0)\;\;\text{and}\;\;\ZZ(0)\to \V_{\ZZ}^{\vee}\to H_1(E,\ZZ(-2)).
\end{equation}
On the left, a unique class $\mathfrak{R}\in F^0\V_{\CC}$ maps to $1\in \ZZ(0)$; on the right, let $\tau\in \V^{\vee}_{\ZZ}$ be the image of $1$, and $\tilde{\gamma},\tilde{\beta}\in\V^{\vee}_{\ZZ}$ classes mapping to $\tfrac{1}{(2\pi\ay)^2}\gamma,\tfrac{1}{(2\pi\ay)^2}\beta$.   Writing $\ell(Q):=\tfrac{\log(Q)}{2\pi\ay}$, we have 
\begin{equation}\label{eA7}
\tilde{R}_{\beta}:=\langle \mathfrak{R},\tilde{\beta}\rangle=\tfrac{1}{(2\pi\ay)^2}R_{\beta}=\tfrac{r^{\circ}}{2}\ell(Q)^2-\tfrac{r^{\circ}}{2}\ell(Q)+\mathtt{T}+O(Q),
\end{equation}
where $\mathtt{T}=\tfrac{1}{2}+\tfrac{r^{\circ}}{12}$ (cf. Lemma \ref{l3a1}(a)), as well as $\tilde{R}_{\gamma}:=\langle \mathfrak{R},\tilde{\gamma}\rangle=\tfrac{1}{(2\pi\ay)^2}R_{\gamma}=\ell(Q)$ and $\langle \mathfrak{R},\tau\rangle=1$.

To obtain a period matrix for $\V$, we compare Hodge and Betti bases as follows.  Writing $\nabla$ for $\nabla_{\partial_{\ell(Q)}}$, the change-of-basis matrix from $\{\mathfrak{R},\nabla \mathfrak{R},\tfrac{1}{r^{\circ}}\nabla^2\mathfrak{R}\}$ to $\{\tau^{\vee},\tilde{\gamma}^{\vee},\tilde{\beta}^{\vee}\}$ is
\begin{equation}\label{eA8}
\bm{\Omega}:=\left(\begin{smallmatrix} 1&&\\ \tilde{R}_{\gamma}& 1& \\ \tilde{R}_{\beta} &\partial_{\ell(Q)}\tilde{R}_{\beta}& 1\end{smallmatrix}\right)=\left(\begin{smallmatrix} 1&&\\ \ell(Q) & 1&\\ \tfrac{r^{\circ}}{2}\ell(Q)^2-\tfrac{r^{\circ}}{2}\ell(Q)+\mathtt{T}\;&\; r^{\circ}\ell(Q)-\tfrac{r^{\circ}}{2}&\;1 \end{smallmatrix}\right)+O(Q).
\end{equation}
From \eqref{eA8} one easily deduces the monodromies $T\in \mathrm{Aut}(\V)$ and $T^{\vee}\in \mathrm{Aut}(\V^{\vee})$ about $Q=0$:
\begin{equation}\label{eA9}
[T^{\vee}]_{\{\tilde{\beta},\tilde{\gamma},\tau\}}=\left(\begin{smallmatrix} 1&& \\ r^{\circ}&1&\\ 0&1&1 \end{smallmatrix}\right)\;\;\;\implies\;\;\;\bm{T}:=[T]_{\{\tau^{\vee},\tilde{\gamma}^{\vee},\tilde{\beta}^{\vee}\}}=\left(\begin{smallmatrix} 1&& \\ 1&1& \\ 0&r^{\circ} 1\end{smallmatrix}\right).
\end{equation}
Consequently the \emph{limiting} period matrix is 
\begin{equation}\label{eA10}
\bm{\Omega}_{\lim,Q}:=\lim_{Q\to 0}e^{-\ell(Q)\log(\bm{T})}\bm{\Omega}=\begin{pmatrix}1&&\\ 0&1& \\ \mathtt{T}&-\tfrac{r^{\circ}}{2}&1 \end{pmatrix}.
\end{equation}
The LMHS with respect to $a^{-1}$, as mentioned above, gives the same result; but if we change local coordinate to $-Q$ or (equivalently) $u^{-1}$, we get
\begin{equation}\label{eA10}
\bm{\Omega}_{\lim,-Q}:=\lim_{Q\to 0}e^{-\ell(-Q)\log(\bm{T})}\bm{\Omega}=\begin{pmatrix}1&&\\ \tfrac{1}{2}&1& \\ \mathtt{B}^{\circ}&0&1 \end{pmatrix},
\end{equation}
where $\mathtt{B}^{\circ}=\tfrac{1}{2}-\tfrac{r^{\circ}}{24}=\mathtt{T}-\tfrac{r^{\circ}}{8}$.  So we see that both of the constants appearing in Lemma \ref{l3a2}(ii) have a standard asymptotic Hodge-theoretic meaning, in terms of (torsion) extension classes in the LMHS of $\V$ in the large complex structure limit.



\begin{thebibliography}{?????}

\bibitem[Be]{Be} A. Beilinson, \emph{Higher regulators and values of $L$-functions of curves}, Funct. Anal. Appl. 14 (1980), 116-118.

\bibitem[BKV]{BKV} S. Bloch, M. Kerr and P. Vanhove, \emph{A Feynman integral via higher normal functions}, Compos. Math. 151 (2015), 2329-2375.

\bibitem[CGM]{CGM} S. Codesido, A. Grassi, and M. Mari\~no, \emph{Spectral theory and mirror curves of higher genus}, Ann. Henri Poincar\'e 18 (2017), no. 2, 559-622.

\bibitem[7K]{7K} P. del Angel, C. Doran, J. Iyer, M. Kerr, J. Lewis, S. M\"uller-Stach, and D. Patel, \emph{Specialization of cycles and the $K$-theory elevator}, Commun. Number Theory Phys. 13 (2019), no. 2, 299-349.

\bibitem[Ga]{Ga} S. Galkin, \emph{The conifold point}, arXiv:1404.7388, 2014.

\bibitem[DK]{DK} C. Doran and M. Kerr, \emph{Algebraic $K$-theory of toric hypersurfaces}, Commun. Number Theory Phys. 5 (2011), no. 2, 397-600.

\bibitem[GHM]{GHM} A. Grassi, Y. Hatsuda, and M. Mari\~no, \emph{Topological strings from quantum mechanics}, Ann. Henri Poncar\'e 17 (2016), no. 11, 3177-3235.

\bibitem[GKMR]{GKMR} J. Gu, A. Klemm, M. Mari\~no, and J. Reuter, \emph{Exact solutions to quantum spectral curves by topological string theory}, J. High Energy Phys. 2015, no. 10, 025.

\bibitem[GS]{GS} S. Gukov and P. Sulkowski, \emph{A-polynomial, B-model, and quantization}, J. High Energy Phys. 2012, no. 2, 070.

\bibitem[Ir]{Ir} H. Iritani, \emph{Quantum cohomology and periods}, Ann. Inst. Fourier 61 (2011), no. 7, 2909-2958.

\bibitem[KM]{KM} R. Kashaev and M. Mari\~no, \emph{Operators from mirror curves and the quantum dilogarithm}, Commun. Math. Phys. 346 (2016), no. 3, 967-994.

\bibitem[KS]{KS} R. Kashaev and S. Sergeev, \emph{Spectral equations for the modular oscillator}, Rev. Math. Phys. 30 (2018), no. 7, 1840009, 28 pp.

\bibitem[Ke1]{Ke} M. Kerr, \emph{An elementary proof of Suslin reciprocity}, Canad. Math. Bull. 48 (2005), 175-210.

\bibitem[Ke2]{Ke2} ---------, \emph{Unipotent extensions and differential equations (after Bloch-Vlasenko)}, arXiv:2008.03618, preprint, 2020.

\bibitem[KLe]{KLe} M. Kerr and J. Lewis, \emph{The Abel-Jacobi map for higher Chow groups, II}, Invent. Math. 170 (2007), 355-420.

\bibitem[KLM]{KLM} M. Kerr, J. Lewis, and S. M\"uller-Stach, \emph{The Abel-Jacobi map for higher Chow groups}, Compos. Math. 142 (2006), 374-396.

\bibitem[KLi]{KLi} M. Kerr and M. Li, \emph{Two applications of the integral regulator}, Pacific J. Math. 306 (2020), no. 2, 539-556.

\bibitem[Ko]{Ko} Y. Konishi, \emph{Integrality of Gopakumar-Vafa invariants of toric Calabi-Yau threefolds}, Publ. Res. Inst. Math. Sci. 42 (2006), no. 2, 605-648.

\bibitem[LST]{LST} A. Laptev, L. Schimmer, and L. Takhtajan, \emph{Weyl type asymptotics and bounds for the eigenvalues of functional-difference operators for mirror curves}, Geom. Funct. Anal. 26 (2016), 288-305.

\bibitem[Li]{Li} M. Li, \emph{Integral regulators on higher Chow complexes}, SIGMA 14 (2018), 118, 12 pages.

\bibitem[Ma]{Ma} M. Mari\~no, \emph{Spectral theory and mirror symmetry}, in ``String Math 2016'', 259-294, Proc. Symp. Pure Math. 98, AMS, Providence, RI, 2018.

\bibitem[MZ]{MZ} M. Mari\~no and S. Zakany, \emph{Matrix models from operators and topological strings}, Ann. Henri Poincar\'e 17 (2016), no. 5, 1075-1108.

\bibitem[SWH]{SWH} K. Sun, X. Wang, and M.-x. Huang, \emph{Exact quantization conditions, toric Calabi-Yau and non-perturbative topological string}, J. High Energy Phys. 2017, no. 1, 061.

\bibitem[Ty]{Ty07} I. Tyomkin, \emph{On Severi varieties on Hirzebruch surfaces}, IMRN, 2007, no.~23, Art.~ID rnm09.

\bibitem[Za]{Za} S. Zakany, \emph{Quantized mirror curves and resummed WKB}, J. High Energy Phys. 2019, no. 5, 114.

\end{thebibliography}
\end{document}